\documentclass[11pt,english,sumlimits,reqno,oneside]{amsart}

\usepackage[T1]{fontenc}
\usepackage[utf8]{inputenc}
\usepackage{lmodern}

\usepackage{babel}

\usepackage[a4paper,margin=2.5cm]{geometry}

\usepackage{xifthen}%provides ifthenelse command
\usepackage{ifthenx}%complements xifthen. MUST be loaded after xifthen

\usepackage{xargs}

\usepackage{multirow}%provides multirow command for tables
\usepackage{hhline}%provides multirow command for tables
\usepackage{longtable}%for long tables

\usepackage{datetime2}% provides \DTMnow

\usepackage{graphicx}%
\usepackage{caption}%

\usepackage[table,svgnames,x11names]{xcolor}

\usepackage[disable,colorinlistoftodos,prependcaption,textsize=tiny,textwidth=3.5cm]{todonotes}
\presetkeys{todonotes}{fancyline}{}
\newcommandx{\change}[2][1=]{\todo[linecolor=blue,backgroundcolor=blue!25,bordercolor=blue,#1]{#2}}
\newcommandx{\changein}[2][1=]{\change[inline, caption={change}, #1]{%
    \begin{minipage}{\textwidth-20pt}#2\end{minipage}}}
\newcommandx{\todoin}[2][1=]{\todo[inline, caption={todo}, #1]{%
    \begin{minipage}{\textwidth-20pt}#2\end{minipage}}}

\newcommandx{\remove}[2][1=]{\todo[linecolor=Plum,backgroundcolor=Plum!25,bordercolor=Plum,#1]{#2}}
\newcommandx{\removein}[2][1=]{\remove[inline, caption={todo}, #1]{%
    \begin{minipage}{\textwidth-20pt}#2\end{minipage}}}

\usepackage{amssymb}
\usepackage{amsmath}
\usepackage{amsthm}
\usepackage{mathtools}
\usepackage{exscale}%to scale mathematical symbols
\usepackage{relsize}%to relatively adjust size (also works with math symbols)
\usepackage{bbm}%blackboard symbols
\usepackage{bm}
\usepackage{amsbsy}
\usepackage{mathdots}%provides \iddots

\usepackage{mathrsfs}%provides \mathscr --> eulerscript

\usepackage{stmaryrd}%provides double brackets \llbracket

\usepackage{mdframed}

\usepackage{esvect}% provides arrow symbol \vv{}

\usepackage[all]{xy}
\usepackage{fp}

\usepackage[section]{placeins}
\usepackage{float}

\usepackage{enumerate}
\usepackage{enumitem}
\usepackage{varioref}
\usepackage{hyperref}
\hypersetup{
  colorlinks,
  allcolors=DarkBlue
}
\usepackage[nameinlink]{cleveref}%must be loaded after all other packages

\newcounter{proof}
\newenvironment{myproof}%
{\stepcounter{proof}\begin{proof}}%
{\end{proof}}%
\newcounter{proofstep}[proof]
{\refstepcounter{proofstep}\smallskip\par\noindent%
  \ifthenelse{\isempty{#1}}%if
    {\textsf{Step \theproofstep. }}%then
    {\textsf{#1.}}%else
  \noindent}%
{\par}%
\crefname{proofstep}{step}{steps}%

\newcounter{proofcase}[proof]
\newenvironment{proofcase}[1][]%
{\refstepcounter{proofcase}\smallskip\par\noindent%
  \ifthenelse{\isempty{#1}}%if
    {\textsf{Case \theproofcase. }}%then
    {\textsf{#1.}}%else
  \noindent}%
{\par}%
\crefname{proofcase}{case}{cases}%

\theoremstyle{plain}

\newcounter{maintheorem}

\newtheorem{thm}{Theorem}[section]
\newtheorem*{thm*}{Theorem}

\newtheorem{pro}[thm]{Proposition}
\newtheorem{cor}[thm]{Corollary}
\newtheorem{lem}[thm]{Lemma}
\newtheorem{question}[thm]{Question}
\theoremstyle{definition}
\newtheorem{dfn}[thm]{Definition}
\newtheorem{ntn}[thm]{Notation}
\newtheorem{rem}[thm]{Remark}

\newtheorem{asm}[thm]{Assumption}
\numberwithin{equation}{section}

\newcommandx{\textref}[2][1=]{\hyperref[#2]{#1\ref*{#2}}}
\newcommandx{\textrefp}[2][1=]{(\hyperref[#2]{#1\ref*{#2}})}

\DeclareMathOperator{\Id}{Id}
\DeclareMathOperator{\supp}{supp}

\DeclareMathOperator*{\esssup}{ess\,sup}

\DeclareMathOperator{\cond}{\mathbb{E}}
\DeclareMathOperator{\var}{\mathbb{V}}

\DeclareMathOperator{\prob}{\mathbb{P}}

\begin{document}

\title[Multipliers on Bi-parameter Haar system Hardy spaces]{Multipliers on Bi-parameter Haar system Hardy spaces}%

\author[Lechner]{R. Lechner}%
\address{R. Lechner, Institute of Analysis, Johannes Kepler University Linz, Altenberger Strasse 69, A-4040 Linz, Austria}%
\email{richard.lechner@jku.at}%

\author[Motakis]{P. Motakis}%
\address{P. Motakis, Department of Mathematics and Statistics, York University, 4700 Keele Street, Toronto, Ontario, M3J 1P3, Canada}%
\email{pmotakis@yorku.ca}%

\author[M\"uller ]{P.F.X. M\"uller}%
\address{P.F.X. M\"uller, Institute of Analysis, Johannes Kepler University Linz, Altenberger Strasse 69, A-4040 Linz, Austria}%
\email{paul.mueller@jku.at}%

\author[Schlumprecht]{Th.~Schlumprecht}%
\address{Th.~Schlumprecht, Department of Mathematics, Texas A\&M University, College Station, TX 77843-3368, USA, and Faculty of Electrical
  Engineering, Czech Technical University in Prague, Technika 2, 16627, Praha 6, Czech Republic}%
\email{t-schlumprecht@tamu.edu}%

\date{\today}%



\thanks{The first-named author was supported by the Austrian Science Foundation (FWF) grant P32728. The second author was supported by NSERC grant
  RGPIN-2021-03639.  The third author was supported by the Austrian Science Foundation (FWF) under the grants P344114 and I5231.  The fourth author
  was supported by the National Science Foundation under grant DMS-2054443}%

\begin{abstract}
  Let $(h_I)$ denote the standard Haar system on $[0,1]$, indexed by $I\in \mathcal{D}$, the set of dyadic intervals and $h_I\otimes h_J$ denote the tensor product
  $(s,t)\mapsto h_I(s) h_J(t)$, $I,J\in \mathcal{D}$.  We consider a class of two-parameter function spaces which are completions of the linear span
  $\mathcal{V}(\delta^2)$ of $h_I\otimes h_J$, $I,J\in \mathcal{D}$.  This class contains all the spaces of the form $X(Y)$, where $X$ and $Y$ are either the Lebesgue spaces
  $L^p[0,1]$ or the Hardy spaces $H^p[0,1]$, $1\le p < \infty$.
  
  We say that $D\colon X(Y)\to X(Y)$ is a Haar multiplier if $D(h_I\otimes h_J) = d_{I,J} h_I\otimes h_J$, where $d_{I,J}\in \mathbb{R}$, and ask which more elementary
  operators factor through $D$.  A decisive role plays the \emph{Capon projection} $\mathcal{C}\colon \mathcal{V}(\delta^2)\to \mathcal{V}(\delta^2)$ given by
  $\mathcal{C} h_I\otimes h_J = h_I\otimes h_J$ if $|I|\leq |J|$, and $\mathcal{C} h_I\otimes h_J = 0$ if $|I| > |J|$, as our main result highlights:

  Given any bounded Haar multiplier $D\colon X(Y)\to X(Y)$, there exist $\lambda,\mu\in \mathbb{R}$ such that
  \begin{equation*}
    \text{$\lambda \mathcal{C} + \mu (\Id-\mathcal{C})$ approximately $1$-projectionally factors through $D$,}
  \end{equation*}
  i.e., for all $\eta > 0$, there exist bounded operators $A,B$ so that $AB$ is the identity operator $\Id$, $\|A\|\cdot\|B\| = 1$ and
  $\|\lambda \mathcal{C} + \mu (\Id-\mathcal{C}) - ADB\| < \eta$.  Additionally, if $\mathcal{C}$ is unbounded on $X(Y)$, then $\lambda = \mu$ and then $\Id$ either factors through $D$ or $\Id-D$.
\end{abstract}

\maketitle%
\tableofcontents%

% redefining a command that amsart redefined
\makeatletter%
\providecommand\@dotsep{5}%
\def\listtodoname{List of Todos}%
\def\listoftodos{\@starttoc{tdo}\listtodoname}%
\makeatother%

%\listoftodos

\newcounter{mycounter}

\section{Introduction}
\label{sec:intro}
Let us consider a bounded linear operator $T\colon E \to E$ on a Banach space $E$, with Schauder basis $(x_n)$ and biorthogonal functionals $(x_n^*)$.
All information carried by $T$ is then encoded in the matrix,
\begin{equation*}
  ( \langle x_m^*, Tx_n \rangle )_{ m,n =1}^\infty.
\end{equation*}
Indeed for $f\in E$ with basis expansion $f = \sum _{n =1}^\infty \langle x_n^*, f \rangle x_n$ we have
\begin{equation*}
  Tf
  = \sum_{n =1}^\infty \sum_{m=1}^\infty \langle x_m^* , Tx_n \rangle \langle x_n^* , f \rangle x_m.
\end{equation*}
In the special case, where $\langle x_m^* , Tx_n \rangle = 0$ for $m \neq n$, the operator $T$ acts directly as a multiplier on the basis $(x_n)$, and
\begin{equation*}
  Tf
  = \sum_{n =1}^\infty \langle x_n^*, Tx_n \rangle \langle x_n^*, f \rangle x_n .
\end{equation*}
Accordingly, problems arising in operator theory are investigated alongside suitably chosen Schauder bases and biorthogonal systems in Banach spaces.
Thus the Haar system $(h_I)$ and its tensor product arises with the factorization of operators on Lebesgue spaces, Hardy spaces, or mixed norm spaces
like $L^p(L^q)$.

Expressing a given operator problem in terms of Haar functions and exploiting their relative simplicity often leads straight to its combinatorial
core.  In this regard the books Lindenstrauss-Tzafriri~\cite{lindenstrauss:tzafriri:1979:partII}, Bourgain~\cite{bourgain:book:1981}, Hyt\"onen, van
Neerven, Veraar and Weis~\cite{MR3617205,MR3752640} and~\cite{mueller:2005, mueller:2022} point out the following notable examples:
 
Johnson's~\cite{johnson:1972} factorization of operators on $L^p $ through $\ell ^p$, the work of Enflo-Maurey~\cite{maurey:1975:2} and
Enflo-Starbird~\cite{enflo:starbird:1979} on complemented subspaces of $L^p$, Capon's~\cite{capon:1982:2} proof that the space $L^p(L^q)$ is primary,
the construction by Bourgain, Rosenthal and Schechtman~\cite{bourgain:rosenthal:schechtman:1981} of uncountably many complemented subspaces of $L^p$,
the subspace stability theorem for $L^p$, and for rearrangement invariant function spaces by Johnson, Maurey, Schechtman and
Tzafriri~\cite{johnson:maurey:schechtman:tzafriri:1979}, Kalton's $L^1$-embeddings \cite{kalton:1978,kalton:1979}, the construction of resolving
operators in Jones's~\cite{jones:1985} work on the uniform approximation property, Pisier's~\cite{pisier:1975:mart, pisier:2016:martingales}
characterization of martingale cotype and operator valued martingale transforms in the work of Girardi and Weis~\cite{girardi:weis:2005}.

Our recent work~\cite{lechner:motakis:mueller:schlumprecht:2021}, on $L^1(L^p)$ is typical of this line of research.  Establishing that $L^1(L^p)$ is
primary amounts to factoring the identity on $L^1(L^p)$ through $T$ or $\Id -T$, for any $T$ on $L^1(L^p)$.  We reduced the general, functional
analytic task to concrete, probabilistic and combinatorial problems, and solved the latter in \cite{lechner:motakis:mueller:schlumprecht:2021}.  The
Local Theory of Banach spaces provided a framework and tools, see Milman and Schechtman~\cite{milman:schechtman:1986}.  Related quantitative, finite
dimensional factorization problems arose in the work of Bourgain and Tzafriri~\cite{bourgain:tzafriri:1989, bourgain:tzafriri:1987} on restricted
invertibility of operators.

\smallskip

The initial motivation for the present paper was the open problem of whether $L^p(L^1)$, $1 < p < \infty$ and the related $H^1(L^1 )$ and $L^1(H^1 )$ are
primary spaces.  Our main results establish the required factorization theorems in the special case of Haar multipliers on the bi-parameter Haar
system.  We thereby obtain a first step towards proving that these spaces are primary.  Our approach works in a unified way and includes also the
spaces $L^1(L^p)$, $1 < p < \infty$ considered in our previous work~\cite{lechner:motakis:mueller:schlumprecht:2021}.  Even more, our factorization results
include the large class of bi-parameter Haar system Hardy spaces (see \Cref{dfn:hsh-spaces} and \Cref{thm:hsh-factor} for the factorization result).

\subsection{Concepts, methods and techniques}
In the preceding work~\cite{lechner:motakis:mueller:schlumprecht:2021} we proved that $L^1(L^p)$ is primary.  We described there the problem's origin
in Lindenstrauss's~\cite{lindenstrauss:1971} famous research program set forth around 1970, its history and its connection to the indecomposable
Banach spaces constructed by Gowers and Maurey~\cite{gowers:maurey:1993}.

We now discuss ideas, methods and techniques pertaining to factorization of operators on classical Banach spaces.  Our review covers the main steps of
the development.  It begins with the pioneering work of Enflo-Maurey and gradually builds the stage on which current research is taking place.
\subsubsection{\bf The Enflo-Maurey theorem~\cite{maurey:1975:2}}
For $1\le p < \infty$, Enflo via Maurey~\cite{maurey:1975:2} proved that $L^p$ is primary.  Given a linear operator $T\colon L^p \to L^p$ they proved that
there exists $\alpha \in \mathbb{R}$ such that for any $\varepsilon > 0$ there exist bounded operators $A\colon L^p \to L^p$ and $B\colon L^p \to L^p$ such that
\begin{equation}\label{18-9-1}
  AB = \Id,
  \qquad \|A\|\cdot \|B\| \le C_0 ,  
\end{equation}
and
\begin{equation}\label{18-9-2}
  \| ATB -\alpha \Id \| < \varepsilon ,
\end{equation}
where $C_0 < \infty$ is a universal constant independent of $p,T,\alpha,\varepsilon$.  We then say that $T$ is an approximate projectional factor of
$\alpha \Id $, or equivalently that $\alpha \Id $ approximately projectionally factors through $T$.  Separately, we say that $A,B$ are exterior projectional
factors of $\alpha \Id$. (We refer to \Cref{dfn:factorization-modes}.)

\smallskip
\paragraph{\bf Step 1} The first step of the Enflo-Maurey method exhibits a sequence of scalars $(\alpha_I : I \in \mathcal{D}) $ such that for
$\varepsilon > 0$ there exists a block basis of the Haar system $(\tilde{h}_I)$ such that
\begin{align*}
  A_1\colon f
  \mapsto \sum_{I\in \mathcal{D}}
  \langle \tilde{h}_I, f\rangle h_I/|I|,
  \quad\text{and}\quad
  B_1\colon f
  \mapsto \sum_{I\in \mathcal{D}} \langle h_I, f \rangle \tilde{h}_I /|I|,
\end{align*}
are contractions on $L^p$ such that $A_1B_1 = \Id$, and $S_1 = A_1TB_1$ is a small perturbation of a Haar multiplier satisfying
\begin{align*}
  \bigl| \langle S_1 h_I, h_I \rangle - \alpha_I |I| \bigr|
  < \varepsilon |I|^ 5 , \text{ if } I\in \mathcal{D},
  \quad\text{and}\quad
  | \langle S_1 h_I, h_J \rangle |
  < \varepsilon ( |I| \cdot |J|)^5,  \text{ if }I \neq J\in \mathcal{D}.
\end{align*}

The block basis $(\tilde{h}_I)$ constructed by Enflo-Maurey is what we call ``a faithful Haar system'' (see \Cref{dfn:faithful}).  This means that the
orthogonal projection
\begin{equation*}
  P(f)
  = \sum_{I\in \mathcal{D}} \langle \tilde{h}_I, f \rangle \tilde{h}_I/|I|,
\end{equation*}
coincides with the conditional expectation, $\cond( f | \mathcal{F} )$, where $\mathcal{F}$ is the $\sigma$-algebra generated by the sub-system
$(\tilde{h}_I : I \in \mathcal{D} )$, and $f \in \langle \{ h_I : I \in \mathcal{D} \} \rangle$, where $\langle V \rangle$ denotes the linear span of a set of vectors $V$.

\smallskip
\paragraph{\bf Step 2}
The Enflo-Maurey proof utilizes Liapunov's convexity theorem to pass to Haar multipliers with stable entries.  It yields a single scalar $\alpha$ and
a block basis of the Haar basis $(\tilde{k}_I)$ (in fact a faithful Haar system) such that
\begin{align*}
  A_2\colon  f
  &\mapsto \sum_{I\in \mathcal{D}} \langle \tilde{k}_I, f \rangle h_I/|I|,\\
  B_2\colon  f
  &\mapsto \sum_{I\in \mathcal{D}} \langle h_I, f \rangle \tilde{k}_I/|I|,
\end{align*}
are contractions on $L^p$ for which $A_2B_2 = \Id$ and $S_2 = A_2S_1B_2$ is a small perturbation of $\alpha\Id$ on $L^p$ satisfying
\begin{align*}
  & \bigl| \langle S_2 h_I, h_I \rangle - \alpha |I| \bigr|
    < \varepsilon |I|^ 5,
    \quad I  \in  \mathcal{D},\\
  &| \langle S_2 h_I, h_J \rangle |
    < \varepsilon (|I|\cdot |J|)^5,
    \quad I \neq J \in  \mathcal{D} .
\end{align*}

\subsubsection{\bf Capon's theorem on Banach spaces with symmetric bases~\cite{capon:1983}}
Capon~\cite{capon:1983} proved that for any Banach space $E$ with a symmetric basis $(e_i)$, the Bochner Lebesgue space $L^p(E)$ is primary for
$1\le p < \infty$.  For every bounded linear operator $T\colon L^p(E)\to L^p(E)$, Capon~\cite{capon:1983} determines $\alpha \in \mathbb{R}$ such that
$T$ is an approximate factor of $\alpha\Id$.  The norms of the exterior projectional factors obtained by Capon depend only on the symmetry constant of the
basis $(e_i)$ in $E$.  In the course of proving this result, Capon~\cite{capon:1983} was lead to introduce the following property: $E$ together with
the symmetric basis $(e_i)$ satisfies condition $\mathcal{P}_p$ if there exists $K < \infty $ such that for any finite sequence $f_i \in \langle \{ h_I \} \rangle$,
\begin{equation}\label{18-9-4}
  \Bigl\| \sum_i ( \cond_i f_i) e_i \Bigr\|_{L^p(E)}
  \le K \Bigl\| \sum_i f_i e_i \Bigr\|_{L^p(E)} ,
\end {equation}
where $\cond_i$ denotes the orthogonal projection onto $\langle \{ h_I : |I| \ge 2^{-i} \} \rangle$.  In short, $E$ with the symmetric basis $(e_i)$ satisfies
$\mathcal{P}_p$ if the projection
\begin{equation*}
  C(f)
  = \sum_i (\cond_if_i) e_i,
\end{equation*}
introduced by Capon~\cite{capon:1983}, extends boundedly to $L^p(E)$, in which case we write $E\in \mathcal{P}_p$.

As noted by Capon for $1 \le r < \infty$ and $1 < p < \infty$ we find that $\ell^r$ with the standard unit vector basis satisfies $\mathcal{P}_p$, in short
$\ell^r\in \mathcal{P}_p$.  By contrast, she provides examples showing that Capon's projection $C$ does not extend boundedly on the spaces
\begin{equation*}
  L^1(\ell^r),
  \quad\text{and hence}\quad
  \ell^r\notin \mathcal{P}_1,
  \quad 1 < r < \infty .
\end{equation*}
Now we turn to review Capon's method of factoring operators $T\colon L^p(E) \to L^p(E) $ in the special case when $E \notin \mathcal{P}_p$.  This case is of particular
interest to our present paper, which is predominantly motivated by factoring operators on the following spaces
\begin{equation*}
  L^p(L^1)
  \text{ with $1 < p < \infty$},
  \quad H^1(L^1),
  \quad L^1(H^1) .
\end{equation*}

\smallskip
\paragraph{\bf Step 1}
First, Capon reduces all matters to \emph{$E$-diagonal} operators.  For $\varepsilon > 0$, she determines exterior projectional factors
$A\colon L^p(E)\to L^p(E)$ and $B\colon L^p(E)\to L^p(E) $ such that
\begin{equation*}
  ATB
  \text{ is $\varepsilon$-close to an $E$-diagonal operator }
  S\colon L^p(E)\to L^p(E).
\end{equation*}
Recall that $S$ is $E$-diagonal, if there exists a sequence of operators $S_i\colon L^p \to L^p$ such that
\begin{equation*}
  S(f)
  = \sum_i ( S_if_i) e_i,
\end{equation*}
where $f \in L^p(E)$ is given by $f = \sum_i f_i e_i$.

It is worth pointing out that the concept of {$E$-diagonal} operators corresponds to operator valued Haar multipliers. See Girardi~\cite{girardi:2007}
and Wark~\cite{wark:2017,wark:2021}.

\smallskip
\paragraph{\bf Step 2}
After a further reduction, Capon verifies that one may assume that each of the operators $S_i$ is a Haar multiplier, that is,
\begin{equation*}
  S_i (h_I)
  = d_{I,i} h_I,
  \quad d_{I,i}\in \mathbb{R},\ I\in \mathcal{D},\ i \in \mathbb{N}.
\end{equation*}

\smallskip
\paragraph{\bf Step 3a}
Under the mild assumption that the space $E$ does not contain a copy of $\ell^1$, the operators $S = (S_i) $ admit a weakly convergent subsequence.
Specifically, there exists a subsequence $(n_i)$ and $S_\infty\colon L^p \to L^p$ such that
\begin{equation}\label{18-9-9}
  \lim_{i\to\infty} \langle S_{n_i}(f), g \rangle
  = \langle S_{\infty}(f), g \rangle ,
\end{equation}
for $f\in L^p$, $g \in L^r$, with $1/p +1/r =1 $ and $ 1 \le p < \infty$.  One needs the additional assumption on $E$

for the case $p = 1$.  When $1 < p < \infty$ the limiting $S_\infty$ exists without conditions on $E$.

Now Capon obtains a further block basis $(\tilde{h}_I)$ generating contractive exterior projectional factors $A, B\colon L^p \to L^p$ such that
$\tilde{S}_\infty = AS_\infty B$ and $\tilde{S}_i = A S_{n_i} B$ are each Haar multipliers, and that the associated sequences of coefficients
$(\tilde{d}_{I})$ and $ (\tilde{d}_{I,i}) $ satisfy joint stabilization conditions as follows: For $I\in \mathcal{D}$, the limit
\begin{equation*}
  \lambda(I)
  = \lim_{j\to \infty} \sum_{J \subset I,\, |J| =2^{-j}} \tilde{d}_J |J|
\end{equation*}
exists and satisfies
\begin{equation*}
  | \tilde{d}_{I} |I| - \lambda(I) |
  < \varepsilon |I|^ 5,
  \quad \text{for $I\in \mathcal{D}$}.
\end{equation*} 
We then say that the Haar multiplier $\tilde{S}_\infty$ is stable.  Simultaneously, $\tilde{S}_i$ is stable from level $i$ onwards, meaning that
\begin{equation*}
  \lambda_i(I)
  = \lim_{j\to \infty} \sum_{J \subset I , \, |J| =2^{-j} } \tilde{d}_{J,i} |J|
\end{equation*}
exists and satisfies
\begin{equation}\label{18-9-6}
  \bigl| \tilde{d}_{I,i} |I| - \lambda_i(I) \bigr|
  < \varepsilon |I|^5,
  \quad \text{for $I\in \mathcal{D}$, $|I| \le 2^{-i}$, $i\in \mathbb{N}$}.
\end{equation}
The set functions $\lambda(\cdot)$ and $\lambda_i(\cdot)$, as defined initially on dyadic intervals in $[0, 1)$, may be extended uniquely to
absolutely continuous measures, still denoted $\lambda$ and $\lambda_i$.  Moreover, their respective densities form a bounded sequence in
$L^{\infty}$.

\smallskip
\paragraph{\bf Step 3b}
Capon's proof continues with a close examination of the coefficients $(\tilde{d}_{I})$ and $(\tilde{d}_{I,i})$ and gives explicit formulas for the
densities of $\lambda$ and $\lambda_i$.

For $t\in [0,1)$ and $k \in \mathbb{N}$ let the dyadic interval $I = I_k(t)$ be determined by the conditions $|I | = 2^{-k}$, $t \in I$.  Then put
\begin{equation*}
  m_i (t)
  = \lim_{k\to \infty} \tilde{d}_{I_k(t), i}
  \quad\text{and}\quad
  m_\infty (t)
  = \lim_{k\to \infty} \tilde{d}_{I_k(t)}.
\end{equation*}
In view of the above stability conditions, the limits used to form $m_i$ and $m_\infty $ are well defined and
\begin{equation*}
  \lambda(I)
  = \int_I m_\infty (t) dt,
  \quad\text{and}\quad
  \lambda_i (I)
  = \int_I m_i (t) dt,
  \quad\text{for $I\in \mathcal{D}$, $i\in \mathbb{N}$}.
\end{equation*}
Moreover, $\sup_i \|m_i\|_{L^{\infty}} < \infty$, hence there exists a subsequence $(n_i)$ and $m\in L^\infty$ such that
\begin{equation}\label{18-9-14}
  \lim_{i\to\infty} \langle m_{n_i}, f \rangle
  = \langle m , f \rangle,
  \quad f \in L^1.
\end{equation}

\smallskip
\paragraph{\bf Step 3c}
All further development of Capon's proof depends on the relation between the two bounded functions $m_\infty$ and $m$ just defined.  Under the assumption
that $E \notin \mathcal{P}_p$, she proves that the limits defining $m_\infty$ and $m$ are necessarily interchangeable, and that
\begin{equation}\label{18-9-16}
  m(t)
  = \lim_{i\to \infty} \lim_{k\to \infty} \tilde{d}_{I_k(t), n_i}
  = \lim_{k\to \infty} \lim_{i\to \infty} \tilde{d}_{I_k(t), n_i}
  = m_\infty(t) .  
\end{equation}

\smallskip
\paragraph{\bf Step 4}
We continue assuming that $E\notin \mathcal{P}_p$.  Exploiting the crucial identity $m = m_\infty $, Capon improves upon~\eqref{18-9-9}, and shows that strong convergence
holds
\begin{equation}\label{18-9-11}
  \lim_{i\to\infty} \| \tilde{S}_{n_i}(g) -  \tilde{S}_{\infty}(g)\|_{L^p}
  = 0,
  \qquad g\in L^ p.
\end{equation}
Consequently, for a suitable refinement of the subsequence $(n_i) $, still denoted $(n_i) $, we have
\begin{equation}\label{18-9-13}
  \Bigl\|
  \sum_{i=1}^\infty \bigl( \tilde{S}_{n_i}(f_i) e_i - \tilde{S}_{\infty}(f_i) e_i \bigr)
  \Bigr\|_{L^p(E)}
  < \varepsilon \|f\|_{L^p(E)},
\end{equation}
where
\begin{equation*}
  f
  = \sum_i f_i e_i \in L^p(E).
\end{equation*}

\smallskip
\paragraph{\bf Step 5.}
Finally applying the Enflo-Maurey factorization to the limiting operator $\tilde{S}_{\infty}$, Capon readily obtains a simultaneous factorization for the
entire sequence of convergent operators $(\tilde{S}_i )$.

\subsubsection{\bf Capon's theorem on mixed norm spaces~\cite{capon:1982:2}.}
With the results published by Capon in 1983, the important classical Banach spaces
\begin{equation*}
  L^p(\ell^q),
  \quad 1 \le p,q < \infty,
\end{equation*}
were shown to be primary.  At about the same time, Capon turned to the reflexive mixed norm spaces, and in 1982 published the
proof~\cite{capon:1982:2} that the spaces
\begin{equation*}
  L^p(L^q ),
  \quad 1< p,q < \infty,
\end{equation*}
are primary.  Here Capon's~\cite{capon:1982:2} approach utilizes that for $1< p,q < \infty$ the tensor product Haar system,
\begin{equation*}
  h_I \otimes h_J,
  \quad I, J \in \mathcal{D},
\end{equation*}
forms an unconditional basis in $L^p(L^q)$.  Let $T\colon L^p(L^q) \to L^p(L^q)$ be a bounded linear operator and $\varepsilon > 0$.  Capon~\cite{capon:1982:2}
determines a scalar sequence $(a_{I ,J} : I, J \in \mathcal{D})$ satisfying
\begin{equation*}
  |a_{I ,J}|
  \ge 1/2
\end{equation*}
and exterior projectional factors $A,B\colon L^p(L^q) \to L^p(L^q)$ such that for $S = ATB $ or $S = A(\Id -T )B $ we have
\begin{equation*}
  \Bigl\|
  Sf - \sum_{I,J} a_{I , J} \langle h_I \otimes h_J, f \rangle h_I \otimes h_J / ( |I| \cdot |J | )
  \Bigr\|_{ L^p(L^q )}
  \le \varepsilon \|f\|_{L^p(L^q)}.
\end{equation*}
Thus either $T$ or $\Id -T$ projectionally factors through a Haar multiplier with entries satisfying $|a_{I, J}| \ge 1/2 $.

Capon's exterior projectional factors $A, B$ are determined by a bi-parameter block basis
\begin{equation*}
  \tilde{h}_{I , J}
  = \sum_{(K,L)\in \mathcal{A}_{I,J}} h_K \otimes h_L,
\end{equation*}
where $(\mathcal{A}_{I,J} : \, I,J\in \mathcal{D})$ are pairwise disjoint collections of dyadic rectangles.  We have
\begin{align*}
  A\colon f
  &\mapsto \sum_{I, J \in \mathcal{D}} \langle \tilde{h}_{I,J}, f \rangle h_I\otimes h_J / (|I|\cdot|J|),\\
  B\colon f
  &\mapsto \sum_{I, J \in \mathcal{D}} \langle h_I\otimes h_J, f \rangle \tilde{h}_{I,J} / (|I| \cdot |J |).
\end{align*}
We emphasize, that the block basis $(\tilde{h}_{I, J})$ satisfies Capon's \emph{local} product condition; it is, however, \emph{not} of tensor product
structure, and the unconditionality of the Haar system in $L^p(L^q)$ is a crucial ingredient in Capon's~\cite{capon:1982:2} proof of the norm
estimates
\begin{equation}\label{09-10-1}
  \|A \| \cdot \|B\|
  \le C_0(p,q).
\end{equation}
The constants $C_0(p,q) $ are unbounded as $p \text{ or } q \to 1 \text{ or } \infty$.

Exploiting Capon's proof that $L^p(L^q)$, $1< p,q < \infty$ is primary, the third named author~\cite{mueller:1994} obtained the analogous result for the
bi-parameter Hardy space $H^1(H^1)$.  However, proving primarity for any of the spaces
\begin{equation*}
  L^1(L^p )
  \quad\text{or}\quad
  L^p(L^1), \qquad 1 \le p < \infty,
\end{equation*}
is out of reach for the approach developed by Capon~\cite{capon:1982:2}.  Only in 2022 the present
authors~\cite{lechner:motakis:mueller:schlumprecht:2021} proved that the spaces $L^1(L^p) $, $1 < p < \infty$ are primary.

Finally, neither Capon's approach nor the one in our 2022 paper~\cite{lechner:motakis:mueller:schlumprecht:2021} applies to the limiting spaces
\begin{equation*}
  L^1(H^1)
  \quad\text{and}\quad
  H^1(L^1) .
\end{equation*}

\subsubsection{\bf The present paper}
The present paper is motivated by the open problem of whether primarity holds for any of the spaces
\begin{equation*}
  L^p(L^1),\ 1 < p < \infty ,
  \qquad L^1 (H^1)
  \qquad\text{or}\qquad
  H^1(L^1). 
\end{equation*}
One of our main results in this paper yields that Haar multipliers on these spaces are approximate projectional factors of $\alpha \Id $ for some $\alpha \in \mathbb{R}$.

We approach these problems in a unified way, covering also the spaces $L^1(L^p)$ treated in~\cite{lechner:motakis:mueller:schlumprecht:2021}.  We
consider
\begin{equation*}
  \mathcal{F}
  = \{ L^p(L^q),\ H^p(H^q),\ L^p(H^q),\ H^p(L^q) : 1 \le p,q < \infty \},
\end{equation*}
and a subset
\begin{equation*}
  \mathcal{G}
  = \{ L^1(L^p),\ L^1(H^p),\ H^p(L^1) : 1 \le p < \infty \} .
\end{equation*}
By way of example we show that {\em Capon's projection} $ \mathcal{C}$ associated to the bi-parameter Haar system is unbounded on $Z \in \mathcal{G}$, where
\begin{equation*}
  \mathcal{C}(f)
  = \sum_{I \in \mathcal{D}} \sum_{J \in \mathcal{D},\, |I| \leq |J|} \frac{\langle h_I \otimes h_J, f \rangle}{|I|\cdot| J|} h_I \otimes h_J,
  \qquad\text{for}\ f\in \langle \{ h_I \otimes h_J \} \rangle.
\end{equation*}
On the other hand, the bi-parameter Haar system is an unconditional basis in $Z \in \mathcal{F} \setminus \mathcal{G}$ and hence any bounded
multiplier array acting on the bi-parameter Haar system gives rise to a continuous operator on the spaces in $\mathcal{F} \setminus \mathcal{G}$.

\smallskip
\paragraph{\bf Factoring through multipliers of the bi-parameter Haar system}
Let $Z \in \mathcal{F}$ and let $D\colon Z\to Z$ be a bounded multiplier acting on the bi-parameter Haar system by a multiplier array,
$(d_{I,J} : I, J \in \mathcal{D})$, that is,
\begin{equation*}
  D (h_I \otimes h_J )
  = d_{I,J } h_I \otimes h_J,
  \quad I,J \in \mathcal{D} .
\end{equation*}
Given the multiplier array $(d_{I,J})$, we prepare the statement of our main factorization theorems by defining the {\em factorization functionals}
$\lambda$ and $\mu$.
\begin{enumerate}
\item Any two fixed integers $i, j \in \mathbb{N}_0$ give rise to a pavement of the unit square by the following collection of pairwise disjoint
  dyadic rectangles,
  \begin{equation*}
    \mathcal{R}(i, j)
    = \{ (I,J) \in \mathcal{D}\times \mathcal{D} : |I| = 2 ^{-i}, |J| =2^{-j} \}.
  \end{equation*}
  The cardinality of $\mathcal{R}(i, j)$ equals $2^ {i+j}$ and $|I| \cdot |J| = 2^ {-i-j} $ for $(I,J) \in \mathcal{R}(i, j)$.  We use
  $\mathcal{R}(i, j)$ to define an average of the entries in the multiplier array as follows
  \begin{equation*}
    E_{i, j}
    = \sum_{(I,J)\in \mathcal{R}(i, j)} |I| \cdot |J| d_{I,J}
    = 2^{-i-j} \sum_{(I,J)\in \mathcal{R}(i, j)} d_{I,J}.
  \end{equation*}

\item We fix a non-principal ultrafilter $\mathcal{U}$ on $\mathbb{N}$ and form iterated row and column limits of the matrix $(E_{i,j})$ along
  $\mathcal{U}$.  This gives rise to two distinct limits of the averages $E_{i, j}$
  \begin{equation}\label{29-9-2}
    \lambda_\mathcal{U}(D)
    = \lim_{j\to\mathcal{U}} \lim_{i\to\mathcal{U}} E_{i, j}
    \quad\text{and}\quad
    \mu_\mathcal{U}(D)
    = \lim_{i\to\mathcal{U}} \lim_{j\to\mathcal{U}} E_{i, j}.
  \end{equation}
\end{enumerate}

\smallskip
\paragraph{\bf Summary of the main results}
The assertions of our factorization theorem for $D\colon Z \to Z$ are linked directly to $\lambda = \lambda_\mathcal{U}(D) $ and
$\mu = \mu_\mathcal{U}(D)$, and critically depend on the norm of Capon's projection on~$Z$.
\begin{enumerate}
\item If Capon's projection $\mathcal{C}$ does not extend to a bounded operator on $Z$ then, we show that necessarily the iterated limits defining
  $\lambda$ and $\mu$ are interchangeable and we have equality,
  \begin{equation*}
    \lambda
    = \lim_{j\to\mathcal{U}} \lim_{i\to\mathcal{U}} E_{i, j}
    = \lim_{i\to\mathcal{U}} \lim_{j\to\mathcal{U}} E_{i, j}
    = \mu  .
  \end{equation*}
  For $\varepsilon > 0$, there exist two faithful Haar systems $(\tilde{h}_I : I\in \mathcal{D})$ and $(\tilde{k}_J : J\in \mathcal{D})$ such that their respective tensor products
  \begin{equation*}
    \bigl( \tilde{h}_I \otimes \tilde{k}_J : I, J\in \mathcal{D} \bigr)
  \end{equation*}
  generate bounded exterior projectional factors
  \begin{equation}\label{eq:27}
    \begin{aligned}
      A \colon  f
      &\mapsto \sum_{I, J \in \mathcal{D}} \frac{\langle\tilde{h}_I \otimes \tilde{k}_J , f\rangle}{|I|\cdot|J|} h_I\otimes h_J,\\
      B \colon  f
      &\mapsto \sum_{I ,  J \in \mathcal{D}} \frac{\langle h_I\otimes h_J, f\rangle}{|I|\cdot|J|}\tilde{h}_I\otimes \tilde{k}_J.
    \end{aligned}
  \end{equation}
  satisfying
  \begin{equation}\label{eq:28}
    \|A D B - \mu \Id \|_{L^p(L^1)}
    \le \varepsilon.
  \end{equation}

\item If $\mathcal{C}$ extends to a bounded operator on $Z$, then for $\varepsilon > 0$, there exist of two faithful Haar systems $(\tilde{h}_I)$ and
  $(\tilde{k}_J)$ such that, by means of \eqref{eq:27}, their tensor products generate bounded exterior projectional factors $A,B\colon Z \to Z$
  satisfying
  \begin{equation}\label{21-9-1}
    \bigl\| ADB -  \bigl( \lambda \mathcal{C} + \mu (\Id_Z - \mathcal{C}) \bigr)\bigr\|
    < \varepsilon.
  \end{equation}
  We may split the information encoded by~\eqref{21-9-1} into the following three cases:
  \begin{enumerate}
  \item If $(\lambda, \mu)\notin \{(1,0), (0,1)\}$ then $\Id_Z$ factors through
    \begin{equation}\label{09-10-2}
      \lambda \mathcal{C} + \mu (\Id_Z - \mathcal{C})
      \quad\text{or}\quad
      (1-\lambda) \mathcal{C} + (1-\mu) (\Id_Z - \mathcal{C}).
    \end{equation}
    We note in passing that the identities $\lambda_\mathcal{U}(\Id- D) = 1-\lambda_\mathcal{U}(D)$ and
    $\mu_\mathcal{U}(\Id- D) = 1-\mu_\mathcal{U}(D)$ are used in the second alternative of \eqref{09-10-2}.
  \item If $(\lambda, \mu) = (1,0)$ then
    \begin{equation*}
      \|ADB -  \mathcal{C}\|
      < \varepsilon.
    \end{equation*}
  \item If $ (\lambda, \mu ) = (0,1)$ then
    \begin{equation*}
      \|ADB - (\Id_Z - \mathcal{C})\|
      < \varepsilon.
    \end{equation*}
  \end{enumerate}
\end{enumerate}

\smallskip
\paragraph{\bf Remarks:}
\begin{enumerate}[label=(\alph*)]
\item Exploiting bi-tree semi-stabilization, related in spirit to the Semenov-Uksusov~\cite{semenov:uksusov:2012} characterization of bounded Haar
  multipliers on $L^1 $, and the concentration of measure phenomenon for empirical processes, the full combinatorial and probabilistic force of our
  efforts is directed at proving the existence of two faithful Haar systems such that in~\eqref{eq:27}, their respective tensor products generate
  bounded exterior projectional factors $A,B$, for which the factorization estimates~\eqref{eq:28} respectively~\eqref{21-9-1} hold true.

\item The tensor product structure of the block basis defining $A,B$ is then crucial in our proof that
  \begin{equation}\label{09-10-3}
    \|A \| \cdot \|B\|
    \le 1,
  \end{equation}
  independent of $Z \in\mathcal{F}$.  This should be contrasted with Capon's constants $C_0(p,q)$ in~\eqref{09-10-1} which are unbounded as
  $p \text{ or } q \to 1 \text{ or } \infty$.
\end{enumerate}

\smallskip
\paragraph{\bf Haar system Hardy spaces}
Initially, our work was concerned with the factorization of operators defined on the well known spaces in the family
\begin{equation*}
  \mathcal{F}
  = \{ L^p(L^q),\ H^p (H^q),\ L^p(H^q),\ H^p(L^q) : 1 \le p,q < \infty \}.
\end{equation*}
In the present paper, we exhibit a novel method for constructing tensor product bases and exterior projectional factors on a class of bi-parameter
function spaces that includes the family $\mathcal{F}.$ This opened the door for a \emph{unified approach} to solving factorization problems.

In this paper, we consider operators on the class of bi-parameter function spaces called \emph{Haar System Hardy Spaces}, denoted
$\mathcal{H} \mathcal{H}(\delta^2)$. See \Cref{dfn:hsh-spaces}.  The class of Haar system Hardy spaces is broad enough to contain the family
$\mathcal{F}$, and is of independent interest.  Let
\begin{equation*}
  Z \in \mathcal{H}\mathcal{H}(\delta^2).
\end{equation*}
The main result of the present paper asserts that for any bounded multiplier operator $D\colon Z \to Z $ acting by an array $(d_{I,J})$ on the
bi-parameter Haar system $(h_I\otimes h_J)$ at least one of the following statements is true:
\begin{enumerate}[label=(\roman*)]
\item $D\colon Z \to Z$ is an approximate projectional factor of $\lambda \Id_Z$,
\item $D\colon Z \to Z$ is an approximate projectional factor of $\lambda \mathcal{C} + \mu (\Id_Z - \mathcal{C})$, where $\mathcal{C}$ denotes Capon's projection.
\end{enumerate}
The first assertion holds when $\mathcal{C}$ does not extend to a bounded operator on $Z$.  The second one holds when $\mathcal{C}$ is bounded on $Z$. The factorization
functionals $\lambda = \lambda_\mathcal{U} (D) $ and $\mu = \mu_\mathcal{U} (D)$, depending on the ultrafilter $\mathcal{U} $ and the multiplier array
$(d_{I,J})$, are predetermined by \eqref{29-9-2}.

\subsection{Organization of the paper}
\label{sec:organization-paper}

In \Cref{sec:general-definitions} we state the main results of the paper and discuss in detail their usefulness in proving that (some) bi-parameter
Haar system Hardy spaces are primary.

\Cref{sec:semi-stabilization} is devoted to presenting the main combinatorial and probabilistic results, forming our novel stabilization method for
the entries of a multiplier array $(d_{I,J}).$ The inductive proof is facilitated by showing that our stability estimates persist under specific
iterated blockings of faithful Haar systems.

In \Cref{sec:haar-system-hardy-1+2} we present the basic properties of Haar system Hardy spaces, particularly the boundedness of exterior projectional
factors generated by tensor product Haar systems.

In \Cref{sec:mult-capons-proj} we investigate the pointwise multiplication operators generated by a stabilized Haar multiplier $D$.  Those play a
mediating role between the stabilized multiplier array and its derived factorization functionals.  We thus obtain the crucial identity
$\lambda_{\mathcal{U}} = \mu_{\mathcal{U}}$, when the operator norm of Capon's projection $\mathcal{C}$ is unbounded on $Z$.

\section{Main results}
\label{sec:general-definitions}
In this section, we introduce the necessary notation, state our main results, and lay out the pathway towards proving them.

\subsection{Reduction of classes of operators: main results for $L^p(L^1)$, $L^1(H^1)$, and $H^1(L^1)$}
A Banach space $X$ is called \emph{primary} if whenever $X\simeq Y\oplus Z$, where $Y$, $Z$ are closed subspaces of $X$, then $X\simeq Y$ or
$X\simeq Z$.  We focus on Banach spaces of functions of two parameters.  The spaces $L^p(L^q)$, $1<p,q<\infty$ were proved to be primary by Capon in
\cite{capon:1982:2}. $H^1(H^1)$ was proved to be primary by the third author in \cite{mueller:1994}.  Among spaces of this type that involve $L^1$,
only $L^1(L^p)$ was proved to be primary in \cite{lechner:motakis:mueller:schlumprecht:2021}.  The purpose of this paper is to prove results about the
factoring properties of classes of operators of Banach spaces in a general class, called Haar system Hardy spaces, that contains the mixed
Lebesgue-Hardy spaces $L^p(L^1)$, $H^1(L^1)$, $L^1(H^1)$, and explain how these results fit within a scheme of proof of the primarity of these spaces,
which is not yet known.

We begin by recalling the necessary notions to formally state our main results.
\begin{ntn}\label{Haar preliminaries}
  Let $\mathbb{N} = \{1,2,\ldots\}$ and $\mathbb{N}_0 = \mathbb{N}\cup\{0\}$.
  \begin{enumerate}[label=(\alph*)]
  \item For any set of vectors $V$, we denote its linear span by $\langle V \rangle$.
    
  \item We denote by $\mathcal{D}$ the collection of all dyadic intervals in $[0,1)$, namely
    \begin{equation*}
      \mathcal{D}
      = \Big\{\Big[\frac{i-1}{2^j},\frac{i}{2^j}\Big) : j\in\mathbb{N}_0,\ 1\leq i\leq 2^j\Big\}.
    \end{equation*}
    Moreover, we put $\mathcal{D}_j = \{ I \in \mathcal{D} : |I| = 2^{-j}\}$.

  \item For $I\in \mathcal{D}$, let $I^+$ denote the left half of $I$ and $I^-$ denote the right half of $I$, i.e., $I^+$ is the largest
    $J\in\mathcal{D}$ with $J\subsetneq I$ and $\inf J = \inf I$, and $I^- = I\setminus I^+\in\mathcal{D}$.

  \item The Haar system $(h_I)_{I\in\mathcal{D}}$ is defined by
    \begin{equation*}
      h_I
      = \chi_{I^+} - \chi_{I^-},
      \qquad I\in\mathcal{D}
    \end{equation*}
    and we denote by $\mathcal{V}(\delta)$ its linear span.

  \item We let $(k_J:J\in \mathcal{D})$ be a copy of the Haar system $(h_I:I\in\mathcal{D})$ and define for $I,J\in \mathcal{D}$
    \begin{equation*}
      h_I\otimes k_J:[0,1)^2\to \mathbb{R},
      \quad (s,t)\to h_I(s)k_J(t).
    \end{equation*}
    We call $(h_I\otimes k_J:I,J\in\mathcal{D})$ the \emph{bi-parameter Haar system} and denote by $\mathcal{V}(\delta^2)$ its linear span.  More
    generally for $f,g:[0,1)\to \mathbb{R}$ we write
    \begin{equation*}
      f\otimes g: [0,1)^2\to \mathbb{R},
      \quad (s,t)\mapsto f(s)g(t).
    \end{equation*}
  
  \item For $f,g\in \mathcal{V}(\delta)$, or $f,g\in \mathcal{V}(\delta^2)$ we define
    \begin{equation*}
      \langle f,g \rangle
      = \int_0^1 f(s)g(s) \,d s,
      \quad\text{respectively}\quad
      \langle f,g \rangle
      = \int_0^1\int_0^1 f(s,t) g(s,t)\, \mathrm{d} s \mathrm{d} t.
    \end{equation*}
  \item A linear operator $D\colon \mathcal{V}(\delta)\to\mathcal{V}(\delta)$, or $D\colon \mathcal{V}(\delta^2)\to\mathcal{V}(\delta^2)$, is called a
    \emph{Haar multiplier} if every one-parameter Haar vector $h_I$, respectively if every bi-parameter Haar vector $h_I\otimes k_J$ is an eigenvector
    of $D$. We denote the space of all Haar multipliers by ${\text{\rm HM}}(\delta)$ and ${\text{\rm HM}}(\delta^2)$, respectively.  The eigenvalues
    of a Haar multiplier $D$ are called the {\em coefficients of $D$}.
  \end{enumerate}
\end{ntn}

For $1\leq p < \infty$ let $L^p$ denote the Banach space of $p$-integrable functions on $[0,1)$ with mean zero, i.e., the closure of
$\mathcal{V}(\delta)$ under the $\|\cdot\|_p$-norm.  For $1\leq p,q < \infty$ we define the mixed-norm Lebesgue space $L^p(L^q)$ to be the completion
of $\mathcal{V}(\delta^2)$ under the norm
\begin{equation}\label{eq:29}
  \|f\|_{L^p(L^q)}
  = \biggl(
  \int_0^1 \Big(
  \int_0^1|f(t,s)|^q\mathrm{d} s
  \Big)^{p/q}\mathrm{d} t
  \biggr)^{1/p}.
\end{equation}
More generally, for a given Banach space $E$ we define the Bochner-Lebesgue space $L^p(E)$ to consist of all Bochner-measurable $E$-valued functions
$f\colon [0,1]\to E$ for which
\begin{equation*}
  \|f\|_{L^p(E)}
  := \Bigl( \int_0^1 \| f(t) \|_E^p \mathrm{d} t \Bigr)^{1/p}
  < \infty.
\end{equation*}
Endowing $L^p(E)$ with the norm $\|\cdot\|_{L^p(E)}$ gives rise to a Banach space.  For $E = L^q$, we obtain~\eqref{eq:29}.  If $E = H^1$ we obtain
$L^1(H^1)$ where $H^1$ denotes the completion of $\mathcal{V}(\delta)$ under the norm
\begin{equation}\label{10-10-1}
  \Bigl\| \sum_{I\in \mathcal{D}} a_{I} h_I \Bigr\|_{H^1}
  = \int_{0}^{1} \cond \Bigl|\sum_{I\in \mathcal{D}} \sigma_I a_{I} h_I(s) \Bigr| \mathrm{d} s,
\end{equation}
where $\sum_{I\in \mathcal{D}} a_{I} h_I \in \mathcal{V}(\delta)$ and where $\cond$ denote the expectation with respect to the Haar measure over all
$\sigma = (\sigma_I)_{I\in\mathcal{D}}\in\{\pm1\}^\mathcal{D}$.

Next, we define the vector-valued dyadic Hardy space $H^1(L^1)$ as the completion of $\mathcal{V}(\delta^2)$ under the norm
\begin{equation}\label{eq:33}
  \Bigl\| \sum_{I,J\in \mathcal{D}} a_{I,J} h_I\otimes k_J \Bigr\|_{H^1(L^1)}
  = \int_{0}^{1} \int_{0}^{1} \cond \Bigl|
  \sum_{I,J\in \mathcal{D}} \sigma_I a_{I,J} h_I(s) k_J(t)
  \Bigr| \mathrm{d}t\mathrm{d}s.
\end{equation}
By evaluating the inner expectation with respect to $\sigma$ using Khintchine's inequality, we see that the above norm is equivalent to
\begin{equation*}
  \int_0^1\int_0^1 \biggl(
  \sum_{I\in \mathcal{D}} \Big(
  \sum_{J\in \mathcal{D}}a_{I,J}k_J(s)
  \Big)^2
  h^2_I(t)
  \biggr)^{1/2} \mathrm{d}t\mathrm{d}s.
\end{equation*}
We link the above expression to the literature on vector-valued Hardy spaces $H^1(E)$.  Given an $L^1(E)$-convergent dyadic $E$-valued martingale
$f = (f_j)_{j=0}^{\infty}$ we define its $H^1(E)$-norm to be
\begin{equation}\label{eq:32}
  \|f\|_{H^1(E)}
  = \cond \int_{0}^{1} \Bigl\|
  \sum_{j=0}^{\infty} \sigma_j (f_{j+1}(t) - f_j(t))
  \Bigr\|_E
  \mathrm{d}t,
\end{equation}
where now, $\cond$ denote the expectation with respect to the Haar measure over $ (\sigma_j)_{j\in\mathbb{N}}\in\{\pm1\}^\mathbb{N}$.  For $E = L^1$,
the expression in~\eqref{eq:32} coincides with the one in~\eqref{eq:33}.  We refer to~\cite{MR1091439}, \cite{MR2204960}
and~\cite{pisier:2016:martingales}.

We study the factorization properties of operators on such spaces.  Factors of the identity are by now classical and have been considered in varying
contexts, e.g., in \cite{dosev:johnson:2010}, \cite{laustsen:lechner:mueller:2015}, \cite{lechner:motakis:mueller:schlumprecht:2021}, and
\cite{kania:lechner:2022}. Specifically, in \cite{lechner:motakis:mueller:schlumprecht:2021} (and elsewhere), factors of the identity were used to
show that certain Banach spaces are primary. Our main result, \Cref{main theorem with spaces}, is a significant step towards the proof of primarity of
the spaces $L^p(L^1)$, $H^1(L^1)$, $L^1(H^1)$. This is explained in detail, using the language of \Cref{dfn:factorization-modes} in \Cref{factors
  remarks} and \Cref{scheme and justification old}.

\begin{dfn}\label{dfn:factorization-modes}
  Let $X$ be a Banach space.  Denote by $\mathcal{L}(X)$ the space of bounded linear operators on $X$ and denote by $\Id\colon X\to X$ the identity
  map.

  \begin{enumerate}[label=(\alph*)]
  \item\label{enu:dfn:factorization-modes:a} Let $S,T\in\mathcal{L}(X)$.  We say that \emph{$T$ is a $C$-factor of $S$ } if there exist operators
    $A,B\colon X\to X$ with $\|A\|\cdot \|B\|\leq C$ such that $S = ATB$.  In this case, we also say that $S$ $C$-factors through $T$.

  \item\label{enu:dfn:factorization-modes:b} Let $S,T\in\mathcal{L}(X)$.  We say that \emph{$T$ is a $C$-projectional factor with error $\eta\ge 0$ of
      $S$} if there exist operators $A,B\colon X\to X$ with $\|A\|\|B\|\leq C$, $AB = \Id$ and such that $\|S - ATB\|\leq \eta$.  Alternatively, we
    say that $S$ projectionally $C$-factors through $T$ with error $\eta$.
    
  \item\label{enu:dfn:factorization-modes: reduction} Let $\mathcal{A}$, $\mathcal{B}$ be subclasses of $\mathcal{L}(X)$, and $C>0$.  We say that
    \emph{$\mathcal{A}$ approximately $C$-projectionally reduces to $\mathcal{B}$} if for every $T$ in $\mathcal{A}$ and $\eta>0$ there exists $S$ in
    $\mathcal{B}$ such that $T$ is a $C$-projectional factor with error $\eta$ of $S$.

  \item\label{enu:dfn:factorization-modes:c} Let $\mathcal{A}$ be a subclass of $\mathcal{L}(X)$.  We say that \emph{$\mathcal{A}$ has the $C$-primary
      factorization property} if for any operator $T\in\mathcal{A}$ we have that either $T$ or $\Id-T$ is a $C$-factor of $\Id$.
    
  \item\label{enu:dfn:factorization-modes:d} We say that \emph{$X$ has the $C$-primary factorization property} if $\mathcal{L}(X)$ has the $C$-primary
    factorization property.
  \end{enumerate}
  If we omit the constant $C>0$ in any of the above properties, we understand this to mean that this property is satisfied for some $0 < C < \infty$.
\end{dfn}

The versatility of factors of the identity is reflected in the following definition, applicable to any Banach space $X$.  Put
$\mathcal{M}_X = \{T\in\mathcal{L}(X): \text{$T$ is not a factor of $\Id$}\}$.  This set is closed in the operator topology and under multiplication
from the right and the left by bounded operators.  If it happens to be closed under addition, then it is a maximal two-sided ideal in
$\mathcal{L}(X)$.  Note that items \ref{enu:dfn:factorization-modes:a} to \ref{enu:dfn:factorization-modes:d} in \Cref{dfn:factorization-modes} can be
stated for an arbitrary unital Banach algebra in the place of $\mathcal{L}(X)$.

Below is the main theorem of this paper for Banach spaces of functions of two parameters that involve $L^1$. It follows by combining
\Cref{thm:hsh-factor} and \Cref{thm:capon-projection}.
\begin{thm}\label{main theorem with spaces}
  Let $Z$ be one of the spaces $L^1(L^p)$, $L^p(L^1)$, $1\leq p<\infty$, $L^1(H^1)$, or $H^1(L^1)$ and denote by $\mathrm{HM}(Z)$ the unital
  subalgebra of $\mathcal{L}(Z)$ of all bounded Haar multipliers on $Z$.  Then, $\mathrm{HM}(Z)$ approximately $1$-projectionally reduces to the class
  $\{\lambda \Id:\lambda\in\mathbb{R}\}$ of scalar operators.  In particular, $\mathrm{HM}(Z)$ has the $C$-primary factorization property, for every
  $C > 2$.
\end{thm}

We make some elementary remarks that we then use to put our main result into context.
\begin{rem}\label{factors remarks}\hfill
  \begin{enumerate}[label=(\alph*)]
  \item In \Cref{dfn:factorization-modes}~\ref{enu:dfn:factorization-modes:b} note that $BA$ is a projection onto a subspace of $X$ that is isomorphic
    to $X$ and, thus, the term ``projectional factor'' is used.

  \item If $T$ is a $C$-projectional factor of $S$ with error $\eta\ge 0$ then also $\Id-T$ is a $C$-projectional factor of $\Id-S$ with error $\eta$.

  \item\label{factors remarks transitivity} By \cite[Proposition 2.3]{lechner:motakis:mueller:schlumprecht:2021}, the relation in
    \Cref{dfn:factorization-modes}~\ref{enu:dfn:factorization-modes: reduction} satisfies a certain transitivity property; if $\mathcal{A}_1$
    approximately $C$-projectionally reduces to $\mathcal{A}_2$ and $\mathcal{A}_2$ approximately $D$-projectionally reduces to $\mathcal{A}_3$, then
    $\mathcal{A}_1$ approximately $CD$-projectionally reduces to $\mathcal{A}_3$.

  \item\label{scalar primary factorization} It is straight forward that for any Banach space $X$ the class $\{\lambda \Id:\lambda\in\mathbb{R}\}$ of
    scalar operators has the $2$-primary factorization property.

  \item\label{pointwise reduction} For $X = L^p$, $1\leq p\leq \infty$, the class of pointwise multipliers on $X$,
    $\mathrm{PM}(X) = \{M_g\in \mathcal{L}(X) : g\in L^\infty\}$, approximately 1-projectionally reduces to the class of scalar operators.  This
    follows from the fact that for every measurable subset $A$ of $[0,1)$ of positive measure, $L^p(A,P|_A)$ is isometrically isomorphic to $L^p$ and
    1-complemented in $L^p$.
  \end{enumerate}
\end{rem}

Some explanation is due about the placement of \Cref{main theorem with spaces} within the context or primarity of Banach spaces.
\begin{rem}\label{scheme and justification old}
  Assume that a Banach space $X$ satisfies the following two properties.
  \begin{enumerate}[label=(\alph*)]
  \item\label{scheme and justification a} $X$ satisfies the accordion property, i.e., $X\simeq \ell^p(X)$, for some $1\leq p\leq \infty$, or $X\simeq c_0(X)$.

  \item\label{scheme and justification b} $X$ satisfies the primary factorization property.
  \end{enumerate}
  Then, $X$ is primary.
\end{rem}

Condition~\ref{scheme and justification a} was introduced by Pe{\l}czy{\'n}ski in \cite{pelczynski:1960} to prove that the Banach spaces $\ell^p$,
$1\leq p<\infty$, and $c_0$ are prime.  The observation that \ref{scheme and justification a} and \ref{scheme and justification b} yields primarity was
implicitly made by Lindenstrauss and Pe\l czy\'nski in \cite{lindenstrauss:pelczynski:1971} where they used it to deduce that $C[0,1]$ is primary.
For an argument justifying \Cref{scheme and justification old}, see, e.g., \cite[Proposition 2.4]{lechner:motakis:mueller:schlumprecht:2021}.

Sometimes, proving \ref{scheme and justification b} goes through several steps that involve the relation introduced in
\Cref{dfn:factorization-modes}~\ref{enu:dfn:factorization-modes: reduction}.  Thus, condition~\ref{scheme and justification b} is frequently replaced
with the following equivalent condition.
\begin{enumerate}
\item[(b$'$)]\label{scheme and justification new} there are classes of operators $\mathcal{L}(X) = \mathcal{A}_1$,\ldots,$\mathcal{A}_n$ in
  $\mathcal{L}(X)$ such that $\mathcal{A}_i$ approximately projectionally reduces to $\mathcal{A}_{i+1}$, for $1\leq i<n$.
\end{enumerate}
Condition~(b$'$) was used implicitly, e.g., by Maurey via Enflo~\cite{maurey:1975:1}, who proved that in $X = L^p$, $\mathcal{L}(X)$ approximately
1-projectionally reduces to the class of pointwise multipliers.  By \Cref{factors remarks}~\ref{pointwise reduction}, this class approximately
projectionally reduces to the class of scalar operators, which has the primary factorization property.  A similar process was implicitly used by
Alspach, Enflo, and Odell \cite{alspach:enflo:odell:1977} to give a simpler proof that $X = L^p$, $1<p<\infty$ is primary, by reducing
$\mathcal{L}(X)$ to the class of Haar multipliers.  They then used ideas of Gamlen and Gaudet from \cite{gamlen:gaudet:1973} and exploited unconditionality to
further reduce that class to the one of scalar operators.  Capon used more elaborate versions of this scheme for $\ell^p(L^1)$, $1<p<\infty$,
$L^p(\ell^r)$ and $L^p(L^r)$, $1<p,r<\infty$ \cite{capon:1980:2,capon:1982:1,capon:1982:2}.  The authors of this paper also followed this scheme in
\cite{lechner:motakis:mueller:schlumprecht:2021} to prove that $L^1(L^p)$, $1< p <\infty$, is primary.

We conclude this subsection with a comment on a potential proof of primarity of the spaces $L^p(L^1)$, $L^1(H^1)$, and $H^1(L^1)$.  Let $Z$ denote one
of them.  The last missing step is to prove that $\mathcal{L}(Z)$ approximately projectionally reduces to $\mathrm{HM}(Z)$.  Indeed, if this is the case, then,
by the transitivity property explained in \Cref{factors remarks} \ref{factors remarks transitivity}, $\mathcal{L}(Z)$ would approximately projectionally reduce
to the class of scalar operators.  Because $Z$ has the accordion property, by \Cref{scheme and justification old}, $Z$ would be primary.

\subsection{Haar system Hardy spaces}\label{29-9-1}
At this point, we introduce a new class of bi-parameter spaces called {\em Haar System Hardy Spaces}.  This is a broad class of independent interest,
and it includes $L^1(L^p)$, $L^p(L^1)$, $1\leq p < \infty$, $L^1(H^1)$, and $H^1(L^1)$; for the spaces in this class, we prove a generalization of \Cref{main
  theorem with spaces} (see \Cref{thm:hsh-factor}).

We recall the notion of Haar system spaces from \cite{lechner:motakis:mueller:schlumprecht:2021}.  This variation of classical rearrangement invariant
Banach spaces includes all such spaces for which the Haar system is a basis, as well as the Cantor space $C(\Delta)$, defined as the closure of
$\mathcal{V}(\delta)$ in $L^\infty $, equipped with the $L^\infty$-norm.
\begin{dfn}\label{dfn:HS-1d}
  A \emph{Haar system space $W$} is the completion of
  \begin{equation*}
    \mathcal{V}_1(\delta)
    = \langle \{\chi_{[0,1}\}\cup\{ h_I : I\in\mathcal{D}\} \rangle
    = \langle \{\chi_I : I\in\mathcal{D}\} \rangle
  \end{equation*}
  under a norm $\|\cdot\|$ that satisfies the following properties.
  \begin{enumerate}[label=(\alph*)]
  \item\label{dfn:HS-1d:1} If $f$, $g$ are in $\mathcal{V}_1(\delta)$ and $|f|$, $|g|$ have the same distribution then $\|f\| = \|g\|$.
  \item\label{dfn:HS-1d:2} $\|\chi_{[0,1)}\| = 1$.
  \end{enumerate}
  We denote the class of Haar system spaces by $\mathcal{H}(\delta)$.
\end{dfn}

The following proposition lists some basic properties of Haar system spaces. Properties~\ref{RI properties 1}, \ref{RI properties 4}, and~\ref{RI
  properties 2} were shown in \cite[Proposition 2.13]{lechner:motakis:mueller:schlumprecht:2021}, \ref{RI properties 3} follows from \ref{RI
  properties 2}, and \ref{RI properties 5} from \ref{RI properties 3}.
\begin{pro}\label{RI properties}
  Let $W$ be a Haar system space.
  \begin{enumerate}[label=(\roman*),leftmargin=23pt]
  \item\label{RI properties 1} For every $f\in Z = \langle\{\chi_I:I\in\mathcal{D}\}\rangle$ we have $\|f\|_{L_1} \leq \|f\| \leq \|f\|_{L_\infty}$.

  \item \label{RI properties 4} The Haar system, in the usual linear order, is a monotone Schauder basis of $W$.

  \item\label{RI properties 2} $Z = \langle\{\chi_I:I\in\mathcal{D}\}\rangle$ naturally coincides with a subspace of $W^*$ and its closure
    $\overline Z$ in $W^*$ is also a Haar system space.

  \item\label{RI properties 3} For $f\in W$, we have
    \begin{equation*}
      \|f\|
      = \sup_{g\in Z, \|g\|_*\le 1} \int f(x)g(x)\,dx
    \end{equation*}
    where $\|\cdot\|_*$ is the dual norm on $W^*$.

  \item\label{RI properties 5} If $f,g\in W$, $|f| \le |g|$ then $\|f\| \le \|g\|$.
  \end{enumerate}
\end{pro}

Given a space $X\in\mathcal{H}(\delta)$, we can define a one-parameter Hardy space as follows.  Let $\cond$ denote the expectation with respect to the Haar measure
over all $(\sigma_I)_{I\in\mathcal{D}}$ in $\{-1,1\}^\mathcal{D}$.  The Haar system Hardy space $X(\boldsymbol{\sigma})$ is defined as the completion of
$\mathcal{V}(\delta)$ with the following norm.  If $f = \sum_{I\in\mathcal{D}}a_Ih_I\in\mathcal{V}(\delta)$ then
\begin{equation*}
  \|f\|_{X(\boldsymbol{\sigma})}
  = \Big\|\cond\big|\sum_{I\in\mathcal{D}}\sigma_Ia_Ih_I\big|\Big\|_X
  = \Big\|t\mapsto\cond\big|\sum_{I\in\mathcal{D}}\sigma_Ia_Ih_I(t)\big|\Big\|_X.
\end{equation*}

We define and investigate the two-parameter version that involves two spaces $X$, $Y$ in $\mathcal{H}(\delta)$.

\begin{dfn}\label{dfn:hsh-spaces}\hfill
  \begin{enumerate}[label=(\alph*)]
  \item For $X$, $Y\in\mathcal{H}(\delta)$ we define the {\em bi-parameter Haar system space} $X(Y)$ as the completion of $\mathcal{V}(\delta^2)$ under the following norm.  If
    $f = \sum_{I,J}a_{I,J}h_I\otimes k_J$ then
    \begin{equation*}
      \|f\|_{X(Y)}
      = \Bigl\|
      s\mapsto\Bigl\|
      t\mapsto
      \sum_{I,J\in\mathcal{D}} a_{I,J}h_I(s) k_J(t)
      \Bigr\|_Y
      \Bigr\|_X.
    \end{equation*}
    This is well defined, since, if $f\in\mathcal{V}(\delta^2)$ then $\|t\mapsto f(\cdot,t)\|_Y$ is in $\mathcal{V}_1(\delta)$, but, obviously, may
    not be in $\mathcal{V}(\delta)$.
  \item Let $\boldsymbol{\sigma}_{1,1} = \{\pm 1\}^\mathcal{D}\times \{\pm 1\}^\mathcal{D}$,
    $\boldsymbol{\sigma}_{0,0} = \{1\}^\mathcal{D}\times \{1\}^\mathcal{D}$,
    $\boldsymbol{\sigma}_{0,1} = \{1\}^\mathcal{D}\times \{\pm 1\}^\mathcal{D}$, and
    $\boldsymbol{\sigma}_{1,0} = \{\pm 1\}^\mathcal{D}\times \{1\}^\mathcal{D}$.  Denote
    $\boldsymbol{\Sigma} = \{\boldsymbol{\sigma}_{1,1},\boldsymbol{\sigma}_{0,0},\boldsymbol{\sigma}_{0,1},\boldsymbol{\sigma}_{1,0}\}$.  For a fixed
    $\boldsymbol{\sigma}\in\boldsymbol{\Sigma}$ let $\cond$ denote the expectation with respect to the Haar product measure over all
    $\sigma = \big((\sigma^{(1)}_I)_{I\in\mathcal{D}},(\sigma^{(2)}_J)_{J\in\mathcal{D}}\big)$ in $\boldsymbol{\sigma}$.

  \item For $X$, $Y\in\mathcal{H}(\delta)$, $\boldsymbol{\sigma}\in\boldsymbol{\Sigma}$ we define the {\em bi-parameter Haar system Hardy space}
    $Z = Z(\boldsymbol{\sigma},X,Y)$ as the completion of $\mathcal{V}(\delta^2)$ under the following norm.  If $f = \sum_{I,J}a_{I,J}h_I\otimes k_J$
    then
    \begin{align*}
      \|f\|_Z
      &= \Bigl\|\cond\Bigl|
        \sum_{I,J\in\mathcal{D}} \sigma^{(1)}_I\sigma^{(2)}_J a_{I,J}h_I\otimes k_J
        \Bigr|
        \Big\|_{X(Y)}\\
      &= \Bigl\|
        s\mapsto\Bigl\|
        t\mapsto\cond \bigl|
        \sum_{I,J\in\mathcal{D}} a_{I,J}\sigma^{(1)}_I\sigma^{(2)}_J h_I(s) k_J(t)
        \bigr|
        \Bigr\|_Y
        \Bigr\|_X.
    \end{align*}
  \end{enumerate}
  We denote by $\mathcal{HH}(\delta^2)$ the class of all bi-parameter Haar system Hardy spaces.
\end{dfn}

\begin{rem}\label{rem:HSH-classical-examples}
  By specializing the Haar system Hardy spaces, many classical spaces are contained in $\mathcal{HH}(\delta^2)$.  Let
  $Z = Z(\boldsymbol{\sigma},X,Y)\in\mathcal{HH}(\delta^2)$.
  \begin{itemize}
  \item If $\boldsymbol{\sigma} = \boldsymbol{\sigma}_{0,0} = \{1\}^\mathcal{D}\times \{1\}^\mathcal{D}$, then for $f\in Z$,
    \begin{equation*}
      \|f\|_{Z}
      = \Bigl\|
      \sum_{I,J\in\mathcal{D}} a_{I,J} h_I\otimes k_J
      \Bigr\|_{X(Y)}.
    \end{equation*}
    Thus, the two-parameter Haar system Hardy space $Z$ becomes the two-parameter Haar system space $X(Y)$.  Further specifying $X=L^p$, $Y=L^q$,
    $1\leq p,q < \infty$, we obtain the classical two parameter Lebesgue-space $L^p(L^q)$.
    
  \item If $\boldsymbol{\sigma} = \boldsymbol{\sigma}_{1,1} = \{\pm1\}^\mathcal{D}\times \{\pm1\}^\mathcal{D}$, then for $f\in Z$,
    \begin{equation*}
      \|f\|_{Z}
      = \Bigl\|
      \cond\bigl|
      \sum_{I,J\in\mathcal{D}}  \sigma_I^{(1)}\sigma_J^{(2)}a_{I,J} h_I\otimes k_J
      \bigr|
      \Bigr\|_{X(Y)}.
    \end{equation*}
    Specializing $X=L^p$, $Y=L^q$, $1\leq p,q < \infty$, we obtain the classical bi-parameter dyadic Hardy space $H^p(H^q)$.  In general, by the
    monotonicity of the norms in $X$, $Y$ (see \Cref{RI properties}~\ref{RI properties 5}) and Khintchine's inequality, we obtain that $\|\cdot\|_Z$
    is equivalent to the square norm given for $f = \sum_{I,J\in\mathcal{D}} a_{I,J} h_I\otimes k_J$ by
    \begin{equation*}
      \Bigl\|
      \Bigl(
      \sum_{I,J\in\mathcal{D}} a_{I,J}^2 h_I^2\otimes k_J^2
      \Bigr)^{1/2}
      \Bigr\|_{X(Y)}.
    \end{equation*}
    Further specializing $X=L^p$, $Y=L^q$, $1\leq p,q < \infty$, we obtain the classical bi-parameter dyadic Hardy space $H^p(H^q)$.

  \item If $\boldsymbol{\sigma} = \boldsymbol{\sigma}_{0,1} = \{1\}^\mathcal{D}\times \{\pm1\}^\mathcal{D}$, then for
    $f = \sum_{I,J\in\mathcal{D}} a_{I,J} h_I\otimes k_J$, we have
    \begin{equation*}
      \|f\|_{Z}
      = \Bigl\|
      \cond\bigl|
      \sum_{I,J\in\mathcal{D}}  a_{I,J} \sigma_J^{(2)} h_I\otimes k_J
      \bigr|
      \Bigr\|_{X(Y)}.
    \end{equation*}
    Specializing $X=L^p$, $Y=L^q$, $1\leq p,q < \infty$, we obtain the classical space $L^p(H^q)$.  In general, by Khintchine's inequality and the
    monotonicity of the norms in $X,Y$, we obtain that the norm $\|\cdot\|_{Z}$ is equivalent to the partial square function norm given by
    \begin{equation*}
      \Bigl\|
      \Bigl(
      \sum_{J\in\mathcal{D}} \bigl(
      \sum_{I\in\mathcal{D}} a_{I,J} h_I
      \bigl)^2\otimes k_J^2
      \Bigr)^{1/2}
      \Bigr\|_{X(Y)}.
    \end{equation*}

  \item If $\boldsymbol{\sigma} =\boldsymbol{\sigma}_{1,0} = \{\pm1\}^\mathcal{D}\times \{1\}^\mathcal{D}$, then for
    $f = \sum_{I,J\in\mathcal{D}} a_{I,J} h_I\otimes k_J$, we have
    \begin{equation*}
      \|f\|_{Z}
      = \Bigl\|\cond \bigl|
      \sum_{I,J\in\mathcal{D}} a_{I,J} \sigma_I^{(1)} h_I\otimes k_J
      \bigr|
      \Bigr\|_{X(Y)}.
    \end{equation*}
    Specializing $X=L^p$, $Y=L^q$, $1\leq p,q < \infty$, we obtain the classical space $H^p(L^q)$.  In general, using Khintchine's inequality and the
    monotonicity of the norms in $X,Y$, we obtain that $\|\cdot\|_{Z}$ is equivalent to the partial square function norm given by
    \begin{equation*}
      \Bigl\|
      \sum_{I\in\mathcal{D}} \Bigl(
      \bigl(
      \sum_{J\in\mathcal{D}} a_{I,J} k_J
      \bigl)^2\otimes h_I^2
      \Bigr)^{1/2}
      \Bigr\|_{X(Y)}.
    \end{equation*}
  \end{itemize}
\end{rem}

In \Cref{thm:capon-projection} below, we give necessary conditions for the unboundedness of the Capon projection.  In particular, the Capon projection
is unbounded on the following spaces: $L^1(L^p)$, $L^p(L^1)$, $L^1(H^1)$, $H^1(L^1)$.

\begin{thm}\label{thm:capon-projection}
  Let $Z = Z(\boldsymbol{\sigma},X,Y)\in\mathcal{HH}(\delta^2)$ and suppose that
  \begin{itemize}
  \item $(\sigma_I^{(1)})$ is constant and either $\|\cdot\|_X\sim \|\cdot\|_{L^1}$ or $\|\cdot\|_X\sim \|\cdot\|_{L^\infty}$, or
  \item $(\sigma_J^{(2)})$ is constant and either $\|\cdot\|_Y\sim \|\cdot\|_{L^1}$ or $\|\cdot\|_Y\sim \|\cdot\|_{L^\infty}$,
  \end{itemize}
  then $\mathcal{C}\colon Z\to Z$ is unbounded.
\end{thm}

Before proving the above theorem, we give an overview of classical examples and the boundedness of the Capon projection in the table below.  The rows
of the table correspond to the norm of the space $X$ and the letters ``c'' and ``i'' indicate whether the sequence $(\sigma_I^{(1)})$ is constant or
independent.  Similarly, the columns correspond to the norm of the space $Y$, and the letters ``c'' and ``i'' indicate whether the sequence
$(\sigma_J^{(2)})$ is constant or independent .  The entries ``B''/``U'' in the table below mean the Capon projection is bounded/unbounded on the space
$Z = Z(\mathbf{\sigma},X,Y)$.
\begin{longtable}[h]{c|c|c|c|c|c|c|}
  \cline{2-7}
  \vphantom{\LARGE L}                                 & $L^1$, c & $L^1$, i & $L^p$ & $L^q$ & $L^{\infty}$, i & $L^{\infty}$, c \\
  \hline
  \multicolumn{1}{|c|}{\vphantom{\LARGE L}$L^1$, c}   & U        & U        & U     & U     & U          & U          \\
  \hline
  \multicolumn{1}{|c|}{\vphantom{\LARGE L}$L^1$, i}   & U        & B        & B     & B     & B          & U          \\
  \hline
  \multicolumn{1}{|c|}{\vphantom{\LARGE L}$L^p$}      & U        & B        & B     & B     & B          & U          \\
  \hline
  \multicolumn{1}{|c|}{\vphantom{\LARGE L}$L^q$}      & U        & B        & B     & B     & B          & U          \\
  \hline
  \multicolumn{1}{|c|}{\vphantom{\LARGE L}$L^{\infty}$, i} & U        & B        & B     & B     & B          & U          \\
  \hline
  \multicolumn{1}{|c|}{\vphantom{\LARGE L}$L^{\infty}$, c} & U        & U        & U     & U     & U          & U          \\
  \hline
  \caption{Capon projection table.}
\end{longtable}
In the table above, ``$L^1$, i'' corresponds to $H^1$, and ``$L^{\infty}$, i'' corresponds to the norm closure of the Haar system in $SL^{\infty}$.

\begin{myproof}{Proof of \Cref{thm:capon-projection}}
  Given $Z = Z(\boldsymbol{\sigma},X,Y)\in\mathcal{HH}(\delta^2)$, we distinguish between the following $6$ cases:
  
  \begin{minipage}{1.0\linewidth}
    \begin{enumerate}[label=Case~(\arabic*)]
    \item\label{enu:proof:thm:capon-projection:1} $\|\cdot\|_X\sim \|\cdot\|_{L^1}$, $(\sigma_I^{(1)})$ is constant, and
      \begin{enumerate}[label=(\alph*)]
      \item\label{enu:proof:thm:capon-projection:1:a} $(\sigma_J^{(2)})$ is independent;
      \item\label{enu:proof:thm:capon-projection:1:c} $(\sigma_J^{(2)})$ is constant and $\|\cdot\|_Y\nsim \|\cdot\|_{L^{\infty}}$.
      \end{enumerate}
    \item\label{enu:proof:thm:capon-projection:2} $\|\cdot\|_Y\sim \|\cdot\|_{L^1}$, $(\sigma_J^{(2)})$ is constant, and
      \begin{enumerate}[label=(\alph*)]
      \item\label{enu:proof:thm:capon-projection:2:a} $(\sigma_I^{(1)})$ is independent;
      \item\label{enu:proof:thm:capon-projection:2:c} $(\sigma_I^{(1)})$ is constant and $\|\cdot\|_X\nsim \|\cdot\|_{L^{\infty}}$.
      \end{enumerate}
    \item\label{enu:proof:thm:capon-projection:3} $\|\cdot\|_X\sim \|\cdot\|_{L^\infty}$ and $(\sigma_I^{(1)})$ is constant;
    \item\label{enu:proof:thm:capon-projection:4} $\|\cdot\|_Y\sim \|\cdot\|_{L^\infty}$ and $(\sigma_J^{(2)})$ is constant.
    \end{enumerate}
  \end{minipage}

  Since the proof of~\ref{enu:proof:thm:capon-projection:2}~\ref{enu:proof:thm:capon-projection:2:a} and
  \ref{enu:proof:thm:capon-projection:2}~\ref{enu:proof:thm:capon-projection:2:c} are almost identical to the proofs of
  \ref{enu:proof:thm:capon-projection:1}~\ref{enu:proof:thm:capon-projection:1:a} and
  \ref{enu:proof:thm:capon-projection:1}~\ref{enu:proof:thm:capon-projection:1:c} (swapping the functions $f$ and $g$ defined below), we will omit
  them.  The situation for the proofs of \ref{enu:proof:thm:capon-projection:3} and \ref{enu:proof:thm:capon-projection:4} is similar.

  \begin{proofcase}[\ref{enu:proof:thm:capon-projection:1}, \ref{enu:proof:thm:capon-projection:1:a} and~\ref{enu:proof:thm:capon-projection:1:c}]
    We put $I_k = J_k = [0,2^{-k})$, $k\in \mathbb{N}_0$ and define our test functions
    \begin{equation}\label{eq:1}
      f = \sum_{k=0}^{n} |I_k|^{-1} h_{I_k} = |I_{n+1}|^{-1} \chi_{I_{n+1}} - \chi_{[0,1)},
      \qquad\text{and}\qquad
      g = \sum_{l=0}^{n} a_l r_l,
    \end{equation}
    where $r_l = \sum_{J\in \mathcal{D}_l} k_J$ and $(a_k)_{k=0}^n$ is a finite sequence of scalars.  On the one hand, we note that
    \begin{equation}\label{eq:37}
      \|f\otimes g\|_Z
      \sim \Bigl\|
      \bigl|
      \sum_{k=0}^{n} |I_k|^{-1} h_{I_k}
      \bigr|
      \Bigr\|_{L^1}
      \cdot \Bigl\|
      \cond \bigl|
      \sum_{l=0}^n a_l \sum_{J\in \mathcal{D}_l} \sigma_J^{(2)} k_J
      \bigr|
      \Bigr\|_Y
      \leq 2\cdot \Bigl\|
      \cond \bigl|
      \sum_{l=0}^n a_l \sum_{J\in \mathcal{D}_l} \sigma_J^{(2)} k_J
      \bigr|
      \Bigr\|_Y,
    \end{equation}
    while the other hand, using $\|\cdot\|_Y\geq \|\cdot\|_{L^1}$ (see \Cref{RI properties}~\ref{RI properties 1}) and Khintchine's inequality yields
    \begin{align*}
      \|\mathcal{C} f\otimes g\|_Z
      &\sim \Bigl\|
        s\mapsto \Bigl\|
        t\mapsto \cond \bigl|
        \sum_{l=0}^{n} a_l \sum_{J\in \mathcal{D}_l} \sigma_J^{(2)} k_J(t)
        \sum_{k=l}^{n} |I_k|^{-1} h_{I_k}(s)
        \bigr|
        \Bigr\|_Y
        \Bigr\|_{L^1}\\
      &\geq \Bigl\|
        s\mapsto \cond \int_{0}^{1} \Bigl|
        \sum_{l=0}^{n} a_l \sum_{J\in \mathcal{D}_l} \sigma_J^{(2)} k_J(t)
        \sum_{k=l}^{n} |I_k|^{-1} h_{I_k}(s)
        \Bigr|
        \mathrm{d} t
        \Bigr\|_{L^1}\\
      &= \Bigl\|
        s\mapsto \cond \int_{0}^{1} \Bigl|
        \sum_{l=0}^{n} a_l r_l(t)
        \sum_{k=l}^{n} |I_k|^{-1} h_{I_k}(s)
        \Bigr|
        \mathrm{d} t
        \Bigr\|_{L^1}
        \sim \Bigl\|
        \Bigl( 
        \sum_{l=0}^{n} a_l^2 
        \bigl(\sum_{k=l}^{n} |I_k|^{-1} h_{I_k}\bigr)^{2}
        \Bigr)^{1/2}
        \Bigr\|_{L^1}\\
      &= \Bigl\|
        \Bigl( 
        \sum_{l=0}^n a_l^2 \bigl( |I_{n+1}|^{-1} \chi_{I_{n+1}} - |I_l|^{-1} \chi_{I_l}\bigr)^2
        \Bigr)^{1/2}
        \Bigr\|_{L^1}.
    \end{align*}
    Using that $\bigl||I_{n+1}|^{-1} \chi_{I_{n+1}} - |I_l|^{-1} \chi_{I_l}\bigr| \geq |I_l|^{-1} \chi_{I_l^{-}}$ and the disjointness of the
    intervals $I_l^-$, $l\in \mathbb{N}_0$, we can further estimate
    \begin{align*}
      \|\mathcal{C} f\otimes g\|_Z
      &\gtrsim \Bigl\|
        s\mapsto \Bigl( 
        \sum_{l=0}^{n} a_l^2 |I_l|^{-2} \chi_{I_l^-}(s)
        \Bigr)^{1/2}
        \Bigr\|_{L^1}
        = \Bigl\| s\mapsto
        \sum_{l=0}^{n} |a_l| |I_l|^{-1} \chi_{I_l^-}(s)
        \Bigr\|_{L^1}
        = \frac{1}{2} \sum_{l=0}^{n} |a_l|.
    \end{align*}
    Suppose that $\mathcal{C}$ is bounded then the latter estimate together with~\eqref{eq:37} yields
    \begin{equation}\label{eq:5}
      \Bigl\|
      \cond \bigl|
      \sum_{l=0}^n a_l \sum_{J\in \mathcal{D}_l} \sigma_J^{(2)} k_J
      \bigr|
      \Bigr\|_Y
      \gtrsim \sum_{l=0}^{n} |a_l|,
    \end{equation}
    for all finite sequences of scalars $(a_l)_{l=0}^n$.  If $(\sigma_J^{(2)})$ is independent as per
    \ref{enu:proof:thm:capon-projection:1}~\ref{enu:proof:thm:capon-projection:1:a}, then~\eqref{eq:5} is cannot hold since
    \begin{equation*}
      \cond \Bigl| \sum_{l=0}^n a_l \sum_{J\in \mathcal{D}_l} \sigma_J^{(2)} k_J \Bigr|
      \leq \Bigl(\sum_{l=0}^{n} a_l^2\Bigr)^{1/2}.
    \end{equation*}
    On the other hand, if $(\sigma_J^{(2)})$ is constant according to \ref{enu:proof:thm:capon-projection:1}~\ref{enu:proof:thm:capon-projection:1:c},
    then~\eqref{eq:5} yields
    \begin{equation*}
      \Bigl\|
      \sum_{l=0}^n a_l r_l
      \Bigr\|_Y
      \gtrsim \sum_{l=0}^{n} |a_l|,
    \end{equation*}
    which is only possible if $\|\cdot\|_Y\sim \|\cdot\|_{L^{\infty}}$ (see the proof
    of~\cite[Proposition~2.c.10]{lindenstrauss:tzafriri:1979:partII}).
  \end{proofcase}

  \begin{proofcase}[\ref{enu:proof:thm:capon-projection:3}]
    In this case, our test functions are not simple tensor products, but of the following form:
    \begin{equation*}
      z = \sum_{k=1}^{n} (h_{I_{2k}} - h_{I_{2k-1}})\otimes \sum_{J\in \mathcal{D}_{2k}} k_J
      = \sum_{k=1}^{n} (h_{I_{2k}} - h_{I_{2k-1}})\otimes r_{2k}
    \end{equation*}
    where $I_k = [0,2^{-k})$, $k\in \mathbb{N}_0$.  First, observe that the functions $h_{I_{2k}}- h_{I_{2k-1}}$, $k\in \mathbb{N}$ are disjointly
    supported, and hence,
    \begin{equation}\label{eq:34}
      \begin{aligned}
        \|z\|_Z
        &\sim \esssup_s
          \Bigl\| t\mapsto
          \cond \bigl|
          \sum_{k=1}^{n} \bigl( h_{I_{2k}}(s) - h_{I_{2k-1}}(s) \bigr) \sum_{J\in \mathcal{D}_{2k}} \sigma_J^{(2)} k_J(t)
          \bigr|
          \Bigr\|_Y\\
        &= \max_{1\leq k\leq n} \bigl| h_{I_{2k}}(s) - h_{I_{2k-1}}(s)\bigr|
          \leq 2.
      \end{aligned}
    \end{equation}
    On the other hand, since $\|\cdot\|_Y\geq \|\cdot\|_{L^1}$ (see \Cref{RI properties}~\ref{RI properties 1}), we obtain
    \begin{align*}
      \|\mathcal{C} z\|_Z
      &\sim \esssup_s
        \Bigl\| t\mapsto
        \bigl|
        \sum_{k=1}^{n} h_{I_{2k}}(s) r_{2k}(t)
        \bigr|
        \Bigr\|_Y
        \geq \Bigl\| t\mapsto
        \bigl|
        \sum_{k=1}^{n} h_{I_{2k}}(0) r_{2k}(t)
        \bigr|
        \Bigr\|_Y\\
      &\geq \int_{0}^{1} 
        \Bigl|
        \sum_{k=1}^{n} r_{2k}(t)
        \Bigr|
        \mathrm{d} t
        \sim \sqrt{n},
    \end{align*}
    which in view of~\eqref{eq:34} proves that $\mathcal{C}$ is unbounded.\qedhere
  \end{proofcase}
\end{myproof}

A Haar system Hardy space $Z(\boldsymbol{\sigma},X,Y)$ can be identified isometrically with a subspace of $X(Y(L^1(\boldsymbol{\sigma})))$.  This
representation is useful in studying semi-stable Haar multipliers on $Z$ that we define later in \Cref{two-parametric setting}
\begin{dfn}\label{dfn:multiplier-space}
  Let $X,Y\in\mathcal{H}(\delta)$, $\boldsymbol{\sigma}\in\boldsymbol{\Sigma}$, and $Z = Z(\boldsymbol{\sigma},X,Y)$.
  \begin{enumerate}[label=(\alph*),leftmargin=19pt]
  \item Define $Z^\Omega = X(Y(L^1(\boldsymbol{\sigma})))$ as the completion of
    \begin{equation*}
      \mathcal{V}_1(\delta)\otimes\mathcal{V}_1(\delta)\otimes L^1(\boldsymbol{\sigma})
      = \langle \{f\otimes g\otimes a: f,g\in\mathcal{V}_1(\delta), a\in L^1(\boldsymbol{\sigma})\} \rangle
    \end{equation*}
    with the norm
    \begin{align*}
      \Bigl\|
      \sum_{j=1}^nf_j\otimes g_j\otimes a_j
      \Bigr\|_{Z^\Omega}
      &= \Bigl\|
        \cond\bigl|
        \sum_{j=1}^nf_j\otimes g_j\otimes a_j(\boldsymbol{\sigma})
        \bigr|
        \Bigr\|_{X(Y)}\\
      & = \Bigl\|
        s\mapsto \Bigl\|
        t\mapsto\cond
        \bigl|\sum_{j=1}^nf_j(s)g_j(t)a_j(\boldsymbol{\sigma})
        \bigr|
        \Bigr\|_Y
        \Bigr\|_X.
    \end{align*}
  \item Define the linear map $\mathcal{O}\colon Z\to Z^\Omega$ given by
    \begin{equation*}
      z = \sum_{(I,J)\in\mathcal{D}} a_{I,J} h_I\otimes k_J
      \mapsto \sum_{(I,J)\in\mathcal{D}} a_{I,J} h_I\otimes k_J\otimes \big(\sigma^{(1)}_I\otimes \sigma^{(2)}_J\big),
    \end{equation*}
    where, for each $I_0,J_0\in\mathcal{D}$, we denote by $\sigma_{I_0}^{(1)}\otimes\sigma_{J_0}^{(2)}\colon\boldsymbol{\sigma}\to\{-1,1\}$ the
    function given by $\big((\sigma_I^{(1)})_{I\in\mathcal{D}},(\sigma^{(2)}_J)_{J\in\mathcal{D}}\big)\mapsto \sigma_{I_0}^{(1)}\sigma_{J_0}^{(2)}$.
  \end{enumerate}
\end{dfn}
Note that for $f\in X$, $g\in Y$, and $a\in L^1(\boldsymbol{\sigma})$ we have $\|f\otimes g\otimes a\|_{Z^\Omega} = \|f\|_X\|g\|_Y\|a\|_{L^1}$.
Therefore, for every $w\in\mathcal{V}_1(\delta)\otimes\mathcal{V}_1(\delta)\otimes L^1(\boldsymbol{\sigma})$ the quantity $\|w\|_{Z^\Omega}$ is well
defined.  By the definition of $Z$, the operator $\mathcal{O}\colon Z\to Z^\Omega$ is an into isometry.

\begin{rem}\label{rem:pw-multiplier}
  Let $Z = Z(\boldsymbol{\sigma},X,Y)\in\mathcal{HH}(\delta^2)$ and $m\in L^\infty([0,1)^2\times\boldsymbol{\sigma})$ such that, for all
  $w\in Z^\Omega$, the pointwise product $mw$ is in $Z^\Omega$.  We define the pointwise multiplier $M\colon Z^\Omega \to Z^\Omega$ by putting
  \begin{equation*}
    (M w)(s,t,\sigma)
    =  m(s,t,\sigma)\cdot w(s,t,\sigma),
    \qquad s,t\in [0,1),\;\sigma\in\boldsymbol{\sigma}.
  \end{equation*}
  By the monotonicity of the Haar system spaces $X,Y$ (see \Cref{RI properties}~\ref{RI properties 5}), we obtain that $M$ is well-defined and
  $\|M\|_{\mathcal{L}(Z^\Omega)}\leq \|m\|_{L^\infty}$.
  
  For context, we point out two facts that we will not use.
  \begin{enumerate}[label=(\alph*)]
  \item $\|M\|_{\mathcal{L}(Z^\Omega)} = \|m\|_{L^\infty}$.

  \item Given $m\in L^\infty([0,1)^2\times\boldsymbol{\sigma})$ and if neither the $X$ norm nor the $Y$ norm is equivalent to the
    $\|\cdot\|_\infty$-norm then, automatically, for all $w\in Z^\Omega$, $mw\in Z^\Omega$.
  \end{enumerate}
\end{rem}

The main theorem splits, in fact, into two version for a space $Z$ in $\mathcal{HH}(\delta^2)$ depending on whether a certain idempotent Haar
multiplier, the Capon projection, extends to a bounded linear operator on $Z$.

\begin{dfn}
  The \emph{Capon projection} is the bi-parameter Haar multiplier $\mathcal{C}\colon \mathcal{V}(\delta^2)\to\mathcal{V}(\delta^2)$ given by
  \begin{equation*}
    \mathcal{C}(h_I\otimes k_J)
    = \left\{
      \begin{array}{ll}
        h_I\otimes k_J  & \mbox{if } I\in\mathcal{D}_i,\;J\in\mathcal{D}_j\text{ and } i\geq j, \\
        0 & \mbox{otherwise.} 
      \end{array}
    \right.
  \end{equation*}
  That is, $\mathcal{C}$ is a projection onto the space spanned by the ``lower triangular'' part of the bi-parameter Haar system.
\end{dfn}

The main result we prove for Haar system Hardy spaces is the following. It is a direct consequence of \Cref{two-parameter analytic main theorem} and
\Cref{10-10-5}
\begin{thm}\label{thm:hsh-factor}
  \label{general bi-parameter projectional reduction theorem}
  Let $Z \in \mathcal{HH}(\delta^2)$ and let $\mathrm{HM}(Z)$ denote the set of bounded Haar multipliers on $Z$.
  \begin{enumerate}[label=(\roman*)]
  \item If $\mathcal{C}$ is unbounded on $Z$ then $\mathrm{HM}(Z)$ approximately 1-projectionally reduces to the class of scalar operators.  In
    particular, $\mathrm{HM}(Z)$ has the $C$-primary factorization property, for every $C>2$.

  \item If $\mathcal{C}$ is bounded on $Z$ then $\mathrm{HM}(Z)$ approximately 1-projectionally reduces to the two-dimensional subalgebra
    $\{\lambda\mathcal{C}+\mu(\Id-\mathcal{C}):\lambda,\mu\in\mathbb{R}\}$.
  \end{enumerate}
\end{thm}

From \Cref{thm:capon-projection} and \Cref{thm:hsh-factor} we can deduce \Cref{main theorem with spaces} about $L^1(L^p)$, $L^p(L^1)$, $L^1(H^1)$, and
$H^1(L^1)$.

\subsection{Stabilization of bi-parameter Haar multipliers}
In the remainder of this section, we will describe the process behind obtaining \Cref{general bi-parameter projectional reduction theorem}.  This
process includes reducing and solving one-parametric problems, and thus, we include some results in this setting.

\subsubsection{One-parametric stabilization}
\label{one-parametric setting}
We express the stability of a coefficient sequence $(d_I)$ over the dyadic tree $\mathcal{D}$ in terms of the following variational semi-norm, which
is similar in spirit to the characterization by Semenov and Uksusov~\cite{semenov:uksusov:2012} of the bounded Haar multipliers on $L^1$.
\begin{ntn}
  For a Haar multiplier $D\colon \mathcal{V}(\delta)\to \mathcal{V}(\delta)$, whose coefficients are $(d_I)$ define
  \begin{gather*}
    \|D\|_\infty
    = \sup_{I\in \mathcal{D}} |d_I|
    \qquad\text{and}\qquad
    \|D\|_\mathrm{T}
    = \sum_{I\in\mathcal{D}} \Big(
    |d_I - d_{I^+}| + |d_I - d_{I^-}|
    \Big)
    + |d_{[0,1)}|.
  \end{gather*}
  If $\|D\|_\infty < \infty$ we call $D$ \emph{$\ell^{\infty}$-bounded} and and denote the space of all $\ell^\infty$ bounded Haar multipliers by
  $\text{HM}^\infty(\delta)$.  We call $\|D\|_{\mathrm{T}}$ the tree variation norm of $D$ and if $\|D\|_{\mathrm{T}} < \infty$ we call $D$
  tree-stable.
    
  For a non-principal ultrafilter $\mathcal{U}$ on $\mathbb{N}$ let $\lambda_\mathcal{U}$ denote the norm-one linear functional on the space of all
  $\ell^\infty$-bounded Haar multipliers $D$
  \begin{equation*}
    \lambda_\mathcal{U}(D)
    = \lim_{i\to\mathcal{U}}\sum_{I\in\mathcal{D}_i}|I|d_I
    = \lim_{i\to\mathcal{U}} \big\langle \sum_{I\in\mathcal{D}_i} h_I, D\big( \sum_{I\in\mathcal{D}_i}h_I \big) \big\rangle.
  \end{equation*}
  Note that $\lambda_\mathcal{U}(\Id) = 1$.
\end{ntn}
In \Cref{Enflo-Maurey comparison} below, we elaborate on the connection between $\lambda_\mathcal{U}$ and a local version of the inner product
expression of $\lambda_\mathcal{U}$ used by Enflo via Maurey in~\cite{maurey:1975:1}.

\begin{dfn}\label{dfn:faithful}
  A sequence $\tilde H = (\tilde h_I)_{I\in \mathcal{D}}$ in $\mathcal{V}(\delta)$ is called a \emph{faithful Haar system} if the following conditions
  are satisfied.
  \begin{enumerate}[label=(\alph*)]
  \item\label{dfn:faithful i} For each $I\in\mathcal{D}$ there is a finite family $\mathcal{A}_I\subset\mathcal{D}$ of pairwise disjoint members of
    $\mathcal{D}$ and $(\varepsilon_K : K\in\mathcal{A}_I)\in\{\pm1\}^{\mathcal{A}_I}$ such that
    \begin{equation*}
      \tilde h_{I}
      = \sum_{K\in \mathcal{A}_I} \varepsilon_K h_K.
    \end{equation*}

  \item\label{dfn:faithful ii} If for all $I\in\mathcal{D}$ we put $\Gamma_I = \cup\{K:K\in\mathcal{A}_I\}$ then $\Gamma_{[0,1)} = [0,1)$ and
    \begin{equation*}
      \Gamma_{I^+}
      = \{\tilde h_I = 1\}
      \quad\text{and}\quad
      \Gamma_{I^-}
      = \{\tilde h_I = -1\}.
    \end{equation*}
  \end{enumerate}
  If, in addition, there exists a strictly increasing sequence $(m_i)_{i=0}^\infty$ in $\mathbb{N}_0$ such that
  \begin{enumerate}[label=(\alph*)]\setcounter{enumi}{2}
  \item\label{D:1.4.1} for all $i\in\mathbb{N}_0$ and $I\in\mathcal{D}_i$, $\mathcal{A}_I\subset \mathcal{D}_{m_i}$,
  \end{enumerate}
  then we call $(\tilde h_I)_{I\in \mathcal{D}}$ a \emph{faithful Haar system relative to the frequencies $(m_i)_{i=0}^\infty$}.
\end{dfn}

\begin{rem}\label{Enflo-Maurey comparison}\hfill
  \begin{enumerate}[label=(\arabic*)]
  \item\label{P:1.4.2} The following observation can be shown by induction.
    
    Let $(\tilde h_I)_{I\in \mathcal{D}} \subset \mathcal{V}(\delta)$ be a faithful Haar system.  Then for any
    $(a_I:I\in\mathcal{D})\in c_{00}(\mathcal{D})$, the functions $\sum_{I\in \mathcal{D}} a_I h_I$ and $\sum_{I\in \mathcal{D}} a_I \tilde h_I$ have
    the same distribution.

  \item\label{faithful short formulation} A collection $(\tilde h_I)_{I\in\mathcal{D}}$ is a faithful Haar system relative to frequencies
    $(m_i)_{i=0}^\infty$ if and only if $\supp(\tilde h_{[0,1)}) = [0,1)$ and for each $i\in\mathbb{N}_0$ and $I\in\mathcal{D}_{i}$ the following two
    conditions hold.
    \begin{enumerate}[label=(\alph*)]
    \item The vector $\tilde h_I$ is a linear combination of $(h_J)_{J\in\mathcal{D}_{m_i}}$ with coefficients in $\{-1,0,1\}$.

    \item For $\varepsilon\in\{\pm1\}$, $\supp(\tilde h_{I^\varepsilon}) = \{\tilde h_{I} = \varepsilon\}$.
    \end{enumerate}
  \end{enumerate}
\end{rem}
Faithful Haar systems will be used in the sequel to capture a stabilized behavior of a Haar multiplied $D$ on $\mathcal{V}(\delta)$ which, under
certain assumptions on a norm on $\mathcal{V}(\delta)$, leads to a factorization of a ``simple'' operator through $D$.  More precisely, we will only
use faithful Haar systems relative to some frequencies $(m_i)_{i=0}^\infty$. The additional restriction~\ref{D:1.4.1} in their definition allows a
more precise result: a diagonal operator $D$, under the right assumptions, can be expressed as an approximate projectional factor of the scalar
operator $\lambda_\mathcal{U}(D) I$, i.e., the scalar $\lambda_\mathcal{U}(D)$ is independent of the accuracy of the approximation.

\begin{ntn}\label{one-parameter AB notation}
  Let $D\colon \mathcal{V}(\delta)\to \mathcal{V}(\delta)$ be an $\ell^{\infty}$-bounded Haar multiplier, and
  $\tilde H = (\tilde h_I)_{I\in \mathcal{D}}$ be a faithful Haar system relative to the frequencies $(m_i)_{i=0}^\infty \subset \mathbb{N}$.

  Define the linear maps $B = B_{\tilde H}\colon \mathcal{V}(\delta)\to \mathcal{V}(\delta)$, and
  $A = A_{\tilde H}\colon \mathcal{V}(\delta)\to \mathcal{V}(\delta)$ by
  \begin{align*}
    Af
    &= \sum_{I\in \mathcal{D}}\frac{ h_I}{|I|} \langle\tilde  h_I, f\rangle,
      \quad\text{and}\quad
      Bf
      = \sum_{I\in \mathcal{D}} \Big\langle \frac{h_I}{|I|},f\Big\rangle \tilde h_I,
      \quad\text{for $f \in \mathcal{V}(\delta)$},
  \end{align*}
  and put $\tilde D = D|_{\tilde H} = A\circ D\circ B\colon \mathcal{V}(\delta)\to \mathcal{V}(\delta)$. Observe that $AB$ is the identity map, and
  hence $BA$ is a projection onto the linear span of $\tilde h_I$, $I\in\mathcal D$.  Actually, since $(\tilde h_I)$ is a faithful Haar system, $BA$
  is the conditional expectation with respect to the $\sigma$-algebra generated by $\tilde h_I$, $I\in \mathcal D$.
\end{ntn}

As it will be proved later (see \Cref{P:1.4.5}), for $D$ and $\tilde H$ as above, $D|_{\tilde H}$ is a diagonal operator with entries that come from
averaging entries of $D$. The theorem below is a restatement of \Cref{P:1.4.6}.
\begin{thm}\label{one-parameter algebraic stabilization theorem norm version}
  Let $D\colon \mathcal{V}(\delta)\to\mathcal{V}(\delta)$ be an $\ell^\infty$-bounded Haar multiplier.  Then, for every non-principal ultrafilter and
  $\eta>0$ there exists a faithful Haar system relative to some frequencies $(m_i)_{i=0}^\infty$ such that
  \begin{equation*}
    \Big\|D|_{\tilde H} - \lambda_\mathcal{U}(D)\Id\Big\|_{\mathrm{T}}
    < \eta
  \end{equation*}
  and $\lambda_\mathcal{U}(D) = \lambda_\mathcal{U}(D|_{\tilde H})$.
\end{thm}

The following \Cref{one-parametric assumption} and \Cref{one-parameter analytic main theorem} are not used in our paper; however, they serve as a
stepping stone to a later introduction of an analogous, less obvious assumption in two parameters
\begin{asm}\label{one-parametric assumption}
  Assume that $X$ is the completion of $\mathcal{V}(\delta)$ under a norm $\|\cdot\|$ such that the following are satisfied.
  \begin{enumerate}[label=(\roman*)]
  \item\label{one-parametric assumption a}There exists $C>0$ such that for every faithful Haar system $\tilde H$ relative to some frequencies
    $(m_i)_{i=0}^\infty$ we have
    \begin{equation*}
      \|A_{\tilde H}\|_{\mathcal{L}(X)}\|B_{\tilde  H}\|_{\mathcal{L}(X)}
      \leq C,
    \end{equation*}
    where for a linear $T\colon \mathcal{V}(\delta)\to\mathcal{V}(\delta)$, $\|T\|_{\mathcal{L}(X)} = \sup\{\|Tf\|: \|f\|\leq 1\}$.

  \item\label{one-parametric assumption b} There exist a constant $\beta>0$ such that for every Haar multiplier
    $D\colon \mathcal{V}(\delta)\to\mathcal{V}(\delta)$ we have
    \begin{equation*}
      \|D\|_{\mathcal{L}(X)}
      \leq \beta\|D\|_{\mathrm{T}}.
    \end{equation*}
  \end{enumerate}
\end{asm}

In~\cite[Lemma 4.5 and Proposition 7.1]{lechner:speckhofer:2023} the first named author and Thomas Speckhofer showed that the spaces in
$\mathcal{H}\mathcal{H}(\delta)$ satisfy \Cref{one-parametric assumption}.  A direct application of \Cref{one-parameter algebraic stabilization
  theorem norm version} yields the following \Cref{one-parameter analytic main theorem} (see also \cite[Theorem 3.6]{lechner:speckhofer:2023}).
\begin{thm}\label{one-parameter analytic main theorem}
  Let $X$ satisfy \Cref{one-parametric assumption} for some constants $\beta$, and $C$. Then, for every $\eta>0$ and a non-principal ultrafilter
  $\mathcal{U}$ on $\mathbb{N}$ every bounded Haar multiplier $D\colon X\to X$ is a $C$-projectional factor of $\lambda_\mathcal{U}(D)\Id$ with error
  $\eta$.

  If additionally $\lambda_\mathcal{U}(D)\neq 0$, then for every $C'>C$, $D$ is a $C'$-factor of $\lambda_\mathcal{U}(D)\Id$.
\end{thm}

\begin{rem}
  As previously mentioned, the a priori knowledge of the constant $\lambda_\mathcal{U}(D)$ is achieved using faithful Haar systems relative to some
  frequencies.  This can be seen as a parallel approach to the one by Enflo via Maurey in \cite{maurey:1975:1}, where they proved that every bounded
  linear operator on $L_p$, $1\leq p<\infty$, is, for every $\eta>0$, a projectional factor with error $\eta$ of a pointwise multiplier $M_g$ and $g$
  does not depend on the error $\eta$.  To do this, they considered local versions of the functional $\lambda_\mathcal{U}$, each defined for a
  measurable subset of the unit interval of positive measure, that collectively gave rise to a measure.  The derivative of this measure gave the
  pointwise multiplier in question.  In Haar system Hardy spaces bounded functions don't necessarily define bounded pointwise multipliers and, thus,
  the reduction of Haar multipliers to scalar operators is more appropriate in this setting.  Note here that pointwise multipliers preserving $H^ 1 $
  or its dual space are characterized by Stegenga \cite{stegenga:1976}, and the resulting conditions are quite restrictive.  See also Chapter VI in
  Garnett~\cite{garnett:2007}.  An explicit translation of the Stegenga condition to the martingale setting appears, for instance in the paper by
  Nakai and Sadasue \cite{nakai:sadasue:2014}.
\end{rem}

\subsubsection{Two-parametric setting}
\label{two-parametric setting}

\begin{ntn}\label{ntn:basics}
  For Haar multiplier $D$ on $\mathcal{V}(\delta_2)$ with coefficients $(d_{I,J})$ we put $\|D\|_\infty = \sup_{I,J\in \mathcal{D}} |d_{I,J}|$, and let
  $\text{\rm HM}^\infty(\delta^2)$ be the space of $\|\cdot\|_\infty$-bounded Haar multipliers on $\mathcal{V}(\delta^2)$.

  We list four special subsets of $\mathcal{D}\times\mathcal{D}$ as follows (see Figures~\ref{fig:matrix_parts-diagonal}
  to~\ref{fig:matrix_parts-upper_triangular}, below).
  \begin{itemize}
  \item The \emph{lower triangular part of $\mathcal{D}\times\mathcal{D}$} is the subset
    $\cup_{j=0}^\infty\cup_{i=j}^\infty\mathcal{D}_i\times\mathcal{D}_j$.
  \item The \emph{upper triangular part of $\mathcal{D}\times\mathcal{D}$} is the subset
    $\cup_{i=0}^\infty\cup_{j=i+1}^\infty\mathcal{D}_i\times\mathcal{D}_j$, further subdivided into
    \begin{itemize}
    \item \emph{left upper triangular part} $\cup_{i=1}^\infty\cup_{j=i+1}^\infty\mathcal{D}_i\times\{J\in\mathcal{D}_j : J\subset[0,1/2)\}$ and the
    \item \emph{right upper triangular part} $\cup_{i=1}^\infty\cup_{j=i+1}^\infty\mathcal{D}_i\times\{J\in\mathcal{D}_j : J\subset[1/2,1)\}$.
    \end{itemize}
    
  \item The \emph{diagonal part of $\mathcal{D}\times\mathcal{D}$} is the subset $\cup_{i=0}^\infty\mathcal{D}_i\times\mathcal{D}_i$.
  \item The \emph{superdiagonal part of $\mathcal{D}\times\mathcal{D}$} is the subset $\cup_{i=0}^\infty\mathcal{D}_{i}\times\mathcal{D}_{i+1}$.
  \end{itemize}
  
  \begin{minipage}{.5\linewidth}
    \begin{center}
      \includegraphics[scale=0.025]{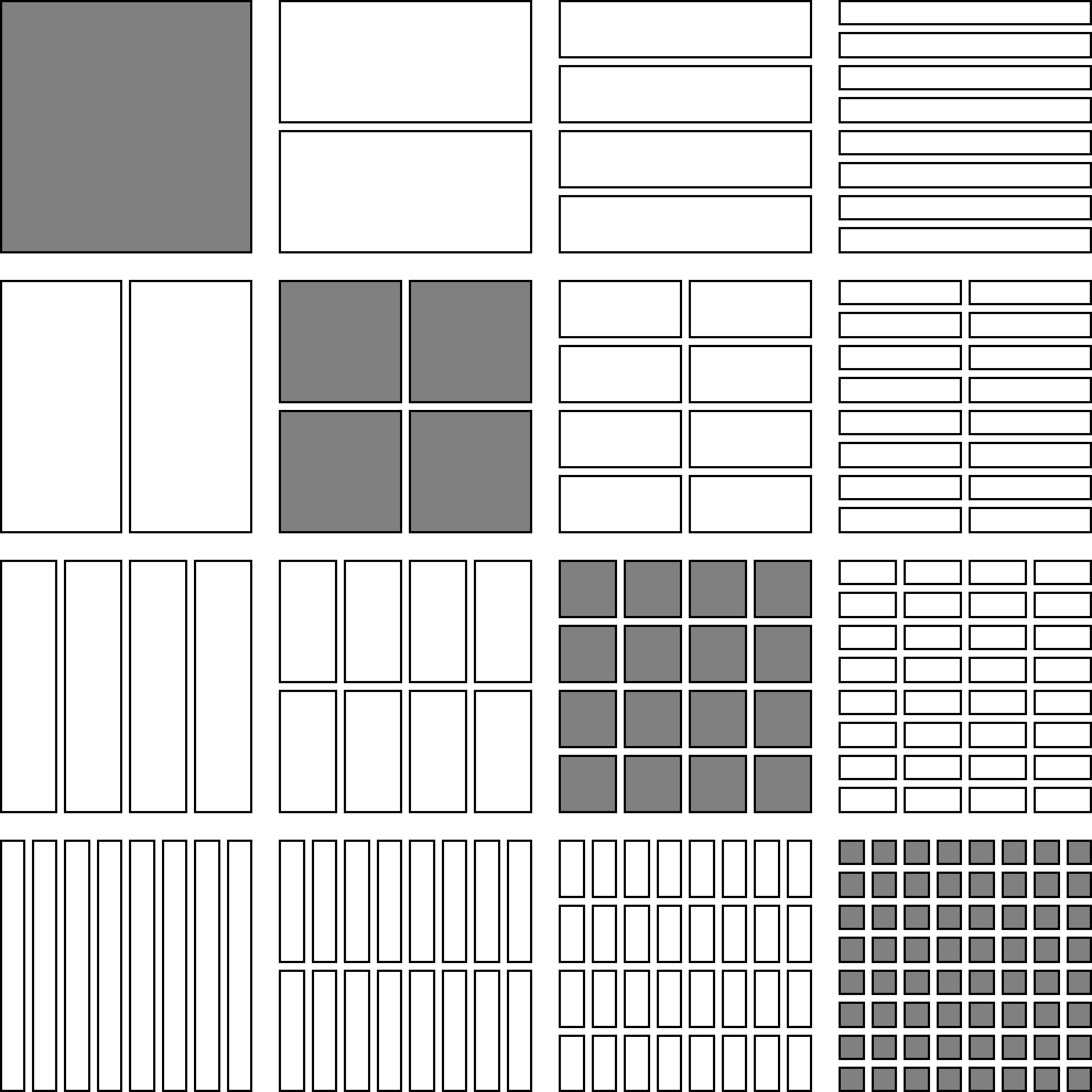}
    \end{center}
    \captionof{figure}{Diagonal, ``squares''.}
    \label{fig:matrix_parts-diagonal}

    \begin{center}
      \includegraphics[scale=0.025]{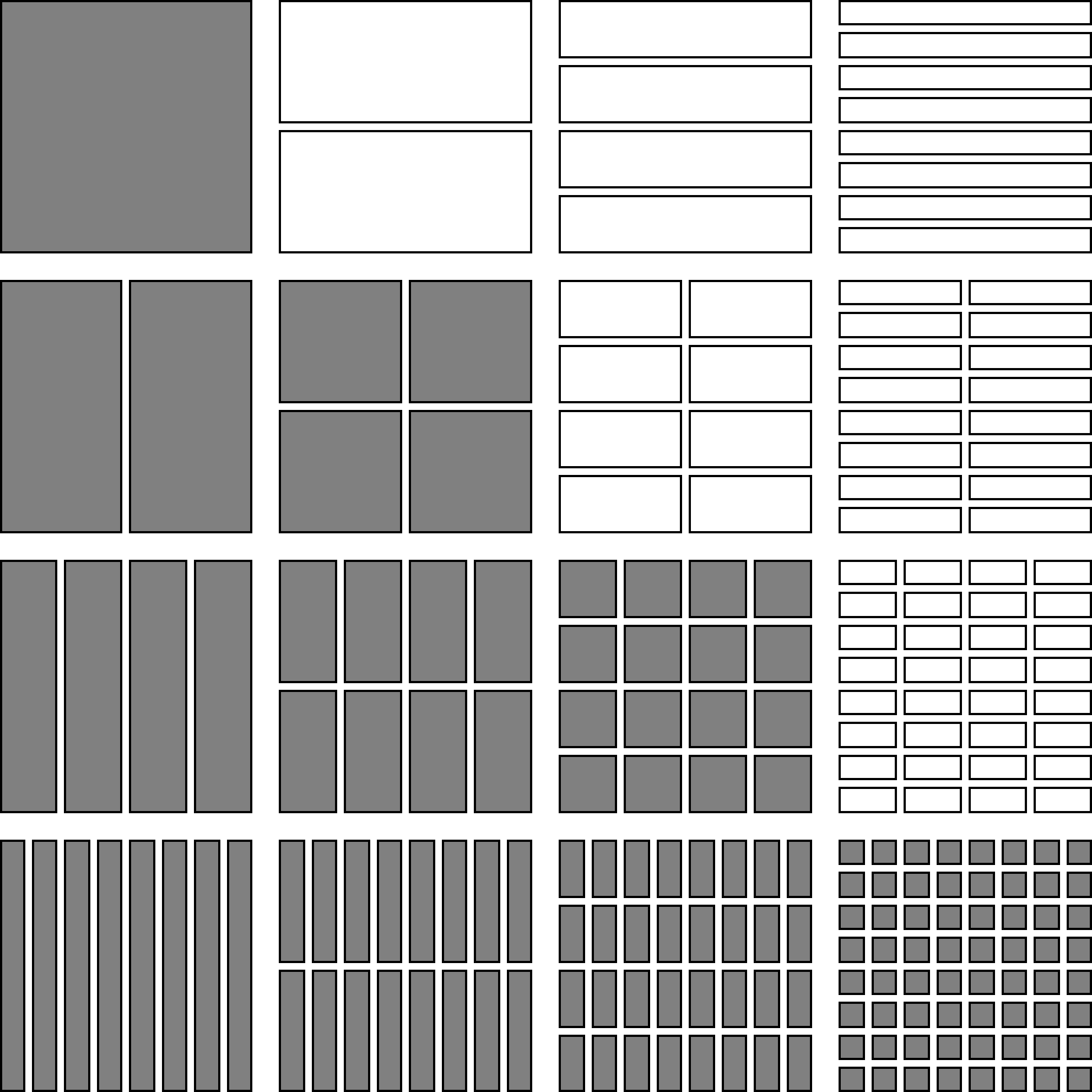}
    \end{center}
    \caption{Lower triangular, ``tall rectangles and squares''.}
    \label{fig:matrix_parts-lower_triangular}
  \end{minipage}
  \begin{minipage}{.5\linewidth}
    \begin{center}
      \includegraphics[scale=0.025]{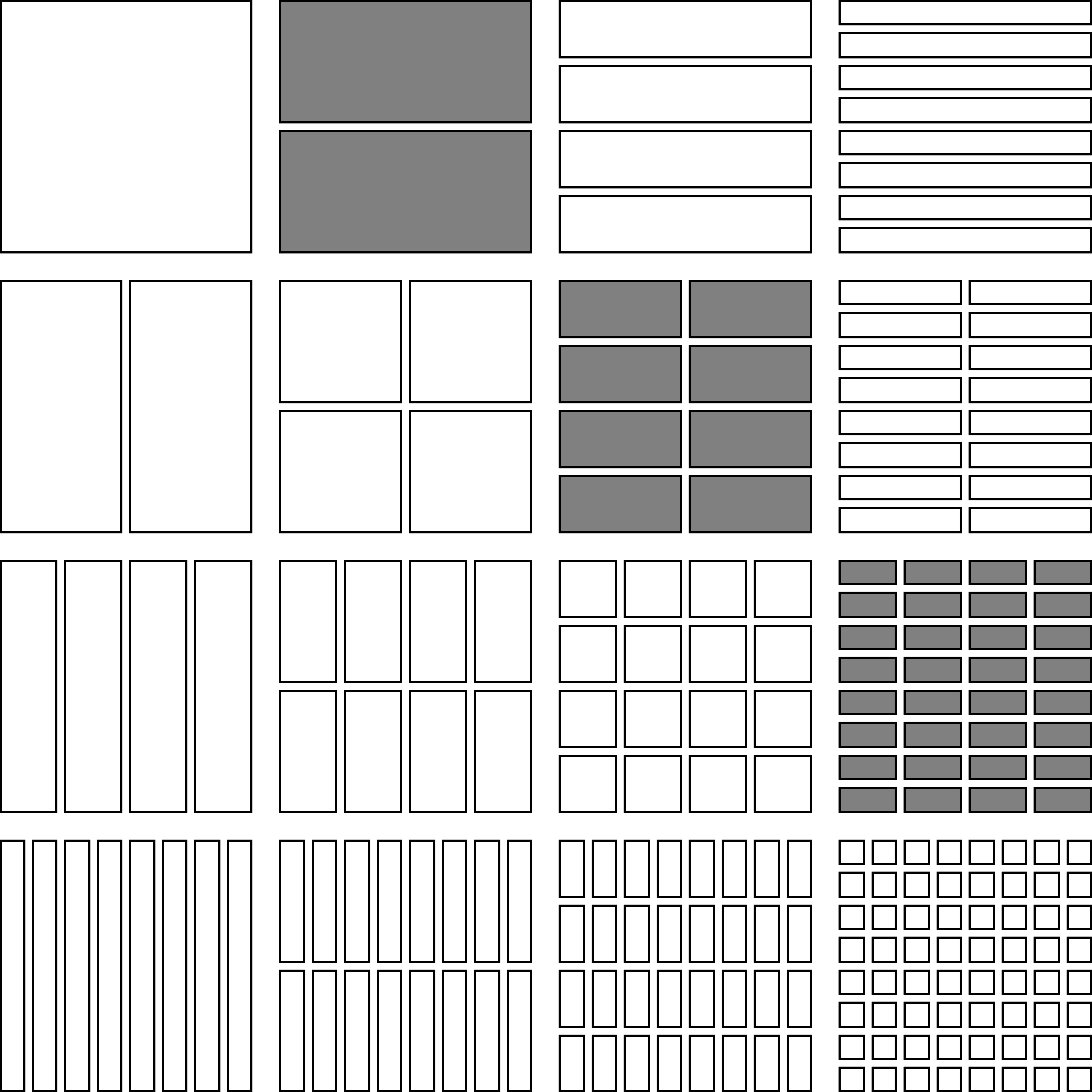}
    \end{center}
    \captionof{figure}{Superdiagonal, ``$2:1$ aspect ratio rectangles''.}
    \label{fig:matrix_parts-superdiagonal}

    \begin{center}
      \includegraphics[scale=0.025]{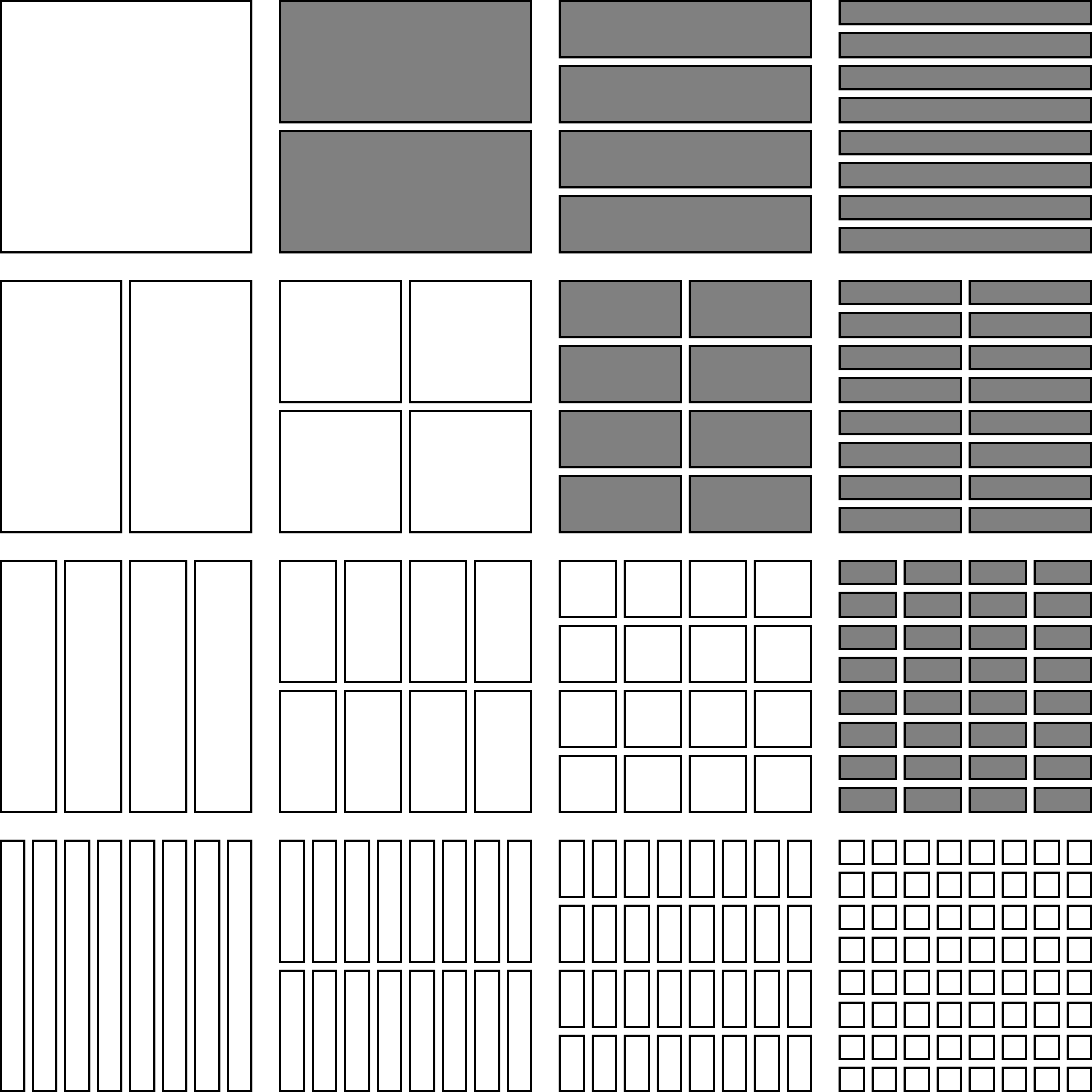}
    \end{center}
    \captionof{figure}{Upper triangular, ``wide rectangles''.}
    \label{fig:matrix_parts-upper_triangular}
  \end{minipage}

  \smallskip Based on these subsets of $\mathcal{D}\times\mathcal{D}$ we define the \emph{bi-tree variation semi-norm} on a subspace of
  $\mathrm{HM}(\delta^2)$ by letting for $D\in\mathrm{HM}(\delta^2)$
  \begin{align*}
    \|D\|_\mathrm{T^2S}
    &= \sum_{k=0}^{\infty} (k+1) \max_{\substack{I,J\in \mathcal{D}_k\\\omega,\xi\in\{\pm 1\}}} |d_{I,J} - d_{I^{\omega},J^{\xi}}|
    + \sum_{k=0}^{\infty} (k+1) \max_{\substack{I\in \mathcal{D}_k,\\J\in \mathcal{D}_{k+1}\\\omega,\xi\in\{\pm 1\}}} |d_{I,J} - d_{I^{\omega},J^{\xi}}|\\
    &\qquad + \sum_{i=0}^{\infty} \sum_{j=0}^i \max_{\substack{I\in \mathcal{D}_i\\J\in \mathcal{D}_j\\\omega\in\{\pm 1\}}} |d_{I,J} - d_{I^{\omega},J}|
    + \sum_{j=0}^{\infty} \sum_{i=0}^j \max_{\substack{I\in \mathcal{D}_i\\J\in \mathcal{D}_{j+1}\\\xi\in\{\pm 1\}}} |d_{I,J} - d_{I,J^{\xi}}|.
  \end{align*}
  Adding the roots $d_{[0,1),[0,1)}$, $d_{[0,1),[0,1/2)}$ and $d_{[0,1),[1/2,1)}$ to $\|\cdot\|_{T^2S}$ yields the \emph{bi-tree variation norm}
  $\|D\|_{\mathrm{T}^2}$ given by
  \begin{equation*}
    \|D\|_{\mathrm{T}^2}
    = \|D\|_{\mathrm{T}^2} + |d_{[0,1),[0,1)}| + |d_{[0,1),[0,1/2)}| + |d_{[0,1),[1/2,1)}|.
  \end{equation*}
  If $\|D\|_{\mathrm{T}^2S}<\infty$ we call $D$ \emph{bi-tree semi-stabilized} and we denote by $\mathrm{HM}^{\mathrm{T}^2}(\delta^2)$ the vector
  space of all bi-tree semi-stabilized Haar multipliers on $\mathcal{V}(\delta^2)$.

\item For a non-principal ultrafilter $\mathcal{U}$ let $\lambda_\mathcal{U}$, $\mu_\mathcal{U}$ denote the norm-one linear functionals on
  $\mathrm{HM}^\infty(\delta^2)$.
  \begin{align*}
    \lambda_\mathcal{U}(D)
    &= \lim_{j\to\mathcal{U}} \lim_{i\to\mathcal{U}} \langle r_i\otimes r_j, D r_i\otimes r_j \rangle = \lim_{j\to\mathcal{U}} \lim_{i\to\mathcal{U}} \sum_{I\in\mathcal{D}_i}\sum_{J\in\mathcal{D}_j}|I||J|d_{I,J}\\
    \mu_\mathcal{U}(D)
    &= \lim_{i\to\mathcal{U}} \lim_{j\to\mathcal{U}} \langle r_i\otimes r_j, D r_i\otimes r_j \rangle = \lim_{i\to\mathcal{U}} \lim_{j\to\mathcal{U}} \sum_{I\in\mathcal{D}_i}\sum_{J\in\mathcal{D}_j}|I||J|d_{I,J}.
  \end{align*}
  Note that $\lambda_\mathcal{U}(I) = \mu_\mathcal{U}(I) = 1$.
\end{ntn}

\begin{rem}\label{root to limit proximity}
  A telescoping argument (see, e.g., the proof of \Cref{pro:path-distance}) easily yields that, for $D\in\mathrm{HM}^\infty(\delta^2)$ with entries
  $(d_{I,J})_{I,J\in\mathcal{D}\times\mathcal{D}}$ and a non-principal ultrafilter $\mathcal{U}$ on $\mathbb{N}$,
  \begin{equation*}
    \Big|\lambda_\mathcal{U}(D) - d_{[0,1)[0,1)}\Big| + \Big|\mu_\mathcal{U}(D) - \frac{1}{2}\Big(d_{[0,1),[0,1/2)} + d_{[0,1),[1/2,1)}\Big)\Big|
    \leq \|D\|_{\mathrm{T}^2}.
  \end{equation*}
\end{rem}

\begin{rem}
  The kernel of $\|\cdot\|_{\mathrm{T}^2S}$ is the three-dimensional subspace of $\in\mathrm{HM}^{\mathrm{T}^2}(\delta^2)$ consisting of all $D$ with
  constant entries $\lambda$ on the lower triangular part, $\mu_1$ on the left upper triangular part, and $\mu_2$ on the right upper triangular part
  of $\mathcal{D}\times\mathcal{D}$. Although this object is implicitly approximated in our proofs, we are ultimately interested in its
  two-dimensional subspace $\{\lambda\mathcal{C}+\mu(\Id-\mathcal{C}):\lambda,\mu\in\mathbb{R}\}$, i.e., the subspace of all $D$ in
  $\mathrm{ker}(\|\cdot\|_{\mathrm{T}^2S})$ for which $\mu_1 = \mu_2 = \mu$ (see \Cref{two-parameter algebraic stabilization theorem norm version}
  below).
\end{rem}

\begin{dfn}
  Let $\tilde H=(\tilde h_I)_{I\in \mathcal{D}}$ and $ \tilde K=( \tilde k_J)_{J\in \mathcal{D}}$ be faithful Haar systems relative to the frequencies
  $(m_i)$ and $(n_j)$, respectively.  We call $\tilde H\otimes \tilde K=(\tilde h_I\otimes \tilde k_J: I,J\in\mathcal{D})$ a 2-parameter faithful Haar
  systems relative to the frequencies $(m_i)_{i=0}^\infty$ and $(n_j)_{j=0}^\infty$.
\end{dfn}

\begin{ntn}
  If $S,T\colon \mathcal{V}(\delta)\to \mathcal{V}(\delta)$ are linear operators we define
  $S\otimes T\colon \mathcal{V}(\delta^2)\to \mathcal{V}(\delta^2)$ given by
  \begin{equation*}
    S\otimes T\Big(\sum_{I,J\in\mathcal{D}} a_{I,J} h_I \otimes k_J\Big)
    = \sum_{I,J\in\mathcal{D}} a_{I,J} S(h_I)\otimes T(k_J).
  \end{equation*}
\end{ntn}

\begin{ntn}\label{bi-parameter B-Q notation}
  Let $\tilde H\otimes \tilde K= (\tilde h_I\otimes \tilde k_J: I,J\in\mathcal{D})$ be a faithful Haar system relative to the frequencies $(m_i)$ and
  $(n_j)$, and let $D\colon \mathcal{V}(\delta^2)\to \mathcal{V}(\delta^2)$ be a Haar multiplier.  Define
  \begin{align*}
    A
    &= A_{\tilde H\otimes \tilde K}
      = A_{\tilde H}\otimes A_{ \tilde K} \colon \mathcal{V}(\delta^2)\to \mathcal{V}(\delta^2),
      \quad Af = \sum_{I,J\in \mathcal{D}} \frac{h_I \otimes k_J}{|I|\cdot|J|}\langle \tilde h_I\otimes \tilde k_J , f\rangle\\
    B
    &= B_{\tilde H\otimes \tilde K}
      = B_{\tilde H}\otimes B_{ \tilde K}\colon \mathcal{V}(\delta^2)\to \mathcal{V}(\delta^2),
      \quad Bf =  \sum_{I,J\in \mathcal{D}} \Big\langle\frac{h_I\otimes k_J}{|I|\cdot |J|},f\Big\rangle \tilde h_I\otimes \tilde k_J,\\
    \tilde D
    &= D|_{\tilde H\otimes \tilde K}
      = A \circ D \circ B.
  \end{align*}
  Observe that $AB$ is the identity map.
\end{ntn}

As we will later see in \Cref{P:3.2}, $D|_{\tilde H\otimes \tilde K}$ is a Haar multiplier with entries that come from averaging those of $D$.

\begin{thm}\label{two-parameter algebraic stabilization theorem norm version}
  Let $D\colon \mathcal{V}(\delta^2)\to\mathcal{V}(\delta^2)$ be an $\ell^\infty$-bounded Haar multiplier. Then, for every non-principal ultrafilter
  $\mathcal{U}$ on $\mathbb{N}$ and $\eta>0$ there exists a bi-parameter faithful Haar system $\tilde H\otimes\tilde K$ relative to some frequencies
  $(m_i)_{i=0}^\infty$ and $(n_j)_{j=0}^\infty$ such that
  \begin{equation*}
    \Big\|
    D|_{\tilde H\otimes\tilde K} - \Big(\lambda_\mathcal{U}(D)\mathcal{C}+ \mu_\mathcal{U}(D)(I-\mathcal{C})\Big)
    \Big\|_{\mathrm{T}^2}
    < \eta
  \end{equation*}
  and $\lambda_\mathcal{U}(D) = \lambda_\mathcal{U}(D|_{\tilde H\otimes\tilde K})$,
  $\mu_\mathcal{U}(D) = \mu_\mathcal{U}(D|_{\tilde H\otimes\tilde K})$.
\end{thm}
\Cref{two-parameter algebraic stabilization theorem norm version} will be proved in \Cref{sec:semi-stabilization} where we also restate it in
\Cref{T:3.4}. But we will first state conditions under which \Cref{two-parameter algebraic stabilization theorem norm version} is directly applicable
to the completion $Z$ of $\mathcal{V}(\delta^2)$ under an appropriate norm. As we will see eventually, every bi-parameter Haar system Hardy space
satisfies these conditions.

\begin{dfn}\label{two-parametric assumption}
  Let $Z$ be the completion of $\mathcal{V}(\delta^2)$ under a norm $\|\cdot\|$ such that the following are satisfied.
  \begin{enumerate}[label=(\alph*)]
  \item\label{two-parametric assumption a} There exists $C>0$ such that for every bi-parameter Haar system $\tilde H\otimes \tilde K$ relative to some
    frequencies $(m_i)_{i=0}^\infty$ and $(n_j)_{j=0}^\infty$ we have
    \begin{equation*}
      \|A_{\tilde H\otimes\tilde K}\|_{\mathcal{L}(X)}\|B_{\tilde  H\otimes\tilde K}\|_{\mathcal{L}(X)} \leq C,
    \end{equation*}
    where for a linear $T\colon \mathcal{V}(\delta^2)\to\mathcal{V}(\delta^2)$, $\|T\|_{\mathcal{L}(X)} = \sup\{\|Tf\|: \|f\|\leq 1\}$.

  \item\label{two-parametric assumption c} There exists a constant $\beta>0$ such that for every bounded Haar multiplier $D\colon Z\to Z$ and a
    non-principal ultrafilter $\mathcal{U}$ on $\mathbb{N}$
    \begin{equation*}
      \Big\|
      D - \Big(\lambda_\mathcal{U}(D)\mathcal{C} + \mu_\mathcal{U}(D)(I-\mathcal{C})\Big)
      \Big\|_{\mathcal{L}(Z)}
      \leq \beta\|D\|_{\mathrm{T^2}}.
    \end{equation*}
  \end{enumerate}
  Then, we say that $Z$ satisfies the \emph{Capon property} or that $Z$ is a \emph{Capon space}. If, additionally, we wish to be particular about the
  constant $C$ for which $Z$ satisfies \ref{two-parametric assumption a}, then we call $Z$ a $C$-Capon space.
\end{dfn}

\begin{rem}\label{alternative c definitions}
  Let us explain the content of condition \ref{two-parametric assumption c}.  Importantly, in it, we assume that $D\colon Z\to Z$ is already bounded.
  \begin{enumerate}[label=(\arabic*)]
  \item If we assume that, on $Z$, $\mathcal{C}$ is bounded, then it is not very hard to see that \ref{two-parametric assumption c} is equivalent to
    the following simpler condition.
    \begin{enumerate}[label=(\alph*')]
      \setcounter{enumii}{2}
    \item\label{two-parametric assumption c bounded} There exists $\beta'>0$ such that for every bounded Haar multiplier $D\colon Z\to Z$,
      \begin{equation*}
        \|D\|_{\mathcal{L}(Z)} \leq \beta'\|D\|{_{\mathrm{T^2}}}.
      \end{equation*}
    \end{enumerate}

  \item If we assume that $\mathcal{C}$ is unbounded on $Z$, then \ref{two-parametric assumption c} is equivalent to the following.
    \begin{enumerate}[label=(\alph*'')]
      \setcounter{enumii}{2}
    \item\label{two-parametric assumption c anybounded} There exists $\beta''>0$ such that for every bounded Haar multiplier $D\colon Z\to Z$ and
      non-principal ultrafilter $\mathcal{U}$ on $\mathbb{N}$
      \begin{equation*}
        \|D\|_{\mathcal{L}(Z)}
        \leq \beta''\|D\|_{\mathrm{T^2}}\text{ and }\lambda_\mathcal{U}(D)
        = \mu_\mathcal{U}(D).
      \end{equation*}
    \end{enumerate}

  \end{enumerate}
\end{rem}

The following theorem is an immediate application of \Cref{two-parameter algebraic stabilization theorem norm version}, \Cref{two-parametric
  assumption}, and, when $\mathcal{C}$ is unbounded, \Cref{alternative c definitions}.
\begin{thm}\label{two-parameter analytic main theorem}
  Let $Z$ be a $C$-Capon space.
  \begin{enumerate}[label=(\roman*)]
  \item For every $\eta>0$ and non-principal ultrafilter $\mathcal{U}$ on $\mathbb{N}$ every bounded Haar multiplier $D\colon X\to X$ is a
    $C$-projectional factor of $\lambda_\mathcal{U}(D)\mathcal{C} + \mu_\mathcal{U}(D)(\Id-\mathcal{C})$ with error $\eta$.
  
  \item If, additionally, $\mathcal{C}$ is unbounded on $Z$ then $\lambda_\mathcal{U}(D) = \mu_\mathcal{U}(D)$ and, thus, for every $\eta>0$, $D$ is a
    $C$-projectional factor with error $\eta$ of $\lambda_\mathcal{U}(D)\Id$.
  \end{enumerate}
  If, in particular, $D\colon Z\to Z$ is a bounded Haar multiplier such that, for some non-principal ultrafilter $\mathcal{U}$ on $\mathbb{N}$,
  $(\lambda_\mathcal{U}(D),\mu_\mathcal{U}(D))\notin\{(1,0),(0,1)\}$ then the identity on $Z$ factors through $D$ or $\Id-D$.
\end{thm}

\subsection{Proving that Haar system Hardy spaces are Capon spaces}\label{proximity Haar to pointwise}
In \Cref{sec:haar-system-hardy-1+2}, we prove that a Haar system Hardy space $Z$ satisfies property \Cref{two-parametric
  assumption}~\ref{two-parametric assumption a} for $C=1$ (\Cref{thm:basic-operators}).  Although these are delicate proofs, they follow from a direct
analysis of the Haar system Hardy space norm.  In contrast, the proof of property~\ref{two-parametric assumption c} involves pointwise multipliers on
the space $Z^\Omega$ and their proximity to Haar multipliers.  This theory is developed in \Cref{sec:mult-capons-proj}, which also serves as a bridge
between the pointwise multiplier approach (introduced by Enflo via Maurey~\cite{maurey:1975:1}, and heavily used by
Capon~\cite{capon:1983,capon:1980:2,capon:1980:1} or Enflo and Starbird~\cite{enflo:starbird:1979}, and the Haar multiplier approach based on Gamlen
Gaudet~\cite{gamlen:gaudet:1973} and the work of Alspach, Enflo, and Odell~\cite{alspach:enflo:odell:1977} which was extensively used by Bourgain,
Rosenthal and Schechtman~\cite{bourgain:rosenthal:schechtman:1981}, Bourgain~\cite{bourgain:book:1981}, Johnson, Maurey, Schechtman and
Tzafriri~\cite{johnson:maurey:schechtman:tzafriri:1979}, Lindenstrauss and Tzafriri~\cite{lindenstrauss:tzafriri:1979:partII} and the authors of the
present paper~\cite{mueller:1988,mueller:2012, laustsen:lechner:mueller:2015, lechner:2018:2-d, lechner:motakis:mueller:schlumprecht:2019,
  lechner:motakis:mueller:schlumprecht:2020, lechner:motakis:mueller:schlumprecht:2021}.

\begin{ntn}
  To each $D\in\mathrm{HM}^{\mathrm{T^2}}(\delta^2)$ we assign sequences of functions $m_{1,k}^D,m_{2,k}^D:[0,1)^2\to\mathbb{R}$, $k\in\mathbb{N}$, as
  follows:
  \begin{equation*}
    m_{1,k}^D
    = \sum_{(I,J)\in\mathcal{D}_k\times\mathcal{D}_k} d_{I,J} \chi_{I\times J}
    \quad\text{and}\quad
    m_{2,k}^D
    = \sum_{(I,J)\in\mathcal{D}_k\times\mathcal{D}_{k+1}} d_{I,J} \chi_{I\times J}.
  \end{equation*}
\end{ntn}

\begin{rem}
  Clearly, for each $k\in\mathbb{N}$, $m_{1,k}^D,m_{2,k}^D$ satisfy the assumption of \Cref{rem:pw-multiplier}. Therefore, they give rise to pointwise
  multipliers $M_{1,k}^D,M_{2,k}^D\colon Z^\Omega\to Z^\Omega$ defined as
  \begin{equation*}
    (M_{i,k}^D w)(s,t,\sigma)
    = m_{i,k}^D(s,t) w(s,t,\sigma),
    \qquad s,t\in [0,1),\;\sigma\in\boldsymbol{\sigma},
  \end{equation*}
  for $i=1,2$ and $k\in\mathbb{N}$, and hence, $\|M_{i,k}^D\|\leq \|m_{i,k}^D\|_{L^{\infty}} \leq \|D\|_{\infty}\leq\|D\|_{\mathrm{T^2}}$, $i=1,2$.
\end{rem}

\begin{rem}
  For each $D\in\mathrm{HM}^{\mathrm{T^2}}(\delta^2)$ and $i=1,2$, by the variational nature of $\|\cdot\|_{\mathrm{T^2}}$, a telescoping argument
  (see, e.g., the proof of \Cref{pro:path-distance}) yields
  \begin{equation*}
    \lim_{n\to\infty}\sup_{n\leq k,\ell} |m_{i,k}^D(s,t) - m_{i,\ell}^D(s,t)|
    = 0,
  \end{equation*}
  uniformly in $s,t\in[0,1)$. Define $m_1^D,m_2^D\in L^\infty[0,1)^2$ with
  \begin{equation}\label{eq:101}
    m_1^D(s,t) = \lim_{k\to \infty} d_{I_k(s),J_k(t)}
    \qquad\text{and}\qquad
    m_2^D(s,t) = \lim_{k\to \infty} d_{I_k(s),J_{k+1}(t)},
  \end{equation}
  where, for $k\in\mathbb{N}$, $I_k(s)$ and $J_k(t)$ are the unique dyadic intervals in $\mathcal{D}_k$ such that $s\in I_k(s)$ and $t\in
  J_k(t)$. Then $\lim_k\|m_{1,k}^D-m_{1}^D\|_{L^\infty} = 0$ and $\lim_k\|m_{2,k}^D-m_{2}^D\|_{L^\infty} = 0$ and, in particular, the sequences
  $(M^D_{1,k})_{k=1}^\infty$ $(M^D_{2,k})_{k=1}^\infty$ converge in the operator norm of $\mathcal{L}(Z^\Omega)$.
  
  We conclude that $D\in\mathrm{HM}^{\mathrm{T^2}}(\delta^2)$ gives rise, for $i=1,2$, to the pointwise multipliers
  $M_i^D\colon Z^{\Omega}\to Z^\Omega$ defined as
  \begin{equation}\label{eq:172}
    (M_i^D w)(s,t,\sigma)
    = m_i^D(s,t) w(s,t,\sigma),
    \qquad s,t\in [0,1),\;\sigma\in\boldsymbol{\sigma}.
  \end{equation}
  such that $\|M_i^D\| \leq \|m_i^D\|_{L^{\infty}} \leq \|D\|_{\infty}\leq\|D\|_{\mathrm{T^2}}$.
\end{rem}

\begin{rem}\label{rem:cvs}
  For $D\in\mathrm{HM}^{\mathrm{T^2}}(\delta^2)$, the definition of $\|\cdot\|_{\mathrm{T^2}\mathrm{S}}$ easily yields
  \begin{equation}\label{eq:7}
    |d_{[0,1),[0,1)} - m_1^D(s,t)|
    \leq \|D\|_{\mathrm{T^2}\mathrm{S}}
    \quad\text{ and}\quad
    |d_{[0,1),J_1(t)} - m_2^D(s,t)|
    \leq \|D\|_{\mathrm{T^2}\mathrm{S}}.
  \end{equation}
\end{rem}

\begin{rem}
  We don't use the following properties but mention them for context.
  \begin{enumerate}[label=(\roman*)]
  \item For any Haar system Hardy space $Z$ and $i=1,2$, the assignment $D\mapsto M_i^D$ is a norm-one linear operator from
    $\mathrm{HM}^{\mathrm{T^2}}(\delta^2)$ to $\mathcal{L}(Z^\Omega)$.
    
  \item For $D\in\mathrm{HM}^{\mathrm{T^2}}(\delta^2)$, the functions $m_1^D, m_2^D$ are continuous on the set $([0,1)\setminus\mathbb{D})^2$, where
    $\mathbb{D}$ denotes the dyadic rationals in $[0,1]$.
  \end{enumerate}
\end{rem}

Recall, for a Haar system Hardy space $Z$ there is a linear isometry $\mathcal{O}\colon Z\to Z^\Omega$ given in \Cref{dfn:multiplier-space} and, for
$D\in\mathrm{HM}^{\mathrm{T^2}}(\delta^2)$, we can define the linear operators
$\mathcal{O}D,M^D_1\mathcal{O},M^D_2\mathcal{O}:\mathcal{V}(\delta^2)\to Z^\Omega$.  Denote $\mathcal{V}_Z = (\mathcal{V}(\delta^2),\|\cdot\|_Z)$.

The following is included in the first main result of \Cref{sec:mult-capons-proj}, \Cref{thm:pw-multipliers:1} and \Cref{cor:pw-multipliers:2}.
\begin{thm}\label{pointwise approximation theorem}
  Let $Z = Z(\boldsymbol{\sigma},X,Y)$ be a Haar system Hardy space. For every $\ell^\infty$-bounded Haar multiplier
  $D\colon \mathcal{V}(\delta^2) \to \mathcal{V}(\delta^2)$ the following estimates hold.
  \begin{enumerate}[label=(\roman*)]
  \item\label{pointwise approximation theorem i}
    $\|(\mathcal{O}D-M_1^D\mathcal{O})\mathcal{C}\|_{\mathcal{L}(\mathcal{V}_Z,Z^\Omega)} \leq 4\|D\|_{\mathrm{T^2}\mathrm{S}}$ and
    $\|(\mathcal{O}D-M_2^D\mathcal{O})(I-\mathcal{C})\|_{\mathcal{L}(\mathcal{V}_Z,Z^\Omega)} \leq 4 \|D\|_{\mathrm{T^2}\mathrm{S}}$.

  \item\label{pointwise approximation theorem ii}
    $\|(M_1^D - M_2^D)\mathcal{O}\mathcal{C}\|_{\mathcal{L}(\mathcal{V}_Z,Z^\Omega)} \leq \|D\|_{\mathcal{L}(Z)} + \|D\|_\infty + 8
    \|D\|_{\mathrm{T^2}\mathrm{S}}$.
  \end{enumerate}
  Furthermore, $\mathcal{C}$ is bounded on $Z$ if and only if there exist a non-principal ultrafilter $\mathcal{U}$ on $\mathbb{N}$ and a bounded Haar
  multiplier $D\colon Z\to Z$ such that $\lambda_\mathcal{U}(D)\neq\mu_\mathcal{U}(D)$.
\end{thm}

For the space $Z$ on which the Capon projection is unbounded, we require an additional statement that may be of independent interest. It is proof is
given in \Cref{thm:pw-multipliers:2}
\begin{thm}\label{pointwise collapse theorem}
  Let $Z = Z(\boldsymbol{\sigma},X,Y)$ be a Haar system Hardy space on which the Capon projection $\mathcal{C}$ is unbounded. If
  $m\in L^\infty[0,1)^2$ is such that, for the induced pointwise multiplier $M$, $M\mathcal{OC}:(\mathcal{V}(\delta^2),\|\cdot\|_Z)\to Z^\Omega$ is
  bounded, then $M = 0$.
\end{thm}

With the above two theorems at hand, we easily verify property~\ref{two-parametric assumption c} of \Cref{two-parametric assumption} for Haar system
Hardy spaces.

\begin{thm}\label{10-10-5}
  Every Haar system Hardy space $Z = Z(\boldsymbol{\sigma},X,Y)$ is a 1-Capon space.
\end{thm}

\begin{proof}
  For $Z$ with bounded Capon operator property~\ref{two-parametric assumption c} of \Cref{two-parametric assumption} follows almost
  immediately. Indeed, for a bounded Haar multiplier $D\colon Z\to Z$, we verify property~\ref{two-parametric assumption c bounded} of
  \Cref{alternative c definitions}.
  \begin{align*}
    \|D\|_{\mathcal{L}(Z)}
    &\leq \|D\mathcal{C}\|_{\mathcal{L}(Z)} + \|D(I-\mathcal{C})\|_{\mathcal{L}(Z)}
      \leq 8\|D\|_{\mathrm{T^2}} + (\|m_1^D\|_{L^\infty}+\|m_2^D\|_{L^\infty})\cdot \|\mathcal{C}\|_{\mathcal{L}(Z)}\\
    &\leq (8+\|\mathcal{C}\|_{\mathcal{L}(Z)})\cdot \|D\|_{\mathrm{T^2}}.
  \end{align*}

  Next assume that the Capon operator is unbounded on $Z$.  We proceed to verify property~\ref{two-parametric assumption c anybounded} of
  \Cref{alternative c definitions}.  For a bounded Haar multiplier $D\colon Z\to Z$, by \Cref{pointwise approximation theorem}, we have
  $\lambda_\mathcal{U}(D) = \mu_\mathcal{U}(D)$. To prove the remainder of~\ref{two-parametric assumption c anybounded}, we assume that
  $\|D\|_{\mathrm{T^2}} < \infty$.  By \Cref{pointwise approximation theorem}~\ref{pointwise approximation theorem ii} and \Cref{pointwise collapse
    theorem}, $M_1^D = M_2^D$ and, thus, by \Cref{pointwise approximation theorem}~\ref{pointwise approximation theorem i},
  $\|\mathcal{O}D - M_1^D\mathcal{O}\|_{\mathcal{V}_Z,Z^\Omega} \leq 8\|D\|_{\mathrm{T^2}\mathrm{S}}$.  By \Cref{root to limit proximity} and
  \Cref{rem:cvs} we have $\|m_1^D - \lambda_\mathcal{U}(D)\|_{L^\infty} \leq 2\|D\|_{\mathrm{T^2}\mathrm{S}}$.  Therefore,
  $\|D\|_{\mathcal{L}(Z)} \leq |\lambda_\mathcal{U}(D)| + 10 \|D\|_{\mathrm{T^2}\mathrm{S}}\leq 11 \|D\|_{\mathrm{T}^2}$.
\end{proof}

The following theorem exhibits an important connection between the pointwise multiplier function $m_1^D$ and $\lambda_{\mathcal{U}}(D)$ respectively
$m_2^D$ and $\mu_{\mathcal{U}}(D)$. See also \Cref{cor:pw-multipliers:2}.
\begin{thm}\label{thm:pointwise-int}
  Let $D\in\mathrm{HM}^{\mathrm{T^2}}(\delta^2)$.  Then the functions $m_1^D, m_2^D\colon [0,1)^2\to \mathbb{R}$ given by~\eqref{eq:101} satisfy, for
  any non-principal ultrafilter $\mathcal{U}$ on $\mathbb{N}$,
  \begin{equation*}
    \lambda_{\mathcal{U}}(D)
    = \int_{0}^{1} \int_{0}^{1} m_1^D(s,t) \mathrm{d} t \mathrm{d} s
    \qquad\text{and}\qquad
    \mu_{\mathcal{U}}(D)
    = \int_{0}^{1} \int_{0}^{1} m_2^D(s,t) \mathrm{d} t \mathrm{d} s.
  \end{equation*}
\end{thm}

\begin{proof}
  We will only verify the identities for $\lambda_{\mathcal{U}}(D)$, the other identity follow by similar arguments.  First, recall that
  $m_1^D(s,t) = \lim_k d_{I_k(s),J_k(t)}$ and $|m_1^D(s,t)|\leq \|D\|_{\infty} < \infty$ for all $s,t\in [0,1)$.  Next, let $j\leq i$,
  $J\in \mathcal{D}_j$ and $I,L\in \mathcal{D}_i$ be fixed, let $K_k,L_k\in \mathcal{D}_k$ be such that $K_j\supset K_{j+1}\supset \dots K_i = I$ and
  $J = L_j\supset L_{j+1}\supset \dots L_i = L$.  Thus,
  \begin{align*}
    |d_{I,J} - d_{K_j,J}|
    &\leq \sum_{k=j}^{i-1} |d_{K_{k+1},J} - d_{K_k,L}|
      \leq \sum_{k=j}^{\infty} \max_{\substack{K\in \mathcal{D}_k\\M\in \mathcal{D}_j\\\omega\in \{\pm 1\}}} |d_{K^{\omega},M} - d_{K,M}|,\\
    |d_{K_j,J} - d_{I,L}|
    &\leq \sum_{k=j}^{i-1} |d_{K_k,L_k} - d_{K_{k+1},L_{k+1}}|
      \leq \sum_{k=j}^{\infty} \max_{\substack{K,M\in \mathcal{D}_k\\\omega,\xi\in \{\pm 1\}}} |d_{K,M} - d_{K^{\omega},M^{\xi}}|.
  \end{align*}
  Adding the above inequalities yields
  \begin{equation*}
    |d_{I,J} - d_{I,L}|
    \leq \sum_{k=j}^{\infty} \Bigl(
    \max_{\substack{K\in \mathcal{D}_k\\M\in \mathcal{D}_j\\\omega\in \{\pm 1\}}} |d_{K^{\omega},M} - d_{K,M}|
    + \max_{\substack{K,M\in \mathcal{D}_k\\\omega,\xi\in \{\pm 1\}}} |d_{K,M} - d_{K^{\omega},M^{\xi}}|
    \Bigr).
  \end{equation*}
  We use the latter estimate and obtain
  \begin{align*}
    &\Bigl|
      \sum_{\substack{I\in \mathcal{D}_i\\J\in \mathcal{D}_j}} |I| |J| d_{I,J}
    - \sum_{I,L\in \mathcal{D}_i} |I| |L| d_{I,L}
    \Bigr|
    = \Bigl|
    \sum_{I\in \mathcal{D}_i} |I| \Bigl(
    \sum_{J\in \mathcal{D}_j} |J| d_{I,J}
    - \sum_{L\in \mathcal{D}_i} |L| d_{I,L}
    \Bigr)
    \Bigr|\\
    &\qquad = \Bigl|
      \sum_{I\in \mathcal{D}_i} |I| \Bigl(
      \sum_{J\in \mathcal{D}_j} \sum_{\substack{L\in \mathcal{D}_i\\L\subset J}} |L| (d_{I,J} -  d_{I,L})
    \Bigr)
    \Bigr|
    \leq \sum_{I\in \mathcal{D}_i} |I|
    \sum_{J\in \mathcal{D}_j} \sum_{\substack{L\in \mathcal{D}_i\\L\subset J}} |L| |d_{I,J} -  d_{I,L}|\\
    &\qquad \leq \sum_{I\in \mathcal{D}_i} |I|
      \sum_{J\in \mathcal{D}_j} \sum_{\substack{L\in \mathcal{D}_i\\L\subset J}} |L|
    \sum_{k=j}^{\infty} \Bigl(
    \max_{\substack{K\in \mathcal{D}_k\\M\in \mathcal{D}_j\\\omega\in \{\pm 1\}}} |d_{K^{\omega},M} - d_{K,M}|
    + \max_{\substack{K,M\in \mathcal{D}_k\\\omega,\xi\in \{\pm 1\}}} |d_{K,M} - d_{K^{\omega},M^{\xi}}|
    \Bigr)\\
    &\qquad = \sum_{k=j}^{\infty} \Bigl(
      \max_{\substack{K\in \mathcal{D}_k\\M\in \mathcal{D}_j\\\omega\in \{\pm 1\}}} |d_{K^{\omega},M} - d_{K,M}|
    + \max_{\substack{K,M\in \mathcal{D}_k\\\omega,\xi\in \{\pm 1\}}} |d_{K,M} - d_{K^{\omega},M^{\xi}}|
    \Bigr).
  \end{align*}
  Since $\|D\|_{\mathrm{T^2}\mathrm{S}} < \infty$, the latter expression tends to $0$ if $j\to \infty$, thus,
  \begin{align*}
    0
    &= \lim_{j\to \mathcal{U}} \lim_{i\to \mathcal{U}} \Bigl|
      \sum_{\substack{I\in \mathcal{D}_i\\J\in \mathcal{D}_j}} |I| |J| d_{I,J}
    - \sum_{I,L\in \mathcal{D}_i} |I| |L| d_{I,L}
    \Bigr|
    = \Bigl| \lambda_{\mathcal{U}}(D) - \lim_{i\to \mathcal{U}} \sum_{I,L\in \mathcal{D}_i} |I| |L| d_{I,L}
    \Bigr|\\
    &= \Bigl| \lambda_{\mathcal{U}}(D) - \lim_{k\to \infty} \int_{0}^{1} \int_{0}^{1} d_{I_k(s),J_k(t)} \mathrm{d} t \mathrm{d} s
      \Bigr|
      = \Bigl| \lambda_{\mathcal{U}}(D) -  \int_{0}^{1} \int_{0}^{1} m_1^D(s,t) \mathrm{d} t \mathrm{d} s
      \Bigr|
  \end{align*}
  as claimed.
\end{proof}

\subsection{Conclusions and open problems}
\Cref{thm:hsh-factor}, our main theorem on the class of Haar multipliers $\mathrm{HM}(Z)$ acting on a Haar system Hardy space
$Z\in\mathcal{HH}(\delta^2)$, gives the following factorization results:
\begin{enumerate}[label=(\roman*)]
\item If Capon's projection $\mathcal{C}$ is unbounded on $Z$ then $\mathrm{HM}(Z)$ approximately 1-projectionally reduces to the class of scalar
  operators, and hence, $\mathrm{HM}(Z)$ has the $C$-primary factorization property, for every $C>2$.

\item If $\mathcal{C}$ is bounded on $Z$ then $\mathrm{HM}(Z)$ approximately 1-projectionally reduces to the two-dimensional subalgebra
  $\{\lambda\mathcal{C}+\mu(\Id-\mathcal{C}):\lambda,\mu\in\mathbb{R}\}$.
\end{enumerate}
We now list questions and problems arising in connection with our current work.
\begin{question}
  Let $Z = Z(\boldsymbol{\sigma},X,Y)$ be a Haar system Hardy space.
  \begin{enumerate}
  \item Under which conditions on $Z$ is it true that $Z$ has the $C$-primary factorization property, for some $C < \infty$?  More specifically, when
    does the class of operators $\mathcal{L} (Z)$ approximately projectionally reduce to the class of scalar multiples of the identity?
             
  \item In view of our main result, \Cref{thm:hsh-factor}, we ask under which conditions on $Z$ is it true that every bounded linear operator on $Z$
    approximately projectionally reduces to a Haar multiplier in $Z ?$
  \end{enumerate}
\end{question}
If a space has the primary factorization property, and if Pe{\l}czy{\'n}ski's decomposition method is applicable, then the space $Z$ is primary. This
leads to our next set of questions.
\begin{question}
  Let $Z = Z(\boldsymbol{\sigma},X,Y)$ be a Haar system Hardy space.
  \begin{enumerate}
  \item Under which conditions is $Z$ and/or its dual space primary?
  \item Specifically, find a way to make Pe{\l}czy{\'n}ski's decomposition work in the context of Haar system Hardy spaces. Find a suitable substitute
    for the accordion property.
  \item As a special case of the above problem: Find all primary spaces in the family
    \begin{equation*}
      \mathcal{F}
      = \{ L^p(L^q),\ H^p (H^q),\ L^p(H^q),\ H^p(L^q) : 1 \le p,q < \infty \}. 
    \end{equation*}
  \end{enumerate}
\end{question}

Capon's projection plays a pivotal role in our factorization theorems. We suggest to investigate it more closely.  With the following problems we
intend to find intrinsic, and easily verifiable conditions for its boundedness on bi-parameter Haar system Hardy spaces.
\begin{question}
  Let $Z = Z(\boldsymbol{\sigma},X,Y)$ be a Haar system Hardy space.
  \begin{enumerate}
  \item Is it true that $\mathcal{C}$ is bounded on $Z$ if and only if the bi-parameter Haar system $(h_I\otimes h_J)_{I,J\in\mathcal{D}}$ is
    unconditional in $Z$?
  \item Is it true that $\mathcal{C}$ is bounded on $Z$ if and only if $Z$ has an unconditional basis?
  \item Is it true that $Z$ has an unconditional basis if and only if the bi-parameter Haar system is unconditional in $Z$?
  \end{enumerate}
\end{question}
Fix $Z = L^p(L^q) $ for $1< p,q < \infty$.  In 1982 Capon~\cite{capon:1982:2} proved that the two dimensional algebra of bounded operators
$\{\lambda\mathcal{C}+ \mu (\Id -\mathcal{C}) : \lambda,\mu\in\mathbb{R}\}$ acting on $Z$ approximately projectionally reduces to
$\{\mu \Id_Z : \mu \in \mathbb{R}\}$.  In 1994, the third named author~\cite{mueller:1994} observed that Capon's method applies to the space
$Z = H^1(H^1)$.  This gives rise to the next question:
\begin{question}
  Assume that the Capon projection $\mathcal{C}$ is bounded on a Haar system Hardy space $Z = Z(\boldsymbol{\sigma},X,Y)$.
  \begin{enumerate}
  \item Does the two dimensional sub-algebra $\{\lambda\mathcal{C}+ \mu(\Id -\mathcal{C}) : \lambda,\mu \in \mathbb{R}\}$ (of operators on $Z$)
    approximately projectionally reduce to $\{\mu \Id_Z : \mu \in \mathbb{R}\}$?
  \item What about the special cases when $Z = L^p(H^1)$ or $Z = H^1(L^p )$, $1 < p < \infty $?
  \end{enumerate}
\end{question}

We next point to several papers in which related questions are addressed.

In 1974, Casazza and Lin \cite{casazza:lin:1974} proved: If a Banach space $E$ has a sub-symmetric basis $(e_j)_{j=1}^{\infty}$, then $E$ satisfies
the primary factorization property. This led them to ask if every Banach space $E$ with a symmetric basis is primary.  To this date, the problem is
open.  Let $(e_j^{*})_{j=1}^{\infty}$ denote the biorthogonal functionals of the sub-symmetric basis $(e_j)_{j=1}^{\infty}$.  Assuming that
$(e_j^{*})_{j=1}^{\infty}$ is ``non-$\ell^1$-splicing'', the first named author showed that $E^{*}$ satisfies the primary factorization property
(see~\cite{lechner:2019:subsymm}).  In view of this result, we ask:
\begin{question}
  Is $E^{*}$ a primary space, whenever the biorthogonal functionals $(e_j^{*})_{j=1}^{\infty}$ are ``non-$\ell^1$-splicing''?
\end{question}

In 1977 Casazza~\cite{casazza:1977} obtained the, by now classical, result that the James space $J$ is primary. We point out that Casazza's
proof~\cite{casazza:1977} circumvented the accordion property.
 
The James space $J$ is the earliest example of a space admitting a conditional spreading basis.  It is not true that every space with such a basis is
primary (see, e.g., \cite[page 1235]{argyros:motakis:sari:2017}).  The subclass of equal signs additive (ESA) bases was introduced by Brunel and
Sucheston in 1976 (\cite{brunel:sucheston:1976}), and its isomorphically invariant version of convex block homogeneous bases was introduced and
studied in \cite{argyros:motakis:sari:2017}.  There it was shown that every space with a conditional spreading basis is isomorphic to the direct sum
of a space with an ESA basis and a space with an unconditional finite dimensional Schauder decomposition (UFDD).  Spaces $Z$ with an ESA basis are
natural candidates for being primary and the authors of \cite{argyros:motakis:sari:2017} attempted to prove this.  Arguments from
\cite{argyros:motakis:sari:2017} yield that such $Z$ has the primary factorization property.  As usual, attaining primariness stumbles on the lack of
an accordion-type property for certain complemented subspaces with a UFDD with a type of distributional subsymmetry.

As part of the pursuit of new methods to show primariness, we reiterate here a problem from~\cite[page 1235]{argyros:motakis:sari:2017}:
\begin{question}
  Is every Banach space with an ESA basis primary?
\end{question}

\section{Coefficient Stabilization}
\label{sec:semi-stabilization}
In this section, we prove \Cref{two-parameter algebraic stabilization theorem norm version}, i.e., that every $\ell^\infty$-bounded Haar multiplier
$D\colon \mathcal{V}(\delta^2) \to \mathcal{V}(\delta^2)$ can be reduced to one with bounded bi-tree variation norm that is, in fact, close to a
linear combination of $\mathcal{C}$ and $\Id -\mathcal{C}$.

\subsection{One-parameter faithful Haar systems}
We first develop the necessary algebraic and probabilistic language in the one-parameter setting.  Thereafter we derive its two-parametric
counterpart.
\begin{pro}\label{P:1.4.5}
  Let $D\colon \mathcal{V}(\delta) \to \mathcal{V}(\delta)$ be an $\ell^{\infty}$-bounded diagonal operator,
  $\tilde H = (\tilde h_I)_{I\in\mathcal{D}}$ be a faithful Haar system relative to the frequencies $(m_i)_{i=0}^\infty$, and let
  $\tilde D = D|_{\tilde H}\colon \mathcal{V}(\delta) \to \mathcal{V}(\delta)$ be defined as in \Cref{one-parameter AB notation}.  Then $\tilde D$ is
  a diagonal operator whose coefficients are
  \begin{align}\label{E:1.4.5.1a}
    \tilde d_I
    &= \sum_{K\in \mathcal{D}_{m_i}\atop K\subset \supp(\tilde h_I)} \frac{|K|}{|I|} d_K,
      \qquad\text{$i\geq 0$, $I\in\mathcal{D}_i$.}
  \end{align}
\end{pro}

\begin{proof}
  Suppose that $i\geq 0$ for $I\in \mathcal{D}_i$, the collection $\mathcal{A}_I$ is given by
  $\mathcal{A}_I = \{ K\in \mathcal{D}_{m_i} : K\subset \supp(\tilde h_I)\}$, and let $(\theta_K)\in \{\pm 1\}^{\mathcal{D}}$.  We then define
  \begin{equation*}
    \tilde h_I
    = \sum_{K\in \mathcal{A}_I} \theta_{K} h_K.
  \end{equation*}
  It follows for $I\neq J$ that
  \begin{equation*}
    \langle \tilde h_I, D(\tilde h_J)\rangle
    = \Big\langle \sum_{K\in \mathcal{A}_I} \theta_{K} h_K, \sum_{K'\in \mathcal{A}_J} \theta_{K'} d_{K'} h_{K'} \Big\rangle
    = 0,
  \end{equation*}
  and
  \begin{equation*}
    \langle \tilde h_I, D(\tilde h_I)\rangle
    = \Big\langle \sum_{K\in \mathcal{A}_I} \theta_{K} h_K , \sum_{K'\in \mathcal{A}_I} \theta_{K'} d_{K'} h_{K'} \Big\rangle
    = \sum_{K\in \mathcal{A}_I } d_K |K|.
  \end{equation*}
  Thus
  \begin{align*}
    \tilde D(h_I)
    &= (A D)(\tilde h_I)
      = \sum_{J\in\mathcal{D}} \frac{h_J}{|J|} \langle \tilde h_J, D\tilde h_I \rangle
      = \frac{h_I}{|I|} \langle \tilde h_I, D\tilde h_I \rangle\\
    &= \frac{h_I}{|I|} \Big\langle \sum_{K\in \mathcal{A}_I} \theta_{K} h_K, \sum_{K'\in \mathcal{A}_I} \theta_{K'} d_{K'} h_{K'} \Big\rangle
      = h_I \sum_{K\in\mathcal{D}_{m_i},\atop K\subset\supp(\tilde h_I)} \frac{|K|}{|I|} d_K.
  \end{align*}
\end{proof}

\begin{ntn}
  Let $\widetilde H^{(1)} = (\tilde h_I^{(1)})_{I\in\mathcal{D}}$ be a faithful Haar system relative to frequencies $(n_i)_{i=0}^\infty$ and let
  $\widetilde H^{(2)} = (\tilde h^{(2)}_I)_{I\in\mathcal{D}}$ be a faithful Haar system relative to frequencies $(m_i)_{i=0}^\infty$.  For each
  $I\in\mathcal{D}$, define $\tilde h^{(3)}_I = \sum_{J\in\mathcal{D}} |J|^{-1}\langle h_J, \tilde h^{(2)}_I \rangle \tilde h^{(1)}_J$ and denote
  $\widetilde H^{(2)} \ast \widetilde H^{(1)} = (h^{(3)}_I)_{I\in\mathcal{D}}$.
\end{ntn}

\begin{pro}\label{blocking faithful haar}
  Let $\widetilde H^{(1)} = (\tilde h_I^{(1)})_{I\in\mathcal{D}}$ be a faithful Haar system relative to frequencies $(n_i)_{i=0}^\infty$ and let
  $\widetilde H^{(2)} = (\tilde h^{(2)}_I)_{I\in\mathcal{D}}$ be a faithful Haar system relative to frequencies $(m_i)_{i=0}^\infty$.  Then,
  $\widetilde H^{(2)}\ast \widetilde H^{(1)}$ is a faithful Haar system relative to the frequencies $(n_{m_i})_{i=0}^\infty$.
\end{pro}

\begin{proof}
  We will use \Cref{faithful short formulation}.  For $i\in\mathbb{N}_0$ and $I\in\mathcal{D}_i$ write
  \begin{align*}
    \tilde h^{(3)}_I
    = \sum_{J\in\mathcal{D}_{m_i}\atop J\subset\supp(\tilde h^{(2)}_I)}
    \sum_{L\in\mathcal{D}_{n_{m_i}}\atop L\subset\supp(\tilde h^{(1)}_J)}
    \theta^{(2)}_{J} \theta^{(1)}_{L} h_L
  \end{align*}
  This yields that $h^{(3)}_I$ is a linear combination of $(h_J)_{J\in\mathcal{D}_{n_{m_i}}}$ with coefficients in $\{-1,0,1\}$.  By taking $i=0$ we
  also easily obtain $\supp(\tilde h^{(3)}_{[0,1)}) = [0,1)$.

  Next, consider the maps $B_1, B_2\colon \mathcal{V}(\delta) \to \mathcal{V}(\delta)$ given by $B_1h_I = \tilde h_I^{(1)}$ and
  $B_2h_I = \tilde h_I^{(2)}$ which implies that for $B = B_1B_2$, $Bh_I = \tilde h^{(3)}_I$.  We extend the vector space $\mathcal{V}(\delta)$ by one
  dimension to $\tilde{\mathcal{V}} = \langle\chi_{[0,1)}\rangle + \mathcal{V}(\delta) = \langle \{\chi_I:I\in\mathcal{D}\}\rangle$ and extend $B$ on
  $\tilde{\mathcal{V}}$ by putting $B(\chi_{[0,1)}) = \chi_{[0,1)}$.  By \Cref{P:1.4.2} the map $B$ preserves distribution and therefore, if
  $\mathcal{A}(\mathcal{D})$ is the algebra generated by $\mathcal{D}$, there exists a measure preserving
  $\phi\colon\mathcal{A}(\mathcal{D})\to \mathcal{A}(\mathcal{D})$ such that for every $A\in\mathcal{A}(\mathcal{D})$, $B(\chi_A) = \chi_{\phi(A)}$.
  In particular, for every $f\in\tilde{\mathcal{V}}$ and $a\in\mathbb{R}$ we have $\{Bf = a\} = \phi(\{f = a\})$ and furthermore
  $\supp(Bf) = \phi(\supp(f))$.  Therefore, for every $I\in\mathcal{D}$ and $\varepsilon\in\{\pm1\}$ we have
  \begin{align*}
    \supp(\tilde h^{(3)}_{I^\varepsilon})
    = \supp(Bh_{I^\varepsilon})
    = \phi(\supp(h_{I^\varepsilon}))
    = \phi(\{h_I = \varepsilon\})
    = \{Bh_I = \varepsilon\}
    = \{\tilde h_I^{(3)} = \varepsilon\}.
  \end{align*}
\end{proof}

The above proof also implies that the set of all faithful Haar systems $\tilde H$ relative to some frequencies with the operation ``$\ast$'' is a
nonabelian monoid.  The next corollary's content is that $(\tilde H,D)\mapsto D|_{\tilde H}$ defines a left monoidal action.  Therefore, if $D$ and
$\widetilde H^{(1)},\ldots,\widetilde H^{(n)}$ then
\begin{equation*}
  D\stackrel{\widetilde H^{(1)}}{\longrightarrow}\tilde D_1\stackrel{\widetilde
    H^{(2)}}{\longrightarrow}\tilde D_2\stackrel{\widetilde H^{(3)}}{\longrightarrow}\cdots
  \stackrel{\widetilde H^{(n)}}{\longrightarrow}\tilde D_n
\end{equation*} coincides with
\begin{equation*}
  D \stackrel{\underrightarrow{\widetilde H^{(n)}\ast\cdots\ast \widetilde H^{(1)}}}{}\tilde D_n.
\end{equation*}
In other words, iterating the operation $D\stackrel{\tilde H}{\longrightarrow}\tilde D$ always yields a diagonal operator associated to the initial
$D$ and a faithful Haar system relative to some frequencies.

\begin{cor}\label{one-parameter monoidal action}
  Let $D\colon \mathcal{V}(\delta) \to \mathcal{V}(\delta)$ be an $\ell^{\infty}$-bounded diagonal operator and $\widetilde H^{(1)}$,
  $\widetilde H^{(2)}$ be faithful Haar systems relative to frequencies $(n_i)_{i=0}^\infty$ and $(m_i)_{i=0}^\infty$ respectively.  Then, we have
  \begin{equation*}
    (D|_{\widetilde H^{(1)}})|_{\widetilde H^{(2)}}
    = D|_{\widetilde H^{(2)} \ast \widetilde H^{(1)}}.
  \end{equation*}
\end{cor}

\begin{proof}
  If we associate $B_1,A_1$ to $\widetilde H^{(1)}$, $B_2,A_2$ to $\widetilde H^{(2)}$, and $B,A$ to $\widetilde H^{(2)} \ast \widetilde H^{(1)}$, it
  easily follows that $B_1B_2 = B$ and $A_2A_1 = Q$.  Therefore,
  \begin{equation*}
    D|_{\widetilde H^{(2)} \ast \widetilde H^{(1)}}
    = ADB
    = A_2(A_1DB_1)B_2
    = (D|_{\widetilde H^{(1)}})|_{\widetilde H^{(2)}}.
  \end{equation*}
\end{proof}

\subsection{Tree-stabilized operators in one-parameter}
Tree-stabilized diagonal operators have entries, in a strong sense, very close to each other.  As pointed out in \Cref{near scalar one par} below,
such an operator is in the $\ell^{\infty}$ sense close to a scalar operator, however, as we will see later, in the appropriate setting, this proximity
is achievable in operator norm as well.  The main result of this section is that for every $\ell^{\infty}$-bounded diagonal operator
$D\colon\mathcal{V}(\delta)\to\mathcal{V}(\delta)$ there exists a faithful Haar system $\tilde H$ relative to some frequencies $(m_i)_{i=0}^\infty$ such that
$D|_{\tilde H}$ is tree-stabilized.

\begin{dfn}\label{D:1.4.6}
  Let $D\colon \mathcal{V}(\delta)\to \mathcal{V}(\delta)$ be a diagonal operators whose coefficients are $(d_I : I\in \mathcal{D})$, let
  $\eta=(\eta_i)_{i=0}^\infty\subset (0,1)$, and let $k_0\in \mathbb{N}_0$, we call $D$ \emph{$\eta$-tree-stabilized from $k_0$ on} if
  \begin{equation*}
    |d_{L}- d_{L^{\omega}}|\le \eta_k,
    \qquad \text{for all $k\ge k_0$, $L\in \mathcal{D}_k$, and $\omega=\pm1$}.
  \end{equation*}
  If $k_0=0$ we simply say that $D$ is \emph{$\eta$-tree-stabilized.}
\end{dfn}
The notion of eventual tree-stability becomes, as we will see later, relevant in the two-parameter setting.

\begin{rem}\label{near scalar one par}
  Let $D\colon \mathcal{V}(\delta)\to \mathcal{V}(\delta)$ be a Haar multiplier and $\eta=(\eta_i)_{i=0}^\infty\subset (0,1)$.  If $D$ is $\eta$-tree-stabilized and $\sum_{i=0}^\infty2^{i}\eta_i <\infty$ then
  \begin{itemize}
  \item $\displaystyle{\lambda = \lim_i\sum_{i\in\mathcal{D}_i}|I|d_I}$ exists and
  \item $\displaystyle{ \big\|D - \lambda \Id\big\|_{\mathrm{T}} \leq \sum_{i=0}^\infty(2^{i+1}+1)\eta_i\leq 3\sum_{i=0}^\infty2^i\eta_i.}$
  \end{itemize}
\end{rem}

\begin{thm}\label{P:1.4.6}
  Let $D\colon \mathcal{V}(\delta)\to \mathcal{V}(\delta)$ be an $\ell^{\infty}$-bounded operator.  Then, for every sequence of positive numbers
  $\eta=(\eta_{j})_{j=0}^\infty$ and infinite $M\subset \mathbb{N}$ there exists a faithful Haar system $\tilde H=(\tilde h_I)_{I\in \mathcal{D}}$ relative to frequencies
  $(m_i)_{i=0}^\infty$ in $M$, such that the diagonal operator
  \begin{equation*}
    \tilde D= D|_{ \tilde H}
    \quad\text{is $\eta$-tree-stabilized.}
  \end{equation*}
  In particular, for every non-principal ultrafilter $\mathcal{U}$ on $\mathbb{N}$ we may choose $\tilde D$ such that $\lambda_\mathcal{U}(\tilde D) = \lambda_\mathcal{U}(D)$ and, thus,
  \begin{equation*}
    \|\tilde D - \lambda_\mathcal{U}(D)\Id\|_{\mathrm{T}}
    \leq 3 \sum_{i=0}^\infty2^{i}\eta_i.
  \end{equation*}
\end{thm}
The proof of the above goes through iteratively applying the probabilistic \Cref{L:1.4.3}.  In fact, we will prove the following stronger statement,
that we also apply in the two-parameter scenario.

\begin{pro}\label{stabilizing game}
  Consider the following infinite round two-player game between Player (I) and Player (II).  Denote $N^{(-1)} = \mathbb{N}_0$.  For $k=0,1,2,\ldots$ in round $k$:
  \begin{enumerate}
  \item First, Player (I) chooses $\eta_k>0$, an infinite $M_k\subset N^{(k-1)}$, and a finite collection $\mathcal{F}_k$ of $\ell^{\infty}$-bounded diagonal operators on
    $\mathcal{V}(\delta)$.

  \item Then, Player (II) chooses an infinite $N^{(k)}\subset M_k$ and, for $n_k = \min(N^{(k)})$, vectors $(\tilde h_I)_{I\in\mathcal{D}_k}$ in
    $\langle\{h_J:J\in\mathcal{D}_{n_k}\}\rangle$.
  \end{enumerate}
  Then, Player (II) has a winning strategy to force the following outcome: $\tilde H = (\tilde h_I)_{I\in\mathcal{D}}$ is a faithful Haar system
  relative to the frequencies $(n_k)_{k=0}^\infty$ and for each $k_0\in\mathbb{N}$ and $D\in\mathcal{F}_{k_0}$, $\tilde D = D|_{\tilde H}$ is
  $(\eta_k)_{k=0}^\infty$-tree -stabilized from $k_0$ on.
\end{pro}

\begin{proof}[Proof of \Cref{P:1.4.6} using \Cref{stabilizing game}]
  Let $D$ and $\eta = (\eta_j)_{j=0}^\infty$ be given and let $\lambda = \lambda_\mathcal{U}(D)$, for some non-principal ultrafilter $\mathcal{U}$ on
  $\mathbb{N}$.  Let $M$ be an infinite set such that $\lim_{i\in M}\sum_{i\in\mathcal{D}_i}|I|d_I = \lambda$.  In the game above, we assume the role
  of Player (I) and let Player (II) follow their winning strategy.  In round $k$, we chose as error the given $\eta_k$ and $\mathcal{F}_{k} = \{D\}$.
  The specific choice of $M_k$ only matters in the first round, in which we choose $M_0 = M$.

  The resulting faithful Haar system $\tilde H = (\tilde h_I)_{I\in\mathcal{D}}$ is relative to some frequencies $(m_i)_{i=0}^\infty$ with
  $\{m_i:i\in\mathbb{N}_0\}\subset M$ and $D|_{\tilde H}$ is $\eta$-tree-stabilized.  By \Cref{P:1.4.5}
  $\lim_i\sum_{I\in\mathcal{D}_i}|I|\tilde d_I = \lim_i\sum_{I\in\mathcal{D}_{m_i}}|I| d_I = \lambda$ and by \Cref{D:1.4.6},
  $\|\tilde D - \lambda I\|_\mathrm{T} \leq3 \sum_{i=0}^\infty2^i\eta_i.$
\end{proof}

The choice of the appropriate faithful Haar system $\tilde H$ is based on the following probabilistic lemma, which was first introduced
in~\cite{lechner:motakis:mueller:schlumprecht:2021} (see also~\cite[Lemma~8.3]{lechner:speckhofer:2023}).
\begin{lem}\label{L:1.4.3}
  Let $m< n$, $\Gamma\in\sigma(\mathcal{D}_m)$, $(d_I:I\in \mathcal{D}_n, I\subset \Gamma)\subset \mathbb{R}$ and put for
  $\theta=(\theta_{J}:J\in \mathcal{D}_m, J\subset \Gamma)\subset\{\pm1\}$
  \begin{equation*}
    \Gamma^\omega(\theta)
    = \Big \{ \sum_{J\in \mathcal{D}_m\atop J\subset \Gamma} \theta_{J} h_J=\omega \Big\},
    \qquad \text{for $\omega=\pm1$}.
  \end{equation*}
  Let $\Theta = (\Theta_{J}:J\in \mathcal{D}_m, J\subset \Gamma)$ be a Rademacher distributed family on a probability space $(\Omega,\Sigma,\prob)$,
  meaning that the family $(\Theta_{J}:J\in \mathcal{D}_m, J\subset \Gamma)$ is independent and
  \begin{equation*}
    \prob(\Theta_{J}=1)
    = \prob(\Theta_{J}=-1)
    = \frac{1}{2},
    \qquad J\in \mathcal{D}_m,\ J\subset \Gamma.
  \end{equation*}
  For $\omega=\pm1$ define the random variable
  \begin{equation*}
    X_\omega
    = \sum_{I\in \mathcal{D}_n,\atop I\subset \Gamma^\omega(\Theta)} d_I \frac{|I|}{| \Gamma^\omega(\Theta)|}
    = \sum_{I\in \mathcal{D}_n,\atop I\subset \Gamma^\omega(\Theta)} \frac{|I|}{|\Gamma|/2} d_I.
  \end{equation*}
  Then it follows for the expected value of $X_\omega$
  \begin{equation}\label{E:1.4.3.1}
    \cond(X_\omega)
    = \sum_{I\in \mathcal{D}_n,\atop I\subset \Gamma} \frac{|I|}{|\Gamma|} d_I,
  \end{equation}
  and for the variance of $X_\omega$
  \begin{align}\label{E:1.4.3.2}
    \var(X_\omega)
    &= \cond\Big(\Big(X_\omega-\cond(X_\omega)\Big)^2\Big)
      \le \frac{2^{-m}}{|\Gamma|} \max_{I\in \mathcal{D}_n, I\subset\Gamma}|d_I|^2.
  \end{align} 
\end{lem}

\begin{proof}
  Without loss of generality we assume $\omega=1$.  We compute
  \begin{align*}
    \cond\Big(\sum_{I\in \mathcal{D}_n,\atop I\subset \Gamma^+(\Theta)} d_I \frac{|I|}{| \Gamma^+(\Theta)|}\Big)
    &= \!\sum_{J\in \mathcal{D}_m, \atop J\subset \Gamma} \cond\Big(\sum_{I\in \mathcal{D}_n,\atop I\subset J\cap\Gamma^+(\Theta)} d_I \frac{2|I|}{| \Gamma|}\Big)
      = \!\sum_{J\in \mathcal{D}_m, \atop J\subset \Gamma} \cond\Big(\sum_{I\in \mathcal{D}_n,\atop I\subset J^{\Theta_{J}}} d_I \frac{2|I|}{| \Gamma|}\Big)
      =\!\sum_{I\in \mathcal{D}_n,\atop I\subset \Gamma} d_I \frac{|I|}{|\Gamma|} ,
  \end{align*}
  where the last equality follows from the fact that for $J\in \mathcal{D}_m$, $J\subset \Gamma$, and $I\in \mathcal{D}_n$, with $I\subset J$ we have
  \begin{equation*}
    \prob(I\subset J\cap \Gamma^+(\Theta))
    = \prob(I\subset J^{\Theta_{J}})
    = \frac{1}{2}.
  \end{equation*}
  Since the family of random variables
  \begin{equation*}
    \Big(\sum_{I\in \mathcal{D}_n,\atop I\subset J^{\Theta_{J}}} d_I \frac{2|I|}{|\Gamma|}: J\in \mathcal{D}_m,\ J \subset \Gamma\Big)
  \end{equation*}
  is independent, we deduce that
  \begin{align*}
    \cond\Big(\sum_{I\in \mathcal{D}_n,\atop I\subset \Gamma^+(\Theta)}
    d_I \frac{|I|}{| \Gamma^+(\Theta)|}- \sum_{I\in \mathcal{D}_n,\atop I\subset \Gamma} d_I \frac{|I|}{|\Gamma|}\Big)^2
    &= \sum_{J\in \mathcal{D}_m, \atop J\subset \Gamma} \cond\Big(\sum_{I\in \mathcal{D}_n,\atop I\subset J^{ \Theta_{J}}} d_I \frac{2|I|}{|\Gamma|}-
      \sum_{I\in \mathcal{D}_n,\atop I\subset J} d_I \frac{|I|}{|\Gamma|} \Big)^2\\
    &= \sum_{J\in \mathcal{D}_m, \atop J\subset \Gamma} \cond\Big(\sum_{I\in \mathcal{D}_n,\atop I\subset J^{\Theta_{J}}} d_I \frac{|I|}{|\Gamma|}-
      \sum_{I\in \mathcal{D}_n,\atop I\subset J^{-\Theta_{J}}} d_I \frac{|I|}{|\Gamma|} \Big)^2\\
    &\le \frac{\max_{I\in \mathcal{D}_n, I\subset\Gamma}|d_I|^2}{|\Gamma|^2}\sum_{J\in \mathcal{D}_m\atop J\subset \Gamma} |J|^2
      = \frac{2^{-m}}{|\Gamma|}\max_{I\in \mathcal{D}_n, I\subset\Gamma}|d_I|^2.
  \end{align*}
\end{proof}

\begin{rem}\label{simultaneous probabilistic choice}
  From \Cref{L:1.4.3} it follows that, under the assumption stated in the Lemma, for any $\delta>0$,
  \begin{equation*}
    \mathbb{P}\Big(
    \max_{\omega=\pm1}\Big|X_\omega - \sum_{J\in\mathcal{D}_m\atop J\subset \Gamma} \frac{|J|}{|\Gamma|} d_J\Big|
    \leq  \delta + \Big|
    \sum_{J\in\mathcal{D}_m\atop J\subset \Gamma} \frac{|J|}{|\Gamma|} d_J - \sum_{L\in\mathcal{D}_n\atop L\subset \Gamma} \frac{|L|}{|\Gamma|} d_L
    \Big|
    \Big)
    \geq 1 - \frac{2^{-m}}{\delta^2|\Gamma|} \max_{I\in\mathcal{D}_n\atop I\subset\Gamma} d_I^2.
  \end{equation*}
  Indeed, because $X_1 + X_{-1} = 2\cond(X_1) = 2\cond(X_{-1})$ we have $|X_1 - \cond(X_1)| = |X_{-1} - \cond(X_{-1})|$ and thus
  \begin{align*}
    \max_{\omega=\pm1}\Big|X_\omega - \sum_{J\in\mathcal{D}_m\atop J\subset \Gamma} \frac{|J|}{|\Gamma|} d_J\Big|
    &\leq \max_{\omega=\pm 1} |X_\omega - \mathbb E(X_\omega)|
      + \Big|
      \sum_{J\in\mathcal{D}_m\atop J\subset \Gamma} \frac{|J|}{|\Gamma|} d_J
      - \sum_{L\in\mathcal{D}_n\atop L\subset \Gamma} \frac{|L|}{|\Gamma|} d_L
      \Big|\\
    &= |X_1 - \cond(X_1)|
      + \Big|
      \sum_{J\in\mathcal{D}_m\atop J\subset \Gamma} \frac{|J|}{|\Gamma|} d_J
      -  \sum_{L\in\mathcal{D}_n\atop L\subset \Gamma} \frac{|L|}{|\Gamma|} d_L
      \Big|.
  \end{align*}
  The conclusion follows from Chebyshev's inequality.

  Therefore, if we have fixed $m$, $n$, and $\Gamma$ and a varying, but finite, collection of sets of coefficients
  $(d_I^\alpha:I\in\mathcal{D}_n, I\subset \Gamma)$, $\alpha\in\mathcal{A}$, such that
  \begin{equation*}
    \frac{2^{-m/2}}{|\Gamma|^{1/2}}
    \max_{I\subset \mathcal{D}_n, I\subset\Gamma\atop\alpha\in\mathcal{A}} |d^\alpha_I|
    \leq \frac{\delta}{|\mathcal{A}|^{1/2}},
  \end{equation*}
  we can choose a common $(\theta_{J}: J\in \mathcal{D}_m, J\subset \Gamma)\subset\{\pm 1\}$ such that for $\omega\in\{\pm1\}$ and
  \begin{equation*}
    \Gamma^{\omega}
    = \big\{ \sum_{J\in\mathcal{D}_m, J\subset \Gamma} \theta_{J} h_J
    = \omega \big\}
  \end{equation*}
  it follows that for all $\alpha\in\mathcal{A}$,
  \begin{equation}\label{simultaneously close successor}
    \biggl|
    \sum_{I\in \mathcal{D}_n,\atop I\subset\Gamma^{\omega}} d^\alpha_I \frac{|I|}{|\Gamma^{\omega}|}
    - \sum_{J\in \mathcal{D}_m,\atop J\subset \Gamma}  d^\alpha_J \frac{|J|}{|\Gamma|}
    \biggr|
    \leq \delta + \biggl|
    \sum_{J\in\mathcal{D}_m\atop J\subset \Gamma} \frac{|J|}{|\Gamma|} d^\alpha_J
    - \sum_{L\in\mathcal{D}_n\atop L\subset \Gamma} \frac{|L|}{|\Gamma|} d^\alpha_L
    \biggr|.
  \end{equation}
\end{rem}

\begin{ntn}
  For $i\in\mathbb{N}$ and $I\in\mathcal{D}_i$ we denote $\pi(I)$ the unique $K\in\mathcal{D}_{i-1}$ such that $I = K^+$ or $I = K^-$.
\end{ntn}

\begin{lem}\label{one par stab conditions}
  Let $D\colon \mathcal{V}(\delta)\to \mathcal{V}(\delta)$ be an $\ell^{\infty}$-bounded operator with coefficients $(d_I)_{I\in\mathcal{D}}$,
  $d = \sup_{I\in \mathcal{D}} |d_{I}|$, $k_0\in\mathbb{N}_0$, and $\eta = (\eta_i)_{i=k_0}^\infty\subset(0,1)$.  If $k_0 >0$, additionally assume
  that we have an initially defined faithful Haar system $(\tilde h_I)_{I\in\mathcal{D}^{k_0-1}}$ relative to frequencies $(m_i)_{i=0}^{k_0 - 1}$.
  Let $\mathbb{N} \supset N^{(k_0)}\supset N^{(k_0+1)}\supset\cdots$ be such that, if $m_i = \min(N^{(k_0)})$, for $i\geq k_0$, then
  $(m_i)_{i=0}^\infty$ is strictly increasing and let $\theta^{(i)} = (\theta^{(i)}_{L}: L\in \mathcal{D}_{m_i})\subset \{\pm1\}$, $i\geq k_0$, be
  such that for $i\geq k_0$, $(\tilde h_I)_{I\in\mathcal{D}_i}$ are defined as below and the following are satisfied.
  \begin{enumerate}[label=(\roman*)]
  \item\label{initial} If $k_0 = 0$ then
    \begin{equation*}
      \tilde h_{[0,1)}
      = \sum_{L\in \mathcal{D}_{m_0}} \theta^{(0)}_{L} h_L
    \end{equation*}
    and for all $m,n\in N^{(0)}$
    \begin{equation*}
      \Big|\sum_{J\in\mathcal{D}_m} |J| d_J - \sum_{L\in\mathcal{D}_n} |L| d_L\Big|
      < \eta_0.
    \end{equation*}

  \item\label{continuing} For $i\geq \max\{1,k_0\}$, $I\in\mathcal{D}_i$, and $\omega \in\{\pm 1\}$ such that $I = \pi(I)^{\omega}$ we have
    \begin{equation*}
      \tilde h_I
      = \sum_{L\in \mathcal{D}_{m_i}\atop L\subset \{\tilde h_{\pi(I)}=\omega\}} \theta^{(i)}_{L} h_L.
    \end{equation*}

  \item\label{successor} For $i\geq k_0$, $I\in\mathcal{D}_i$ and $\omega\in\{\pm 1\}$
    \begin{equation*}
      \Big |
      \sum_{J\in \mathcal{D}_{m_i}\atop J \subset \supp(\tilde h_I)} \frac{|J|}{|I|} d_J
      - \sum_{L\in \mathcal{D}_{m_{i+1}},\atop L\subset \{\tilde h_I=\omega\}} \frac{|L|}{|I|/2} d_L
      \Big|
      < \eta_i.
    \end{equation*}
  \end{enumerate}
  Then for $\tilde H = (\tilde h_I)_{I\in\mathcal{D}}$ the diagonal operator $\tilde D = D|_{\tilde H}$ is $\eta$-tree-stabilized from $k_0$ on.
  Furthermore, if $k_0 = 0$ and $\lambda$ is a limit point of $(\sum_{J\in\mathcal{D}_n} |J| d_J)_{n\in N^{(0)}}$, then
  $|\tilde d_{[0,1)} - \lambda| \leq \eta_0$.
\end{lem}

\begin{proof}
  It follows from~\ref{initial} and~\ref{continuing} that $(\tilde h_I:I\in \mathcal{D})$ is a faithful Haar system relative to the frequencies
  $(m_i)_{i=0}^\infty$.  For $i\geq k_0$, $I\in \mathcal{D}_i$, and $\omega\in\{\pm1\}$ we have, by \Cref{P:1.4.5}, that
  \begin{equation*}
    |\tilde d_I - \tilde d_{I^\omega}|
    = \Big|
    \sum_{J\in \mathcal{D}_{m_i}\atop J \subset \supp(\tilde h_I)} \frac{|J|}{|I|} d_J
    - \sum_{L\in \mathcal{D}_{m_{i+1}},\atop L\subset \{\tilde h_I=\omega\}} \frac{|L|}{|I|/2} d_L
    \Big|
    < \eta_i.
  \end{equation*}
  The ``furthermore'' statement is an immediate consequence of the second part of~\ref{initial}.
\end{proof}

For the purpose of the following lemma, an initially defined faithful Haar system relative to frequencies $(m_i)_{i=0}^{k_0}$ is a finite collection
$(\tilde h_I)_{I\in\mathcal{D}^{k_0}}$ that satisfies \Cref{dfn:faithful}~\ref{dfn:faithful i} for $I\in\mathcal{D}^{k_0}$, \ref{dfn:faithful ii} for
$I\in\mathcal{D}^{k_0-1}$, and \ref{D:1.4.1} for $i\leq k_0$.

\begin{lem}\label{one par stab step}
  Let $D_\alpha\colon \mathcal{V}(\delta)\to \mathcal{V}(\delta)$, $\alpha\in\mathcal{A}$, be a finite collection of $\ell^{\infty}$-bounded
  operators, each with coefficients $(d^\alpha_I)_{I\in\mathcal{D}}$.  Let $k_0\in\mathbb{N}_0$, $\delta > 0$, and $N$ and be an infinite subset of
  $\mathbb{N}$.  If $k_0 > 0$, let also $(\tilde h_I)_{I\in \mathcal{D}^{k_0-1}}$ be an initially defined faithful Haar system relative to frequencies
  $(m_i)_{i=0}^{k_0-1}$.  Then, there exist an infinite subset $N^{(k_0)}$ of $N$ and for $m_{k_0} = \min(N^{(k_0)})$ there exist
  $(\theta^{(k_0)}_{L} : L\in \mathcal{D}_{m_{k_0}})\subset \{\pm1\}$ such that $(\tilde h_I)_{I\in\mathcal{D}_{k_0}}$ are defined as follows and the
  following are satisfied.
  \begin{enumerate}[label=(\roman*)]
  \item\label{initial step} If $k_0 = 0$ then
    \begin{equation*}
      \tilde h_{[0,1)}
      = \sum_{L\in \mathcal{D}_{m_0}} \theta^{(0)}_{L} h_L
    \end{equation*}
    and for all $m,n\in N^{(0)}$ and $\alpha\in\mathcal{A}$
    \begin{equation*}
      \Big|\sum_{J\in\mathcal{D}_m} |J| d^\alpha_J - \sum_{L\in\mathcal{D}_n} |L| d^\alpha_L\Big|
      < \delta.
    \end{equation*}
  \item\label{continuing step} If $k_0\geq 1$ then $m_{k_0}>m_{k_0-1}$.  Also, for $I\in\mathcal{D}_{k_0}$ and $\omega \in\{\pm 1\}$ such that
    $I = \pi(I)^{\omega}$,
    \begin{equation*}
      \tilde h_I
      = \sum_{L\in \mathcal{D}_{m_{k_0}}\atop L\subset \{\tilde h_{\pi(I)}=\omega\}}
      \theta^{(k_0)}_{L} h_L.
    \end{equation*}

  \item\label{successor step} For $I\in\mathcal{D}_{k_0}$, $\omega\in\{\pm 1\}$, $n\in N^{(k_0)}\setminus\{m_{k_0}\}$, and $\alpha\in\mathcal{A}$
    \begin{equation*}
      \Big|
      \sum_{J\in \mathcal{D}_{m_{k_0}}\atop J \subset \supp(\tilde h_I)} \frac{|J|}{|I|} d^\alpha_J
      - \sum_{L\in \mathcal{D}_{n}\atop L\subset \{\tilde h_I=\omega\}} \frac{|L|}{|I|/2} d^\alpha_L
      \Big|
      < \delta.
    \end{equation*}
  \end{enumerate}
  In particular, $(\tilde h_I)_{I\in\mathcal{D}^{k_0}}$ is an initially defined faithful Haar system relative to the frequencies $(m_i)_{i=0}^{k_0}$.
\end{lem}

\begin{proof}
  The only notable difference in the case $k_0 = 0$ is that the second part of \ref{initial step} must be additionally achieved.  This can be easily
  achieved using the $\ell^{\infty}$-boundedness of the given diagonal operators.  Therefore, we assume that $k_0 \geq 1$.  For each
  $I\in\mathcal{D}_{k_0}$ denote $\tilde I = \{\tilde h_{\pi(I)} = \xi\}$, where $\xi\in\{\pm1\}$ is such that $I = \pi(I)^\xi$.  Choose an infinite
  $M\subset N$ with $m_{k_0} = \min(M) >m_{k_0-1}$ and
  \begin{equation*}
    2^{(k_0-m_{k_0})/2}\max_{\alpha\in\mathcal{A}}\|D_\alpha\|_\infty
    < \frac{\delta/2}{|\mathcal{A}|^{1/2}}
  \end{equation*} such that for all $I\in\mathcal{D}_{k_0}$ and $m,n\in M$
  we have
  \begin{equation}\label{appropriate proximity}
    \Big|
    \sum_{K\in \mathcal{D}_{m}\atop K\subset \tilde I } \frac{|K|}{|I|} d^\alpha_K
    - \sum_{L\in \mathcal{D}_{n}\atop L\subset \tilde I} \frac{|L|}{|I|} d^\alpha_L
    \Big|
    < \delta/2.
  \end{equation}

  We next wish to apply \Cref{simultaneous probabilistic choice} and specify parameters $m, n$ and $\Gamma$.  We take $m = m_{k_0}$ and, for
  $I\in\mathcal{D}_{k_0}$, take $\Gamma = \tilde I$.  We consider these parameters to be fixed and proceed with the argument for arbitrary
  $n\in M\setminus\{m_{k_0}\}$, which will lead to an infinite choice of appropriate $n$ to define our set $N^{(k_0)}$.

  By \eqref{appropriate proximity} and \eqref{simultaneously close successor} we can find
  $\kappa_{I,n} = (\theta_{n,L}: L\in\mathcal{D}_m,L\subset\tilde I)$ such that for $\omega\in\{\pm1\}$ and
  \begin{equation*}
    \tilde I_n^{\omega}
    = \big\{ \sum_{J\in\mathcal{D}_m, J\subset \tilde I} \theta_{n,J}h_J = \omega\big\}
  \end{equation*}
  it follows that for all $\alpha\in\mathcal{A}$,
  \begin{equation}\label{simultaneously close successor step}
    \biggl|
    \sum_{L\in \mathcal{D}_n,\atop L\subset\tilde I_n^{\omega}} \frac{|L|}{|I|/2} d^\alpha_L
    - \sum_{J\in \mathcal{D}_{m_{k_0}},\atop J\subset \tilde I} \frac{|J|}{|I|} d^\alpha_J
    \biggr|
    < \delta.
  \end{equation}
  By the pigeonhole principle, we may choose an infinite subset $\Lambda$ of $M$ such that for all $I\in\mathcal{D}_{k_0}$ and $n,n'\in \Lambda$ we
  have $\kappa_{I,n} = \kappa_{I,n'}$.  Define $N^{(k_0)} = \{m_{k_0}\}\cup \Lambda$ and concatenate the selected vectors of signs $\kappa_{I,n}$,
  $I\in \mathcal{D}_{k_0}$, and $n\in \Lambda$ by (slightly abusing the notation and) putting
  $\theta^{(k_0)} = (\kappa_{I,n} : I\in\mathcal{D}_{k_0})$.  For each $I\in\mathcal{D}_{k_0}$ define
  \begin{equation*}
    \tilde h_I
    = \sum_{L\in\mathcal{D}_{m_{k_0}}\atop L\subset\tilde I} \theta^{(k_0)}_{L} h_L
  \end{equation*}
  and thus, for $\omega\in\{\pm1\}$, $\{\tilde h_I = \omega\} = \tilde I_\omega$.  In other words, \eqref{simultaneously close successor step}
  yields~\ref{successor step}.
\end{proof}

\begin{proof}[Proof of \Cref{stabilizing game}]
  In this game, Player (II) recursively applies \Cref{one par stab step}, in each step $k$, to the finite collection $\cup_{i=0}^k\mathcal{F}_i$ and
  using the set $M_k$ and $\eta_k$ given by Player (I) to achieve the assumptions of \Cref{one par stab conditions} for each $\mathcal{F}_{k_0}$ from
  $k_0$ on.
\end{proof}

\subsection{The 2-parameter case}

\begin{pro}\label{P:3.2}
  Let $\tilde H\otimes \tilde K= (\tilde h_I\otimes \tilde k_J: I,J\in\mathcal{D})$ be a faithful Haar system relative to the frequencies $(m_i)$ and
  $(n_j)$, and let $D\colon \mathcal{V}(\delta^2)\to \mathcal{V}(\delta^2)$ be a diagonal operator.  Then $\tilde D = D|_{\tilde H\otimes \tilde K}$
  is a diagonal operator with coefficients $(\tilde d_{I,J}: I,J\in \mathcal{D})$ given by
  \begin{equation*}
    \tilde d_{ I,J}
    = \sum_{L\in \mathcal{D}_{m_i},\atop L\subset \supp(\tilde h_I)} \frac{|L|}{|I|}
    \sum_{M\in \mathcal{D}_{n_j},\atop M\subset \supp( \tilde k_J)} \frac{|M|}{|J|}
    d_{L,M}
  \end{equation*}
  for $i,j\in\mathbb{N}_0, I\in\mathcal{D}_i$ and $J\in\mathcal{D}_j$.
\end{pro}

\begin{proof}
  For $I\in \mathcal{D}_i$ write $\tilde h_I$ and $\tilde k_I$ as
  \begin{equation*}
    \tilde h_I
    = \sum_{L\in \mathcal{A}_I} \theta_L h_L
    \qquad\text{and}\qquad
    \tilde k_I
    = \sum_{M\in \mathcal{B}_I } \varepsilon_M k_M,
  \end{equation*}
  with $\mathcal{A}_I\subset \mathcal{D}_{m_i}$, $\mathcal{B}_I\subset \mathcal{D}_{n_i}$,
  $(\theta_L: L\in \mathcal{A}_I), (\varepsilon_M: M\in \mathcal{B}_I)\subset \{\pm1\}$.  Let $i,i',j,j'\in \mathbb{N}_0$, and $I\in\mathcal{D}_i$,
  $I'\in \mathcal{D}_{i'}$, $J\in\mathcal{D}_j$ and $J'\in\mathcal{D}_{j'}$.  Then
  \begin{align*}
    \langle \tilde h_{I'}\otimes  \tilde k_{J'}, D(\tilde h_{I}\otimes \tilde k_{J}) \rangle
    &= \Big\langle \sum_{L'\in \mathcal{A}_{I'}} \theta_{L'} h_{L'}\otimes\tilde  k_{J'},
      D\Big( \sum_{L\in \mathcal{A}_I} \theta_L h_{L}\otimes \tilde k_J\Big)\Big\rangle\\
    &= \sum_{L'\in \mathcal{A}_{I'}, L\in \mathcal{A}_I} \theta_{L'}\theta_L \Big\langle h_{L'}\otimes \tilde k_{J'},  h_L\otimes D_{L,\cdot}( \tilde k_J)\Big\rangle.
  \end{align*}
  where $D_{L,\cdot}\colon \mathcal{V}(\delta)\to \mathcal{V}(\delta)$ is the $1$-parameter diagonal operator defined by
  \begin{align*}
    D_{L,\cdot}(h_M)
    = d_{L,M} h_M,
    \qquad L,M\in \mathcal{D}.
  \end{align*}
  Thus by \Cref{P:1.4.5}, we obtain for $J\neq J'$ that
  \begin{equation*}
    \langle \tilde h_{I'}\otimes  \tilde k_{J'}, D(\tilde h_{I}\otimes \tilde k_{J})\rangle
    = \sum_{\substack{L'\in \mathcal{A}(I')\\L\in \mathcal{A}(I)}} \theta_{I'}\theta_I \langle h_L,h_{L'}\rangle \langle \tilde k_{J'}, D_{L,\cdot}( \tilde k_J)\rangle\\
    = 0.
  \end{equation*}
  Similarly using \Cref{P:1.4.5} in the case that $J=J'$, and $I\neq I'$, we obtain that
  \begin{align*}
    \langle \tilde h_{I'}\otimes k_{L'}, D(\tilde h_{K}\otimes k_{L}) \rangle
    &= \sum_{\substack{L'\in \mathcal{A}_{I'}\\L\in \mathcal{A}_I}} \theta_{L'}\theta_L \langle h_L,h_{L'}\rangle \langle \tilde k_{J}, D_{L,\cdot}( \tilde k_J)\rangle\\
    &= \sum_{\substack{L'\in \mathcal{A}_{I'}\\L\in \mathcal{A}_I}} \theta_{L'}\theta_{L} \langle h_{L'},h_{L}\rangle \sum_{M\in \mathcal{B}_J} |M| d_{L,M} 
    = 0.
  \end{align*}
  In the case that $I=I'$ and $J=J'$, we obtain that
  \begin{align*}
    \langle \tilde h_{I'}\otimes k_{J'}, D(\tilde h_{I}\otimes k_{J}) \rangle
    &= \sum_{L\in \mathcal{A}_I} |L| \sum_{M\in \mathcal{B}_J} |M| d_{L,M}.
  \end{align*}
  It follows, therefore, that
  \begin{align*}
    A\circ D\circ B(h_I \otimes k_J)
    &= \sum_{I',J'\in \mathcal{D}}\frac{h_{I'} \otimes k_{J'} }{|I'|\cdot |J'|}  \langle \tilde h_{I'}\otimes \tilde k_{J'} ,  D(\tilde  h_{I}\otimes \tilde k_{J})\rangle\\
    &= \frac{h_{I} \otimes k_{J} }{|I|\cdot |J|} \sum_{L\in \mathcal{A}_I} |L| \sum_{M\in \mathcal{B}_J} |M| d_{L,M}\\
    &= h_I \otimes k_J\sum_{L\in \mathcal{D}_{m_i}\atop L\subset \supp(\tilde h_I)} \frac{|L|}{|I|}
      \sum_{M\in \mathcal{D}_{n_j}\atop M\subset \supp(\tilde k_J)} \frac{|M|}{|J|} d_{L,M}.
  \end{align*} 
\end{proof}

As in the one-parameter case, in the two-parameter case, starting with an $\ell^{\infty}$ diagonal operator and iterating the process of passing to
faithful Haar systems is in a sense stable.
\begin{ntn}
  \label{bi-parameter monoidal operation}
  Let $\widetilde H^{(1)}$, $\widetilde K^{(1)}$ be a faithful Haar system relative to frequencies $(m_i)_{i=0}^\infty$ and $(n_i)_{i=0}^\infty$
  respectively and let $\widetilde H^{(2)}$, $\widetilde K^{(2)}$ be a faithful Haar system relative to frequencies $(s_i)_{i=0}^\infty$ and
  $(t_i)_{i=0}^\infty$ respectively.  We define
  \begin{equation*}
    (\widetilde H^{(2)}\otimes \widetilde K^{(2)}) \ast (\widetilde H^{(1)} \otimes \widetilde K^{(1)})
    = (\widetilde H^{(2)}\ast \widetilde H^{(1)}) \otimes (\widetilde K^{(2)}\ast \widetilde K^{(1)}).
  \end{equation*}
\end{ntn}

The proof of the following is an easy adaptation of the proof \Cref{one-parameter monoidal action} using \Cref{bi-parameter B-Q notation} and
\Cref{bi-parameter monoidal operation}.
\begin{pro}\label{blocking bi-parameter faithful haar}
  Let $\widetilde H^{(1)}$, $\widetilde K^{(1)}$ be a faithful Haar system relative to frequencies $(m_i)_{i=0}^\infty$ and $(n_i)_{i=0}^\infty$
  respectively and let $\widetilde H^{(2)}$, $\widetilde K^{(2)}$ be a faithful Haar system relative to frequencies $(s_i)_{i=0}^\infty$ and
  $(t_i)_{i=0}^\infty$ respectively.  Then, for every $\ell^{\infty}$-bounded diagonal operator
  $D\colon\mathcal{V}(\delta^2)\to\mathcal{V}(\delta^2)$,
  \begin{equation*}
    (D|_{\widetilde H^{(1)}\otimes \widetilde K^{(1)}})|_{\widetilde H^{(2)} \otimes \widetilde K^{(2)}}
    = D|_{(\widetilde H^{(2)}\otimes \widetilde K^{(2)})\ast (\widetilde H^{(1)} \otimes \widetilde K^{(1)})}.
  \end{equation*}
\end{pro}

\subsection{Bi-tree-semi-stabilized operators in two parameters}
This is a notion on diagonal operators on $\mathcal{V}(\delta^2)$ that means that they can be ``split'' into two pieces, an upper triangular and a
lower triangular one, each of which has essentially constant coefficients.

\begin{dfn}\label{dfn:semi-stable}
  (Bi-tree-semi-stabilized operators in two parameters) Let $D\colon \mathcal{V}(\delta^2)\to \mathcal{V}(\delta^2)$ be a diagonal operator whose
  coefficients we denote by $(d_{I,J} : I,J\in\mathcal{D})$ and let $\eta = (\eta_{i,j})_{i,j=0}^\infty\in (0,1)^{\mathbb{N}^2}$ and $\delta>0$.  We
  say that $D$ satisfies:
  \begin{enumerate}[label=(\alph*)]
  \item\label{matrix:lower} \emph{The lower triangular condition for $\eta$} if for all $i\geq j\geq 0$, $I\in\mathcal{D}_i$, $J\in\mathcal{D}_j$, and
    $\omega\in\{\pm1\}$
    \begin{equation*}
      |d_{I^\omega,J} - d_{I,J}|
      < \eta_{i,j}.
    \end{equation*}
    
  \item\label{matrix:upper} \emph{The upper triangular condition for $\eta$} if for all $j > i \geq 0$, $I\in\mathcal{D}_i$, $J\in\mathcal{D}_j$, and
    $\xi\in\{\pm1\}$
    \begin{equation*}
      |d_{I,J} - d_{I,J^\xi}|
      < \eta_{i,j}.
    \end{equation*}
    
  \item\label{matrix:diag} \emph{The diagonal condition for $\eta$} if for all $i\ge 0$, $I,J\in \mathcal{D}_i$, and $\omega,\xi\in\{\pm1\}$ then\\
    \begin{equation*}
      |d_{I,J} - d_{I^\omega,J^\xi}|
      \le \eta_{i,i}.
    \end{equation*}
    
  \item\label{matrix:sup} \emph{The superdiagonal condition for $\eta$} if for all $i\ge 0$, $I\in \mathcal{D}_{i}$ and $J\in \mathcal{D}_{i+1}$, and
    $\omega,\xi\in\{\pm1\}$ then
    \begin{equation*}
      |d_{I,J} -d_{I^\omega,J^\xi}|
      \le \eta_{i,i+1}.
    \end{equation*}

  \item\label{matrix:bal} \emph{The balancing condition for $\delta$} if we have $|d_{[0,1),[0,1/2)} - d_{[0,1),[1/2,1)}| < \delta$.
  \end{enumerate}
  If \ref{matrix:lower}-\ref{matrix:bal} are all satisfied then we call $D$ \emph{$(\eta,\delta)$-bi-tree-semi-stabilized}.  (For a visualization we
  refer to the figures below.)
\end{dfn}

\begin{minipage}{.5\linewidth}
  \begin{center}
    \includegraphics[scale=0.05]{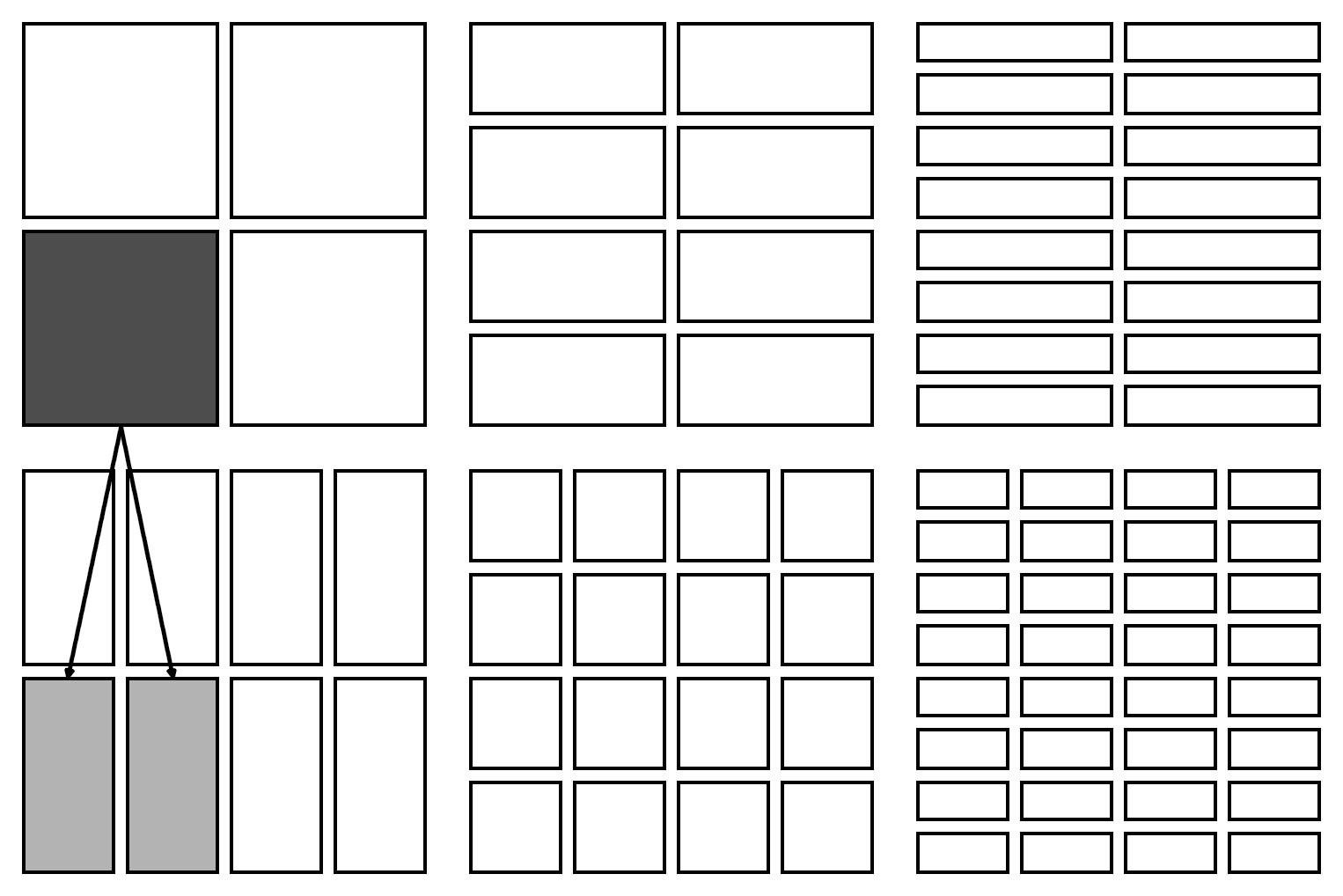}
  \end{center}
  \captionof*{figure}{\small Lower triangular condition~\ref{matrix:lower}.}
  \label{fig:bi-tree-stable-lower_triangular}

  \begin{center}
    \includegraphics[scale=0.05]{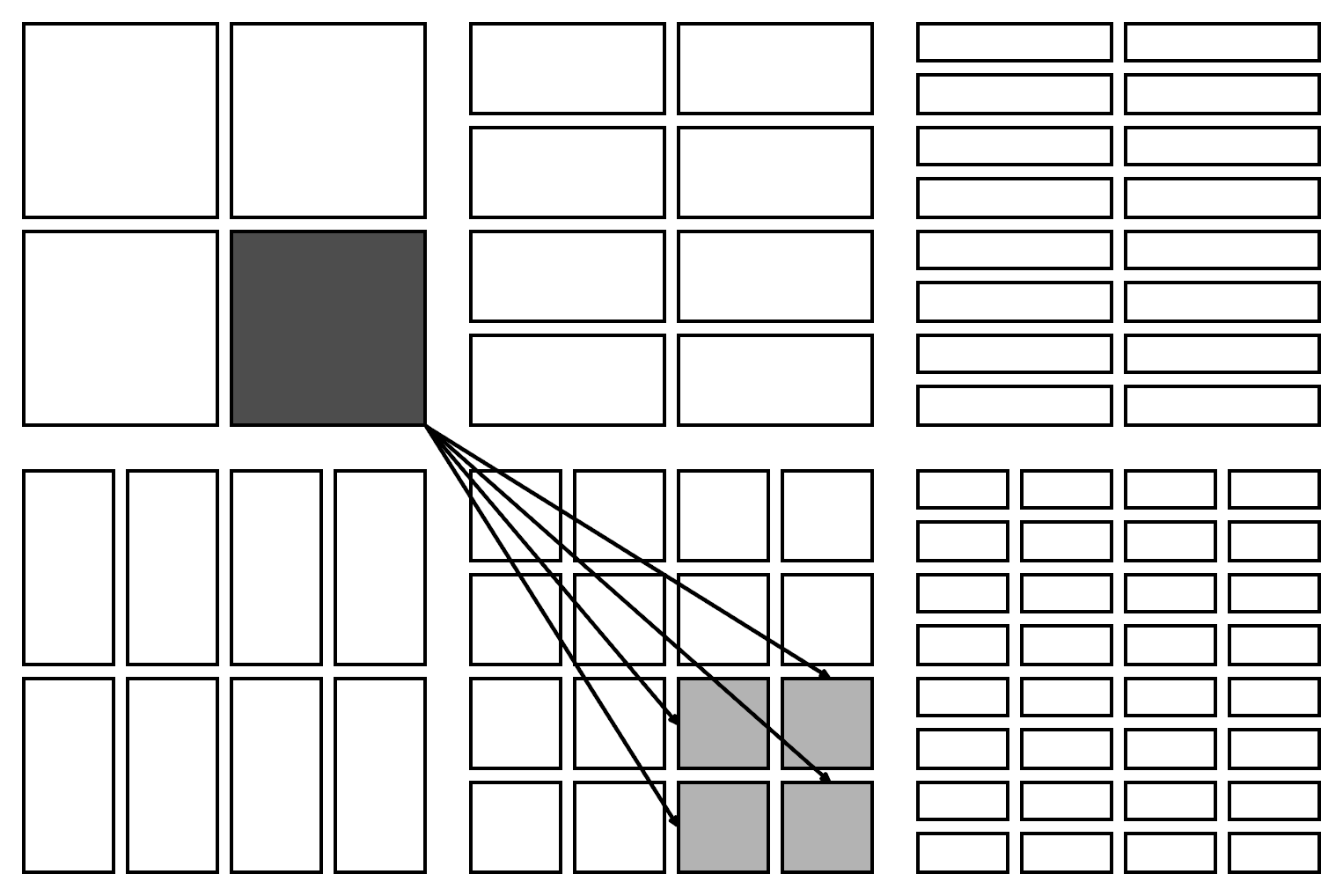}
  \end{center}
  \captionof*{figure}{\small Diagonal condition~\ref{matrix:diag}.}
  \label{fig:bi-tree-stable-diagonal}
\end{minipage}
\begin{minipage}{.5\linewidth}
  \begin{center}
    \includegraphics[scale=0.05]{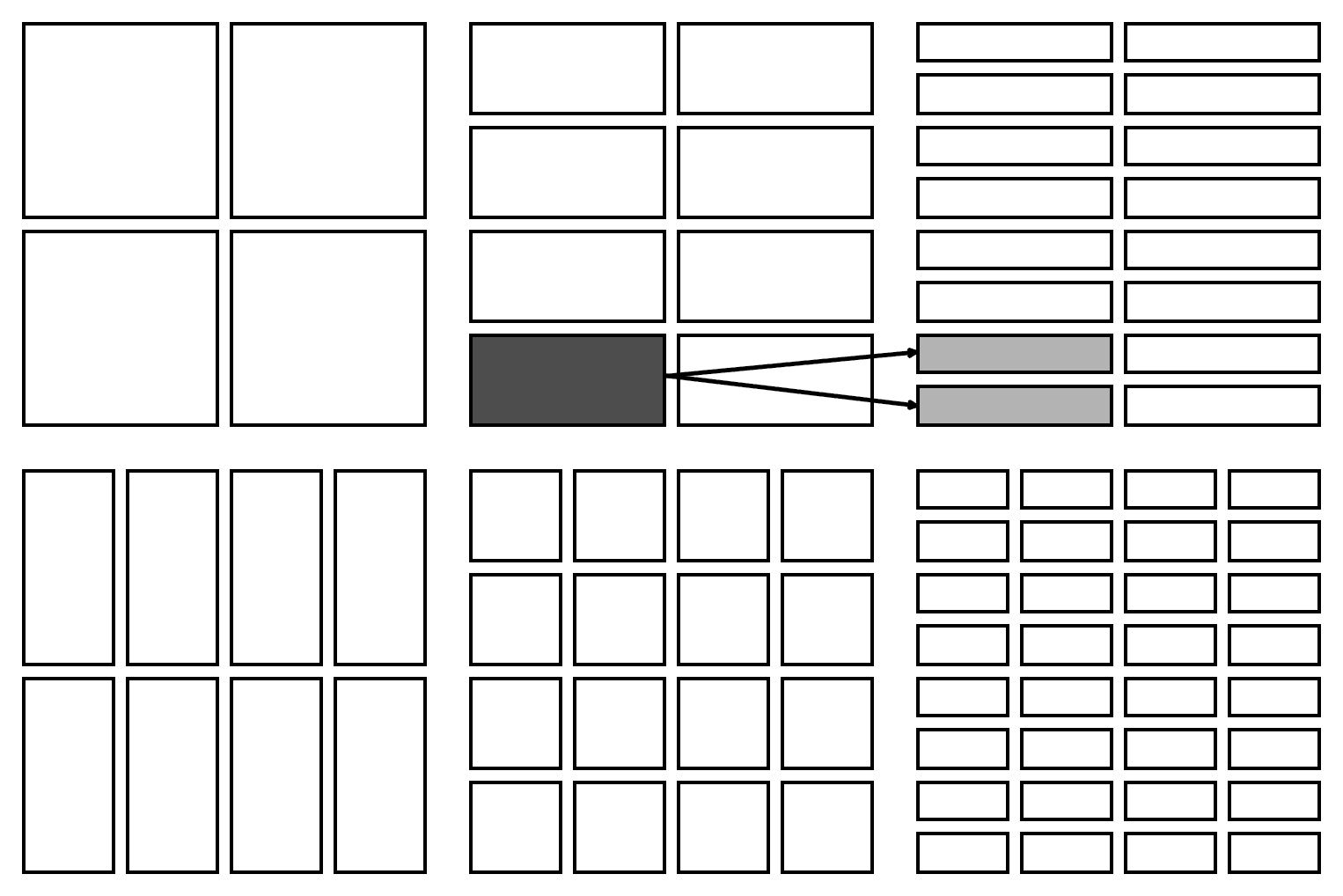}
  \end{center}
  \captionof*{figure}{\small Upper triangular condition~\ref{matrix:upper}.}
  \label{fig:bi-tree-stable-upper_triangular}

  \begin{center}
    \includegraphics[scale=0.05]{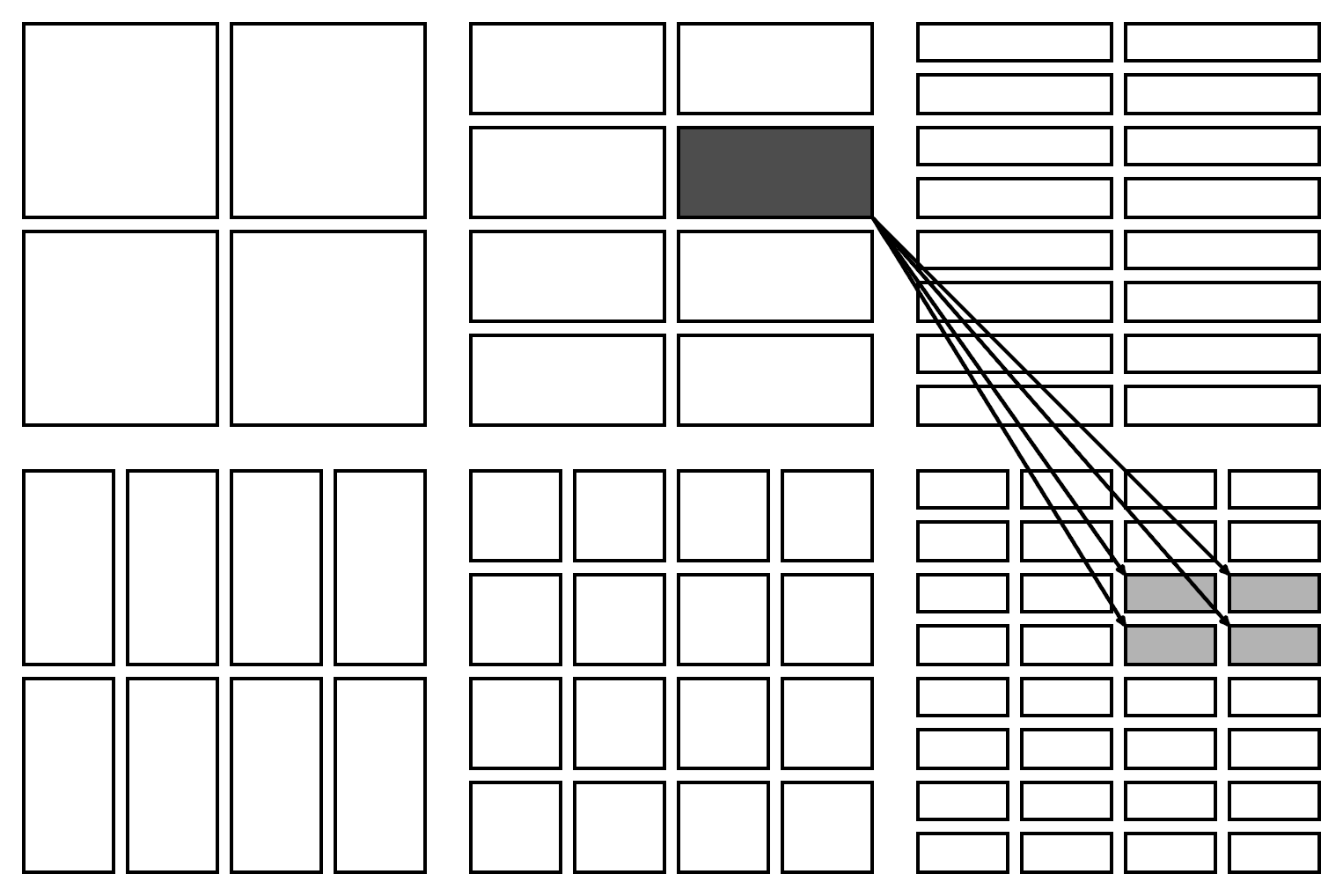}
  \end{center}
  \captionof*{figure}{\small Superdiagonal condition~\ref{matrix:sup}.}
  \label{fig:bi-tree-stable-superdiagonal}
\end{minipage}

\begin{rem}\label{equivalent upper and lower triangular}
  Conditions \ref{matrix:lower} and \ref{matrix:upper} of \Cref{dfn:semi-stable} may be reformulated as follows.
  \begin{enumerate}[label=(\alph*')]
  \item\label{matrix:upper:equiv} (The lower triangular condition) For $j\in\mathbb{N}_0$ and $J\in \mathcal{D}_j$, the diagonal operator
    $D_{(\cdot,J)}\colon \mathcal{V}(\delta)\to \mathcal{V}(\delta)$, whose coefficients are $(d_{I,J}: I\in \mathcal{D})$ is
    $(\eta_{i,j})_{i=0}^\infty$-stabilized from $j$ on.
  \item\label{matrix:lower:equiv} (The upper triangular condition) For $i\in\mathbb{N}_0$ and $I\in \mathcal{D}_i$, the diagonal operator
    $D_{(I,\cdot)}\colon \mathcal{V}(\delta)\to \mathcal{V}(\delta)$ whose coefficients are $(d_{I,J}: I\in \mathcal{D})$ is
    $(\eta_{i,j})_{j=0}^\infty$-stabilized from $i+1$ on.
  \end{enumerate}
  The above will be used later to apply one-parametric stabilization arguments to a bi-parameter Haar multiplier.
\end{rem}

The following is the bi-parameter version of \Cref{near scalar one par}.  Observe that, in this case, a bi-tree-semi-stabilized Haar multiplier is not
close to a scalar operator.
\begin{rem}\label{proximity in sup norm}
  Let $D\colon\mathcal{V}(\delta^2)\to \mathcal{V}(\delta^2)$ be a Haar multiplier, $\eta = (\eta_{i,j})_{i,j=0}^\infty\in (0,1)^{\mathbb{N}^2}$, and
  $\delta>0$.  Assume that $D$ is $(\eta,\delta)$-bi-tree-semi-stabilized and $\sum_{i,j=0}^\infty2^{i+j}\eta_{i,j}<\infty$.  Then,
  \begin{enumerate}[label=(\roman*)]
  \item $\displaystyle{\lambda(D) = \lim_{j\to\infty} \lim_{i\to\infty} \sum_{I\in\mathcal{D}_i}\sum_{J\in\mathcal{D}_j}|I||J|d_{I,J}}$ \quad and\quad
    $\displaystyle{\mu(D) = \lim_{i\to\infty} \lim_{j\to\infty} \sum_{I\in\mathcal{D}_i}\sum_{J\in\mathcal{D}_j}|I||J|d_{I,J}}$ exist and

  \item
    $\displaystyle{\Big\|D - \Big(\lambda(D)\mathcal{C} + \mu(D)(\Id-\mathcal{C})\Big)\Big\|_{\mathrm{T^2}}\leq \sum_{i=0}^\infty\sum_{j=0}^\infty
      (i+j+4)\eta_{i,j} + \delta}$.
  \end{enumerate}
  In particular, for any non-principal ultrafilter $\mathcal{U}$ on $\mathbb{N}$, $\lambda_\mathcal{U}(D) = \lambda(D)$ and
  $\mu_\mathcal{U}(D) = \mu(D)$.

  This estimate yields that a bi-tree-semi-stabilized diagonal operator is in $\|\cdot\|_{\mathrm{T^2}}$-norm close to a diagonal operator
  $\lambda\mathcal{C} + \mu (\Id - \mathcal{C})$.  If $Z$ is a Capon space, this proximity can also be achieved in operator norm (see
  \Cref{two-parameter analytic main theorem}).  The proof of the above goes along the same lines as the proof of \Cref{pro:path-distance} further
  below.
\end{rem}

The main theorem of this section states that for every $\ell^\infty$-bounded Haar multiplier $D$ there exist faithful Haar systems such that the
corresponding $\tilde D$ is bi-tree-semi-stabilized.  Furthermore, the values $\lambda$ and $\mu$ can, up to an arbitrarily small error, be determined
by the valuation of the linear functionals $\lambda_\mathcal{U}(\cdot),\mu_\mathcal{U}(\cdot)$ defined on the vector space of all
$\ell^\infty$-bounded bi-parameter Haar multipliers.
\begin{thm}\label{T:3.4}
  Let $D\colon \mathcal{V}(\delta^2)\to \mathcal{V}(\delta^2)$ be a diagonal operator whose coefficients we denote by $(d_{I,J} : I,J\in\mathcal{D})$
  which is $\ell^{\infty}$-bounded, $\eta = (\eta_{i,j} )_{i,j=0}^\infty\in (0,1)^{\mathbb{N}^2}$, and $\delta>0$.  Then for any infinite $M$,
  $N\subset\mathbb{N}$ there exist faithful Haar systems $\tilde H=(\tilde h_I:I\in \mathcal{D})$ and $ \tilde K=( \tilde k_J:J\in \mathcal{D})$
  relative to frequencies $(m_i)_{i=0}^\infty$ in $M$ and $(n_j)_{j=0}^\infty$ in $N$, such that
  \begin{equation*}
    \tilde D = D|_{\tilde H\otimes\tilde K}
    \quad\text{is $(\eta,\delta)$-bi-tree-semi-stabilized}.
  \end{equation*}
  In particular, for any $\mathcal{U}\in\beta \mathbb{N}\setminus \mathbb{N}$, we may choose $\tilde D$ such that
  $\lambda_\mathcal{U}(\tilde D) = \lambda_\mathcal{U}(D)$ and $\mu_\mathcal{U}(\tilde D) = \mu_\mathcal{U}(D)$ and, thus,
  \begin{equation*}
    \big\|\tilde D - \big(\lambda_\mathcal{U}(D)\mathcal{C}+\mu_\mathcal{U}(\Id-\mathcal{C})\big)\big\|{_{\mathrm{T^2}}}
    \leq \sum_{i=0}^\infty\sum_{j=0}^\infty (i+j+4)\eta_{i,j} + \delta.
  \end{equation*}
  Moreover, $|\tilde d_{[0,1),[0,1)} - \lambda(D)| < \delta$ and $|\tilde d_{[0,1),[0,1/2)} - \mu(D)| < \delta$.
\end{thm}

\subsection{Permanence of stability under blocking}
To make the proof of \Cref{T:3.4} more tractable, we will show, separately, how each of properties \ref{matrix:lower}-\ref{matrix:bal} of
\Cref{dfn:semi-stable} can be achieved in isolation by appropriately blocking the Haar system in each parameter.  The permanence of
\ref{matrix:lower}-\ref{matrix:bal}, under appropriate blocking, guarantees that the resulting diagonal operator satisfies all of them.

We start with a summability assumption on sequences $(\eta_{i,j})_{i,j=1}^\infty$ with respect to which we consider semi-stability the purpose of
which is to improve notation.
\begin{asm}\label{R:3.5}
  Assume $\eta = (\eta_{i,j} )_{i,j=0}^\infty\in (0,1)^{\mathbb{N}^2}$ has the property
  \begin{equation}\label{E:3.4.1}
    \sum_{j=j_0+1}^\infty \eta_{i_0,j}
    < \frac{\eta_{i_0,j_0}}{3}
    \quad\text{and}\quad
    \sum_{i=i_0+1}^\infty \eta_{i,j_0}
    < \frac{\eta_{i_0,j_0}}{3},
    \qquad\text{for $i_0,j_0\in\mathbb{N}$}.
  \end{equation}
\end{asm}

\begin{pro}\label{pro:path-distance}
  Let $D\colon \mathcal{V}(\delta^2)\to\mathcal{V}(\delta^2)$ be an $\ell^\infty$-bounded Haar multiplier whose coefficients are
  $(\eta_{i,j})_{i,j=0}^\infty$-bi-tree-semi-stabilized and assume \eqref{E:3.4.1}. Let $(K_0,L_0)$, $(K_1,L_1)\in\mathcal{D}\times\mathcal{D}$ such
  that $K_0\supset K_1$ and $L_0\supset L_1$.
  \begin{enumerate}[label=(\roman*)]
  \item\label{pro:path-distance lower} If $|K_0|\leq |L_0|$ and $|K_1|\leq |L_1|$, i.e., both $(K_0,L_0)$ and $(K_1,L_1)$ are in the lower triangle
    part of $\mathcal{D}\times\mathcal{D}$, then
    \begin{equation*}
      |d_{(K_0,L_0)} - d_{(K_1,L_1)}|
      \leq \eta_{l_0-1,l_0-1},
    \end{equation*}
    where $2^{-l_0} = |L_0|.$

  \item\label{pro:path-distance upper} If $|K_0| > |L_0|$ and $|K_1| > |L_1|$, i.e., both $(K_0,L_0)$ and $(K_1,L_1)$ are in the upper triangle part
    of $\mathcal{D}\times\mathcal{D}$, then
    \begin{equation*}
      |d_{(K_0,L_0)} - d_{(K_1,L_1)}|
      \leq \eta_{k_0-1,k_0},
    \end{equation*}
    where $2^{-k_0} = |K_0|.$
  \end{enumerate}
\end{pro}

\begin{proof}
  We only prove~\ref{pro:path-distance lower}.  The proof of~\ref{pro:path-distance upper} is similar.  The most efficient way to obtain these
  estimates is by telescoping along a specific path contained in the lower triangular submatrix. Put $2^{-k_1} = |K_1|$, $2^{-l_1} = |L_1|$. The path
  is constructed by first traversing upwards from both points $(K_i,L_i)$, $i=0,1$ until they hit the diagonal and then join those two points on the
  diagonal by moving along the diagonal.  The vertical paths $P_i$, $i=0,1$ are given by
  \begin{equation*}
    P_i
    = \bigl\{ (K,L_i) : K_i' \supsetneq  K \supset K_i \bigr\},
  \end{equation*}
  where $K_i'$ defined as the unique dyadic interval containing $K_i$ such that $|K_i'| = |L_i| = 2^{-l_i}$ for $i=0,1$.  Now we join the above
  vertical paths along the diagonal:
  \begin{equation*}
    P_2
    = \bigl\{ (K,L) : K_0'\supsetneq K \supset K_1',\ L_0\supsetneq L \supset L_1,\ |K| = |L| \bigr\}.
  \end{equation*}
  We will now estimate $|d_{K_0,L_0} - d_{K_1,L_1}|$ by telescoping first along $P_0$, then $P_{2}$ and finally $P_1$:
  \begin{equation}\label{eq:8}
    |d_{K_0,L_0} - d_{K_1,L_1}|
    \leq \sum_{i=0}^1 \sum_{(K,L_i)\in P_i} |d_{(K,L_i)} - d_{(\pi(K),L_i)}|
    + \sum_{(K,L)\in P_2} |d_{(K,L)} - d_{(\pi(K),\pi(L))}|.
  \end{equation}
  By the lower triangular and diagonal condition (see \Cref{dfn:semi-stable}~\ref{matrix:lower} and~\ref{matrix:diag}), we obtain
  \begin{equation}\label{eq:11}
    \begin{aligned}
      |d_{(K,L_i)} - d_{(\pi(K),L_i)}|
      &\leq \eta_{k-1,l_i},
      && (K,L_i)\in P_i,\ |K|=2^{-k},\ i=0,1,\\
      |d_{(K,L)} - d_{(\pi(K),\pi(L))}|
      &\leq \eta_{k-1,k-1},
      && |K| = |L| = 2^{-k}
    \end{aligned}
  \end{equation}
  Inserting~\eqref{eq:11} into~\eqref{eq:8} and using~\eqref{E:3.4.1} yields
  \begin{equation}\label{eq:13}
    \begin{aligned}
      |d_{K_0,L_0} - d_{K_1,L_1}|
      &\leq \sum_{i=0}^1 \sum_{k=l_i+1}^{k_i}  \eta_{k-1,l_i}
        + \sum_{k=l_0+1}^{l_1} \eta_{k-1,k-1}
        \leq \sum_{i=0}^1 \frac{1}{3}\eta_{l_i-1,l_i} + \frac{1}{9}\eta_{l_0-1,l_0-1}\\
      &\leq \eta_{l_0-1,l_0-1}.
    \end{aligned}
  \end{equation}
\end{proof}

\begin{pro}\label{pro:permanence}
  Assume that $D\colon \mathcal{V}(\delta^2)\to \mathcal{V}(\delta^2)$ is a diagonal operator which is $\ell^{\infty} $-bounded and
  $\eta=(\eta_{i,j} )_{i,j=0}^\infty\in (0,1)^{\mathbb{N}^2}$ that satisfies \eqref{E:3.4.1}.  Let $\tilde H=(\tilde h_K:K\in \mathcal{D})$ and
  $ \tilde K=( \tilde k_L:L\in \mathcal{D})$ be faithful Haar systems relative to the frequencies $(m_i)_{i=0}^\infty$ and $(n_j)_{j=0}^\infty$
  respectively, such that, for $1\leq n_i < m_i < n_{i+1}$, $i\geq 0$.  Denote $\tilde D= D|_{\tilde H\otimes \tilde K}$.
  \begin{enumerate}[label=(\roman*)]
  \item\label{enu:pro:permanence:lower} If $D$ satisfies the lower triangular condition for $\eta$ then $\tilde D$ satisfies the lower triangular
    condition for $\eta$.
  \item\label{enu:pro:permanence:upper} If $D$ satisfies the upper triangular condition for $\eta$ then $\tilde D$ satisfies the upper triangular
    condition for $\eta$.
  \item\label{enu:pro:permanence:diagonal} If $D$ satisfies the lower triangular condition and the diagonal condition for $\eta$ then $\tilde D$
    satisfies the diagonal condition for $\eta$.
  \item\label{enu:pro:permanence:super} If $D$ satisfies the upper triangular condition and the superdiagonal condition for $\eta$ then $\tilde D$
    satisfies the superdiagonal condition for $\eta$.
  \item\label{enu:pro:permanence:balance} If $D$ satisfies the upper triangular condition and the superdiagonal condition for $\eta$ and the balancing
    condition for some $\delta>0$ then $\tilde D$ satisfies the balancing condition for $\tilde \delta = \delta
    +\sum_{i,j}\eta_{i,j}$.  \end{enumerate}
\end{pro}

\begin{myproof}
  Since the proof of~\ref{enu:pro:permanence:lower} and~\ref{enu:pro:permanence:upper} as well as the proof of~\ref{enu:pro:permanence:diagonal}
  and~\ref{enu:pro:permanence:super} are completely analogous, we will only present the proof for~\ref{enu:pro:permanence:lower}
  and~\ref{enu:pro:permanence:diagonal}. In all cases, the proof comes down to expressing a difference of successive entries of $\tilde D$ as an
  average of distances along paths and then evoking \Cref{pro:path-distance}. We do not prove \Cref{enu:pro:permanence:balance}, because we do not use
  it, but it follows with similar arguments.
  
  Let $(d_{K,L}:K,L\in \mathcal{D})$ be the coefficients of $D$ and $(\tilde d_{K,L}:K,L\in \mathcal{D})$ be the coefficients of
  $\tilde D=D|_{\tilde H\otimes\tilde K}$.
  \begin{proofcase}[Proof of~\ref{enu:pro:permanence:lower}]
    For $k\geq l$, $I\in \mathcal{D}_k$ and $J\in \mathcal{D}_l$, and $\theta\in\{\pm 1\}$ we have
    \begin{equation}\label{eq:2(isitnow)}
      \begin{aligned}
        |\tilde d_{I,J}-\tilde d_{I^\theta,J}|
        &= \Bigl|
          \sum_{K\in \mathcal{D}_{m_k}\atop K\subset \supp(\tilde h_I)}
          \sum_{L\in \mathcal{D}_{n_{l}}\atop L\subset \supp( \tilde k_J)}
          \frac{|K| |L|}{|I| |J|} d_{K,L}
          - \sum_{K\in \mathcal{D}_{m_{k+1}}\atop K\subset \supp(\tilde h_{I^{\theta}})}
          \sum_{L\in \mathcal{D}_{n_{l}}\atop L\subset \supp( \tilde k_J)}
          \frac{|K| |L|}{|I^{\theta}| |J|} d_{K,L}
          \Bigr|\\
        &\leq \sum_{L\in \mathcal{D}_{n_{l}}\atop L\subset \supp( \tilde k_J)}
          \frac{|L|}{|J|}
          \Bigl|
          \sum_{K\in \mathcal{D}_{m_k}\atop K\subset \supp(\tilde h_I)}
          \frac{|K|}{|I|} d_{K,L}
          - \sum_{K\in \mathcal{D}_{m_{k+1}}\atop K\subset \supp(\tilde h_{I^{\theta}})}
          \frac{|K|}{|I^{\theta}|} d_{K,L}
          \Bigr|.
      \end{aligned}
    \end{equation}
    Next, we record the identity
    \begin{equation}\label{eq:3}
      \sum_{K\in \mathcal{D}_{m_k}\atop K\subset \supp(\tilde h_I)}
      \frac{|K|}{|I|} d_{K,L}
      - \sum_{K\in \mathcal{D}_{m_{k+1}}\atop K\subset \supp(\tilde h_{I^{\theta}})}
      \frac{|K|}{|I^{\theta}|} d_{K,L}
      = \sum_{K_0\in \mathcal{D}_{m_k}\atop K_0\subset \supp(\tilde h_I)}
      \frac{|K_0|}{|I|} \Bigl(d_{K_0,L} - 2\cdot \sum_{K\in \mathcal{D}_{m_{k+1}}\atop K\subset K_0^{\theta}} \frac{|K|}{|K_{0}|} d_{K,L}\Bigr).
    \end{equation}
    Combining~\eqref{eq:2(isitnow)} and~\eqref{eq:3} yields
    \begin{equation}
      \label{eq:9}
      \begin{aligned}
        |\tilde d_{I,J}-\tilde d_{I^\theta,J}|
        &\leq \sum_{L\in \mathcal{D}_{n_{l}}\atop L\subset \supp( \tilde k_J)}
          \frac{|L|}{|J|}
          \sum_{K_0\in \mathcal{D}_{m_k}\atop K_0\subset \supp(\tilde h_I)}
          \frac{|K_0|}{|I|}\cdot \Bigl|
          d_{K_0,L} - 2\cdot \sum_{K\in \mathcal{D}_{m_{k+1}}\atop K\subset K_0^{\theta}} \frac{|K|}{|K_{0}|} d_{K,L}
          \Bigr|\\
        &\leq \sum_{L\in \mathcal{D}_{n_{l}}\atop L\subset \supp( \tilde k_J)}
          \frac{|L|}{|J|}
          \sum_{K_0\in \mathcal{D}_{m_k}\atop K_0\subset \supp(\tilde h_I)}
          \frac{|K_0|}{|I|} \sum_{K\in \mathcal{D}_{m_{k+1}}\atop K\subset K_0^{\theta}} \frac{|K|}{|K_{0}|}  |d_{K_0,L} - d_{K,L}|,
      \end{aligned}
    \end{equation}
    where, by \Cref{pro:path-distance}~\ref{pro:path-distance lower}, each
    $|d_{K_0,L} - d_{K,L}| \leq \eta_{n_l-1,n_l-1} \leq \eta_{l,l} \leq \eta_{k,l}$. Hence, since
    \begin{equation*}
      \sum_{L\in \mathcal{D}_{n_{l}}\atop L\subset \supp( \tilde k_J)}
      \frac{|L|}{|J|}
      \sum_{K_0\in \mathcal{D}_{m_k}\atop K_0\subset \supp(\tilde h_I)}
      \frac{|K_0|}{|I|} \sum_{K\in \mathcal{D}_{m_{k+1}}\atop K\subset K_0^{\theta}} \frac{|K|}{|K_{0}|}
      = 1,
    \end{equation*}
    our estimates \eqref{eq:9} yields $|\tilde d_{I,J}-\tilde d_{I^\theta,J}| \leq \eta_{k,l}$, as claimed.
  \end{proofcase}

  \begin{proofcase}[Proof of~\ref{enu:pro:permanence:diagonal}]
    For $k\geq 0$, $I,J\in \mathcal{D}_k$ and $\theta,\varepsilon\in\{\pm 1\}$ we have
    \begin{equation}\label{eq:2}
      \begin{aligned}
        |\tilde d_{I,J}-\tilde d_{I^\theta,J^\varepsilon}|
        &= \biggl|
          \sum_{K_0\in \mathcal{D}_{m_k}\atop K_0\subset \supp(\tilde h_I)}
          \sum_{L_0\in \mathcal{D}_{n_{k}}\atop L_0\subset \supp( \tilde k_J)}
          \frac{|K_0| |L_0|}{|I| |J|} d_{K_0,L_0}
          - 4\cdot \sum_{K\in \mathcal{D}_{m_{k+1}}\atop K\subset \supp(\tilde h_{I^{\theta}})}
          \sum_{L\in \mathcal{D}_{n_{k+1}}\atop L\subset \supp( \tilde k_{J^\varepsilon})}
          \frac{|K| |L|}{|I| |J|} d_{K,L}
          \biggr|\\
        &\leq \biggl|
          \sum_{K_0\in \mathcal{D}_{m_k}\atop K_0\subset \supp(\tilde h_I)}
          \frac{|K_0|}{|I|}
          \sum_{L_0\in \mathcal{D}_{n_{k}}\atop L_0\subset \supp( \tilde k_J)}
          \frac{|L_0|}{|J|}
          \Bigl( d_{K_0,L_0}
          - 4\cdot \sum_{K\in \mathcal{D}_{m_{k+1}}\atop K\subset K_0^{\theta}}
          \frac{|K|}{|K_0|}
          \sum_{L\in \mathcal{D}_{n_{k+1}}\atop L\subset L_0^{\varepsilon}}
          \frac{|L|}{|L_0|} d_{K,L}
          \Bigr)
          \biggr|.
      \end{aligned}
    \end{equation}
    Moreover, since
    \begin{equation*}
      4\cdot \sum_{K\in \mathcal{D}_{m_{k+1}}\atop K\subset K_0^{\theta}}
      \frac{|K|}{|K_0|}
      \sum_{L\in \mathcal{D}_{n_{k+1}}\atop L\subset L_0^{\varepsilon}}
      \frac{|L|}{|L_0|}
      = 1,
    \end{equation*}
    we further estimate~\eqref{eq:2} by
    \begin{align*}
      |\tilde d_{I,J}-\tilde d_{I^\theta,J^\varepsilon}|
      &\leq 4\cdot \biggl|
        \sum_{K_0\in \mathcal{D}_{m_k}\atop K_0\subset \supp(\tilde h_I)}
        \frac{|K_0|}{|I|}
        \sum_{L_0\in \mathcal{D}_{n_{k}}\atop L_0\subset \supp( \tilde k_J)}
        \frac{|L_0|}{|J|}
        \sum_{K\in \mathcal{D}_{m_{k+1}}\atop K\subset K_0^{\theta}}
        \frac{|K|}{|K_0|}
        \sum_{L\in \mathcal{D}_{n_{k+1}}\atop L\subset L_0^{\varepsilon}}
        \frac{|L|}{|L_0|}
        \bigl(d_{K_0,L_0} - d_{K,L}\bigr)
        \biggr|.
    \end{align*}
    Recall the by \Cref{pro:path-distance}~\ref{pro:path-distance lower} $|d_{K_0,L_0} - d_{K,L}|\leq \eta_{n_{k}-1,n_{k}-1} \leq \eta_{k,k}$, and
    thus by convexity $|\tilde d_{I,J}-\tilde d_{I^\theta,J^\varepsilon}| \leq \tilde{\eta}_{k,k}$.\qedhere
  \end{proofcase}
\end{myproof}

\subsection{Stabilizing two-parameter diagonal operators: the upper and lower diagonal conditions}
For a diagonal operator $D\colon\mathcal{V}(\delta^2)\to\mathcal{V}(\delta^2)$ and a Haar-type vector $\tilde h$, $\tilde k$ we wish to define
diagonal operators $D_{(\tilde h,\cdot)}$ and $D_{(\cdot,\tilde k)}\colon\mathcal{V}(\delta)\to\mathcal{V}(\delta)$.  For motivational purposes, we do
so indirectly by first giving a more general way of reducing $D$ to a one-parameter operator.

\begin{ntn}\hfill
  \begin{enumerate}[label=(\alph*)]
  \item For a linear operator $D\colon\mathcal{V}(\delta^2)\to\mathcal{V}(\delta^2)$ and $f^*,f\in\mathcal{V}(\delta)$ define
    $D_{(f^*,f)}\colon\mathcal{V}(\delta)\to\mathcal{V}(\delta)$ given by
    \begin{equation*}
      D_{(f^*,f)}(g)
      = \sum_{J\in\mathcal{D}}\frac{k_J}{|J|}\langle f^*\otimes k_J,D(f\otimes g)\rangle
    \end{equation*}
    If in particular, if $D\colon\mathcal{V}(\delta^2)\to\mathcal{V}(\delta^2)$ is diagonal with entries $(d_{I,J})_{I,J\in\mathcal{D}}$ then
    $D_{(f^*,f)}$ is diagonal as well with entries $\tilde d_J = \sum_{I\in\mathcal{D}}|I|^{-1}\langle f^*,h_I\rangle\langle h_I,f\rangle d_{I,J}$,
    $J\in\mathcal{D}$.

  \item For a diagonal linear operator $D\colon\mathcal{V}(\delta^2)\to\mathcal{V}(\delta^2)$ and $I\in\mathcal{D}$ take $f^* = h_I/|I|$ and $f = h_I$
    and denote $D_{(I,\cdot)} = D_{(f^*,f)}$.  This coincides with the definition of $D_{(I,\cdot)}$ in \Cref{equivalent upper and lower
      triangular}~\ref{matrix:upper:equiv}.

  \item For a diagonal linear operator $D\colon\mathcal{V}(\delta^2)\to\mathcal{V}(\delta^2)$ and a Haar-type vector of the form
    $\tilde h = \sum_{I\in\mathcal{A}}\varepsilon_I h_I$, where, for some $m\in\mathbb{N}$, $\mathcal{A}\subset\mathcal{D}_{m}$ and
    $(\varepsilon_I)_{I\in\mathcal{A}}\in\{\pm1\}^\mathcal{A}$, denote $D_{(\tilde h,\cdot)} = D_{(f^*,f)}$, for $f^* = \tilde h/|\cup\mathcal{A}|$
    and $f = \tilde h$.  Then,
    \begin{equation*}
      D_{(\tilde h,\cdot)}
      = |\mathcal{A}|^{-1}\sum_{I\in\mathcal{A}}D_{(I,\cdot)}.
    \end{equation*}

    The preceding definitions explain why the signs $(\varepsilon_I)_{I\in\mathcal{A}}$ don't appear in the formula for $D_{(\tilde h,\cdot)}$.

  \item For a diagonal linear operator $D\colon\mathcal{V}(\delta^2)\to\mathcal{V}(\delta^2)$ and a Haar-type vector of the form
    $\tilde k = \sum_{J\in\mathcal{D}}\theta_J\tilde k_J$, where, for some $n\in\mathbb{N}$, $\mathcal{B}\subset\mathcal{D}_n$ and
    $(\theta_J)_{J\in\mathcal{B}}\in\{\pm1\}^\mathcal{B}$, denote
    \begin{equation*}
      D_{(\cdot,\tilde k)}
      = |\mathcal{B}|^{-1}\sum_{J\in\mathcal{B}}D_{(\cdot,J)},
    \end{equation*} 
    where the $D_{(\cdot,J)}$ were defined in \Cref{equivalent upper and lower triangular}~\ref{matrix:lower:equiv}.

    Note that $D_{(\cdot,\tilde k)}$ could have also been defined indirectly, by taking
    $D^{(g^*,g)}(f) = \sum_{I\in\mathcal{D}}\frac{h_I}{|I|}\langle h_I\otimes g^* , D(f\otimes g)\rangle$ for appropriate $g^*$, $g$ in
    $\mathcal{D}(\delta)$.
  \end{enumerate}
\end{ntn}

\begin{lem}\label{restated triangular on blocking}
  Let $D\colon\mathcal{V}(\delta^2)\to\mathcal{V}(\delta^2)$ be an $\ell^{\infty}$-bounded diagonal operator,
  $\tilde H=(\tilde h_I)_{I\in \mathcal{D}}$ and $ \tilde K=( \tilde k_J)_{J\in \mathcal{D}}$ be faithful Haar systems relative to the frequencies
  $(m_i)_{i=0}^\infty$ and $(n_j)_{j=0}^\infty$ respectively and put $\tilde D = D|_{\tilde H\otimes \tilde K}$.  Then, for every $I,J\in\mathcal{D}$,
  \begin{equation}\label{shortcut tilde}
    D_{(\tilde h_I,\cdot)}|_{\tilde K}
    = \tilde D_{(I,\cdot)}\text{ and }D_{(\cdot,\tilde k_J)}|_{\tilde H}
    = \tilde D_{(\cdot,J)}.
  \end{equation}
  Let also $\eta = (\eta_{i,j})_{i,j=0}^\infty\subset(0,1)$.
  \begin{enumerate}[label=(\alph*)]
  \item\label{restated triangular on blocking lower} If for all $j\in\mathbb{N}_0$ and $J\in\mathcal{D}_j$ the operator
    $D_{(\cdot,\tilde k_J)}|_{\tilde H}$ is $(\eta_{i,j})_{i=0}^\infty$-stabilized from $j$ on then $D|_{\tilde H\otimes\tilde K}$ satisfies the lower
    triangular condition.

  \item\label{restated triangular on blocking upper} If for all $i\in\mathbb{N}_0$ and $I\in\mathcal{D}_i$ the operator
    $D_{(\tilde h_I,\cdot)}|_{\tilde K}$ is $(\eta_{i,j})_{j=0}^\infty$-stabilized from $i+1$ on then $D|_{\tilde H\otimes\tilde K}$ satisfies the
    upper triangular condition.
  \end{enumerate}
\end{lem}

\begin{proof}
  Note that items~\ref{restated triangular on blocking lower} and~\ref{restated triangular on blocking upper} follow directly from~\eqref{shortcut
    tilde} and \Cref{equivalent upper and lower triangular}.  We show the first part of~\eqref{shortcut tilde}.  Let $(d_{I,J})_{I,J\in\mathcal{D}}$
  denote the entries of $D$.  For all $i,j\in\mathbb{N}_0$, $I\in\mathcal{D}_i$, and $J\in\mathcal{D}_j$ write
  $\mathcal{A}_I = \{L\in\mathcal{D}_{m_i}:L\subset\mathrm{supp}(\tilde h_I)\}$ and
  $\mathcal{B}_J = \{M\in\mathcal{D}_{n_j}:M\subset\mathrm{supp}(\tilde k_J)\}$.  Then, the operator $\tilde D$ has entries
  $(\tilde d_{I,J})_{I,J\in\mathcal{D}}$ with
  $\tilde d_{I,J} = \sum_{L\in\mathcal{A}_I} \frac{|L|}{|I|} \sum_{M\in\mathcal{B}_J} \frac{|M|}{|J|} d_{L,M}$, whenever $I\in\mathcal{D}_i$ and
  $J\in\mathcal{D}_j$.

  Fix now $i\in\mathbb{N}$ and $I\in\mathcal{D}_i$.  Then, $\tilde D_{(I,\cdot)}$ is the diagonal operator with entries
  $(\tilde d_{I,J})_{J\in\mathcal{D}}$.  By definition, $\tilde D_{(\tilde h_I,\cdot)}$ is the diagonal operator that has entries
  $\bigl(\sum_{L\in\mathcal{A}_I} \frac{|L|}{|I|} d_{L,J}\bigr)_{J\in\mathcal{D}}$.  We then deduce that $D_{(\tilde h_I,\cdot)}|_{\tilde K}$ is the
  diagonal operator such that, for each $J\in\mathcal{D}_j$, has entry
  $c_J = \sum_{M\in\mathcal{B}_J} \frac{|M|}{|J|}\sum_{L\in\mathcal{A}_I} \frac{|L|}{|I|} d_{L,M}$, which coincides with $\tilde d_{I,J}$, the
  corresponding entry of $\tilde D_{(I,\cdot)}$.
\end{proof}

\begin{pro}\label{stabilization by one-parameter reduction}
  Let $D\colon\mathcal{V}(\delta^2)\to\mathcal{V}(\delta^2)$ be an $\ell^{\infty}$-bounded diagonal operator.  Then, for every infinite subsets $M$,
  $N$ of $\mathbb{N}$ and $\eta=(\eta_{i,j} )_{i,j=0}^\infty\in (0,1)^{\mathbb{N}^2}$ there exist faithful Haar systems
  $\tilde H=(\tilde h_I:I\in \mathcal{D})$ and $ \tilde K=( \tilde k_J:J\in \mathcal{D})$ relative to frequencies $(m_i)_{i=0}^\infty\subset M$ and
  $(n_j)_{j=0}^\infty\subset N$ respectively, such that, for $i\geq 0$, $1\leq n_i < m_i < n_{i+1}$, and $D|_{\tilde H\otimes\tilde K}$ satisfies the
  lower and upper triangular conditions for $\eta$.
\end{pro}

\begin{proof}
  We proceed with an application of the game in \Cref{stabilizing game}.  There will be two sets of the game being played simultaneously, one set
  $\mathrm{G}_1$ for defining $\tilde H = (\tilde h_I:I\in \mathcal{D})$ and another set $\mathrm{G}_2$ for defining
  $ \tilde K=( \tilde k_J:J\in \mathcal{D})$.  In both, we assume the role of Player I and let Player II follow their winning strategy.  As Player I,
  in each round $k$ and for $\mathrm{G}_1$ we will choose an appropriate error, a finite collection of $\ell^{\infty}$-bounded diagonal operators
  $\mathcal{F}_k^{1}$, and an infinite sets $M^1_k\subset N^1_{(k-1)}$.  Then, Player II will choose $N^1_{(k)}\subset M^1_k$, $n_k=\min( N^1_{(k)})$,
  and $(\tilde k_J)_{J\in\mathcal{D}_{{k}}}$ in $\langle\{k_L:L\in\mathcal{D}_{n_k}\}\rangle$.  Then we play the $k$'th round of $\mathrm{G}_2$ by
  choosing an error, $\mathcal{F}_k^{2}$, $M_k^2\subset N^2_{(k-1)}$ and let Player II chose $N^2_{(k)}\subset M^2_k$, $m_k = \min(N^2_{(k)})$, and
  $(\tilde h_I)_{I\in\mathcal{D}_{k}}$ in $\langle\{h_K:K\in\mathcal{D}_{m_k}\}\rangle$.
  
  In the zeroth round of $\mathrm{G}_1$ we choose $M^1_0 = N$, but our choice of operators or error does not really matter.  So we let Player II
  choose $n_0$ and $\tilde k_{[0,1)}$.  In the zeroth round of $G_2$, we choose error $\eta_0' = \eta_{0,0}$,
  $\mathcal{F}^{2}_0 = \{D_{(\cdot,\tilde k_{[0,1)})}\}$, while our choice of $M^2_{0}$ is a subset of $M$ with $\min(M^2_0) > n_0$.  Player II then
  chooses $N^2_{(0)}\subset M^2_0$, lets $m_0 = \min(N^2_{(0)})> n_0$, and chooses $\tilde h_{[0,1)}$.
  
  Assume that we have played rounds $0,1,\ldots,k-1$.  In round $k$ of $\mathrm{G}_1$, choose error $\eta_k' = \min_{i,j \leq k}\eta_{i,j}$,
  $\mathcal{F}_k^1 = \{D_{(\tilde h_I,\cdot)}:I\in\mathcal{D}_{k-1}\}$, and $M^1_k\subset N^1_{(k-1)}$ such that $m_{k-1}<\min(M^1_k)$.  Then, Player
  II chooses $N^1_{(k)}\subset M^1_k$, lets $n_k = \min(N^1_{(k)}) > m_{k-1}$, and chooses $(\tilde k_J)_{J\in\mathcal{D}_k}$ in
  $\langle\{k_L:L\in\mathcal{D}_{n_k}\}\rangle$.  In round $k$ of $\mathrm{G}_2$, choose error $\eta_k' = \min_{i,j\leq k}\eta_{i,j}$,
  $\mathcal{F}_k^2 = \{D_{(\cdot,\tilde k_J)}:J\in\mathcal{D}_{k}\}$, and $M^2_k\subset N^2_{(k-1)}$ such that $n_{k}<\min(M^2_k)$.  Then, Player II
  chooses $N^2_{(k)}\subset M^2_k$, lets $m_k = \min(N^1_{(k)}) > n_{k}$, and chooses $(\tilde h_I)_{I\in\mathcal{D}_k}$ in
  $\langle\{h_K:K\in\mathcal{D}_{m_k}\}\rangle$.
  
  Let $\tilde H = (\tilde h_I)_{I\in\mathcal{D}}$ and $\tilde K = (\tilde k_J)_{J\in\mathcal{D}}$ and $\eta' = (\eta'_k)_{k=0}^\infty$.  Then, for
  every $k\in\mathbb{N}$ and for every $I\in\mathcal{D}_{k-1}$, $D_{(\tilde h_I,\cdot)}|_{\tilde H}$ is $\eta'$-stabilized from $k$ on, thus, by
  \Cref{restated triangular on blocking}~\ref{restated triangular on blocking upper}, $D_{\tilde H\otimes \tilde K}$ satisfies the upper triangular
  condition.  Also, for every $k\in\mathbb{N}_0$, $J\in\mathcal{D}_k$, $D_{(\cdot,\tilde k_J)}|_{\tilde H}$ is $\eta'$-stabilized from $k$ on.  By
  \Cref{restated triangular on blocking}~\ref{restated triangular on blocking lower} $D_{\tilde H\otimes \tilde K}$ satisfies the lower triangular
  condition.
\end{proof}

\subsection{Stabilizing two-parameter diagonal operators: the diagonal and superdiagonal conditions}
In order to obtain faithful Haar systems $\tilde H\otimes \tilde K=(\tilde h_I\otimes \tilde k_J: I,J\in\mathcal{D})$ for which
$D|_{\tilde H\otimes \tilde K}$ satisfies the ``diagonal conditions'' \ref{matrix:diag} and~\ref{matrix:sup} of \Cref{dfn:semi-stable}, we will need
the 2-parameter version of the probabilistic \Cref{L:1.4.3}.

\begin{lem}\label{L:3.3}
  Let $D\colon\mathcal{V}(\delta^2)\to\mathcal{V}(\delta^2)$ be an $\ell^{\infty}$-bounded operator.  Let $i<k$, $j<l$ be in $\mathbb{N}_0$ and let
  $\Gamma\in \sigma(\mathcal{D}_i)$ and $\Delta\in \sigma(\mathcal{D}_j)$.  For $\varepsilon=(\varepsilon_{I}: I\in \mathcal{D}_i, I\subset \Gamma)$
  and $\theta=(\theta_J: J\in \mathcal{D}_j, J\subset \Delta)$, $\omega=\pm1$ and $\xi=\pm1$ put
  \begin{align*}
    \Gamma^\omega(\varepsilon)
    = \Big\{ \sum_{I\in\mathcal{D}_i\atop I\subset \Gamma} \varepsilon_{I}h_I=\omega\Big\}
    \qquad\text{and}\qquad
    \Delta^\xi(\theta)
    =\Big\{\sum_{J\in \mathcal{D}_J\atop J\subset \Delta}\theta_J k_J=\xi\Big\}.
  \end{align*} 
  Let $\mathcal{E}=(\mathcal{E}_{I}: I\in\mathcal{D}_i, I\subset \Gamma)$, $\Theta=(\Theta_J: J\in\mathcal{D}_j, J\subset \Delta)$ be two independent
  Rademacher families on some probability space and define the random variable
  \begin{equation*}
    X_{\omega,\xi}
    = \sum_{K\in \mathcal{D}_k\atop K\subset\Gamma^\omega(\mathcal{E})} \frac{|K|}{|\Gamma^\omega(\mathcal{E})|}
    \sum_{L\in \mathcal{D}_l\atop L\subset \Delta^\xi(\Theta)} \frac{|L|}{|\Delta^\xi(\Theta)|} d_{K,L}.
  \end{equation*}
  Then it follows, for the expectation and variance of $X_{\omega,\xi}$,
  \begin{equation*}\label{E:3.3.1}
    \cond(X_{\omega,\xi})
    = \sum_{K\in \mathcal{D}_k\atop K\subset\Gamma} \frac{|K|}{|\Gamma|} \sum_{L\in \mathcal{D}_l\atop L\subset \Delta} \frac{|L|}{|\Delta|} d_{J,L}
    \qquad\text{and}\qquad
    \var(X_{\omega,\xi})
    \le 4 \|D\|^2_\infty \Big( \frac{2^{-i}}{|\Gamma|} + \frac{2^{-j}}{|\Delta|} \Big).
  \end{equation*}
\end{lem}

\begin{proof}
  Without loss of generality, we assume $\omega=\xi=1$ and write $X = X_{\omega,\xi}$.  It follows that $|\Gamma^+(\varepsilon)|=\frac{|\Gamma|}2$ and
  $|\Delta^+(\theta)|=\frac{|\Delta|}2$ for all $\varepsilon=(\varepsilon_I: I\in \mathcal{D}_i, I\subset \Gamma)$ and
  $\theta=(\theta_J: J\in \mathcal{D}_j, J\subset \Delta)$.  For $K\in\mathcal{D}_k$ and $L\in\mathcal{D}_l$ we introduce two auxiliary random
  variables $Y_{K}$ and $Z_{L}$.  Let $I\in\mathcal{D}_i$, $J\in\mathcal{D}_j$ with $K\subset I$ and $L\subset J$.  Define
  \begin{align*}
    Y_K
    &= \langle|K|^{-1}\chi_K,\chi_{I^{\mathcal{E}_I}}\rangle
      =
      \begin{cases}
        1& \text{if } K\subset \Gamma^+(\mathcal{E}),\\
        0& \text{if } K\cap \Gamma^+(\mathcal{E}) = \emptyset,
      \end{cases}\\
    Z_L
    &= \langle|L|^{-1}\chi_L,\chi_{J^{\Theta_J}}\rangle
      =
      \begin{cases}
        1& \text{if } L\subset \Delta^+(\Theta),\\
        0& \text{if } L\cap \Delta^+(\Theta) = \emptyset,
      \end{cases}
  \end{align*}
  which are independent and $\cond(Y_K) = \cond(Z_L) = 1/2$.
  
  Thus, we compute
  \begin{align*}
    \cond(X)
    &= \sum_{I\in\mathcal{D}_{i}\atop I\subset \Gamma} \frac{2|I|}{|\Gamma|} \sum_{J\in\mathcal{D}_{j}\atop J\subset \Delta}  \frac{2|J|}{|\Delta|}
      \cond\Big( \sum_{K\in\mathcal{D}_{k}\atop K\subset I\cap \Gamma^+(\mathcal{E})} \frac{|K|}{|I|}
      \sum_{L\in\mathcal{D}_{l}\atop L\subset J\cap\Delta^+(\Theta)} \frac{|L|}{|J|}
      d_{J,L} \Big)\\
    &= 4\sum_{I\in\mathcal{D}_{i}\atop I\subset \Gamma} \frac{|I|}{|\Gamma|}
      \sum_{J\in\mathcal{D}_{j}\atop J\subset \Delta} \frac{|J|}{|\Delta|}
      \sum_{K\in\mathcal{D}_k, K\subset I} \frac{|K|}{|I|}\sum_{L\in\mathcal{D}_l, L\subset J} \frac{|L|}{|J|} d_{K,L}\cond(Y_KZ_L)\\
    &= \sum_{K\in \mathcal{D}_k\atop K\subset\Gamma} \frac{|K|}{|\Gamma|} \sum_{L\in \mathcal{D}_l\atop L\subset \Delta} \frac{|L|}{|\Delta|} d_{K,L}.
  \end{align*} 

  Define for $I\in \mathcal{D}_i$, with $I\subset \Gamma$, and $J\in\mathcal{D}_j$, with $J\subset\Delta$, the random variable
  \begin{align*}
    X(I,J)
    &= 4 \sum_{K\in \mathcal{D}_k, \atop K\subset I\cap \Gamma^+(\mathcal{E})} \frac{|K|}{|\Gamma|}
      \sum_{L\in \mathcal{D}_l\atop L\subset J\cap \Delta^+(\Theta)} \frac{|L|}{|\Delta|}
      d_{K,L}\\
    &= 4\sum_{K\in \mathcal{D}_k, K\subset I} \frac{|K|}{|\Gamma|}
      \sum_{L\in \mathcal{D}_l, L\subset J} \frac{|L|}{|\Delta|}
      d_{K,L}Y_KZ_L
  \end{align*}
  which only depends on the random variables $\mathcal{E}_I$ and $\Theta_J$, by the definition of $Y_K$, $Z_L$.  Therefore, if
  $I\neq I'\in\mathcal{D}_i$, $I,I'\subset \Gamma$ and $J \neq J'\in\mathcal{D}_j$, $J,J'\subset \Delta$ then
  \begin{equation}
    \label{independent covariance}
    \mathrm{cov}(X(I,J),X(I',J')) = \cond((X(I,J)-\cond(X(I,J)))(X(I',J')-\cond(X(I',J')))) = 0.
  \end{equation}
  Furthermore, since
  \begin{equation*}
    |\{K\in \mathcal{D}_k: K\subset I\cap \Gamma^+(\mathcal{E})\}|
    = 2^l 2^{-i}/2
    \quad\text{and}\quad
    |\{L\in \mathcal{D}_l: L\subset J\cap \Delta^+(\Theta)\}|
    = 2^l2^{-j}/2,
  \end{equation*}
  we observe that
  \begin{equation*}
    |X(I,J)|\le 2^{-i-j}\frac1{|\Gamma|\cdot|\Delta| }\cdot\|D\|_\infty.
  \end{equation*}
  and therefore, for arbitrary $I,I'\in\mathcal{D}_i$, $I,I'\subset \Gamma$ and $J,J'\in\mathcal{D}_j$, $J,J'\subset \Delta$ we have
  \begin{equation}
    \label{dependent covariance}
    \bigl|\mathrm{cov}(X(I,J),X(I',J'))\bigr|
    \leq 4\frac{2^{-2i-2j}}{|\Gamma|^2\cdot|\Delta|^2}\|D\|_\infty^2.
  \end{equation}
  Denote
  \begin{equation*}
    \mathscr{A}
    = \Big\{(I,I',J,J')\in \mathcal{D}_i\times\mathcal{D}_i\times\mathcal{D}_j\times\mathcal{D}_j: I,I'\subset \Gamma, J,J'\subset \Delta\text{ and }I = I'\text{ or }J = J'\Big\}.
  \end{equation*}
  and note that $\#\mathscr{A} = |\Gamma|2^i\cdot |\Delta|2^j(|\Gamma|2^i + |\Delta|2^j - 1)$.  To conclude, we estimate
  \begin{align*}
    \var(X)
    &= \var\Big(\sum_{I\in \mathcal{D}_i\atop I\subset \Gamma}\sum_{J\in \mathcal{D}_j \atop J\subset \Delta} X(I,J) \ \Big)
      = \sum_{I,I'\in \mathcal{D}_i\atop I,I'\subset \Gamma}\sum_{J,J'\in \mathcal{D}_j \atop J,J'\subset \Delta}\mathrm{cov}\Big(X(I,J),X(I',J')\Big)\\
    &\leq \#\mathscr{A} \cdot 4\frac{2^{-2i-2j}}{|\Gamma|^2\cdot|\Delta|^2}\|D\|_\infty^2
      = 4\|D\|_\infty^2\frac{2^{-i-j}}{|\Gamma|\cdot|\Delta|} \Big(|\Gamma|2^i + |\Delta|2^j - 1\Big)\\
    &< 4 \|D\|^2_\infty \Big( \frac{2^{-i}}{|\Gamma|} + \frac{2^{-j}}{|\Delta|} \Big).
  \end{align*}
\end{proof}

\begin{rem}\label{bi-parameter probabilistic choice}
  Under the assumptions of \Cref{L:3.3} put
  \begin{align*}
    d_{\Gamma,\Delta}^{i,j}
    &= \sum_{I\in\mathcal{D}_i\atop I\subset\Gamma}\frac{|I|}{|\Gamma|}
      \sum_{J\in\mathcal{D}_k\atop J\subset\Delta} \frac{|J|}{|\Delta|}
      d_{I,J}\\
    d_{\Gamma,\Delta}^{k,l}
    &= \sum_{K\in\mathcal{D}_k\atop K\subset\Gamma} \frac{|K|}{|\Gamma|}
      \sum_{L\in\mathcal{D}_l\atop L\subset\Delta}  \frac{|L|}{|\Delta|}
      d_{K,L}
      = \cond(X_{\omega,\xi}).
  \end{align*}
  Then an argument similar to \Cref{simultaneous probabilistic choice} yields that, for any $\delta>0$,
  \begin{equation*}
    \cond\Big(
    \max_{\omega,\xi=\pm1}\Big|X_{\omega,\xi} - d_{\Gamma,\Delta}^{i,j} \Big|
    \leq \delta + \Big|d_{\Gamma,\Delta}^{i,j} -  d_{\Gamma,\Delta}^{k,l} \Big|
    \Big)
    \geq 1 - \frac{16}{\delta^2}\|D\|^2_\infty \Big(\frac{2^{-i}}{|\Gamma|} + \frac{2^{-j}}{|\Delta|}\Big).
  \end{equation*}

  Therefore, if $(\Gamma_\alpha)_{\alpha\in\mathcal{A}}$ are disjoint subsets in $\sigma(\mathcal{D}_{i})$ and $(\Delta_\beta)_{\beta\in\mathcal{B}}$
  are disjoint subsets in $\sigma(\mathcal{D}_j)$ and
  \begin{equation*}
    2^{-i}\Big(\sum_{\alpha\in\mathcal{A}}\frac{|\mathcal{B}|}{|\Gamma_\alpha|}\Big) + 2^{-j}\Big(\sum_{\beta\in\mathcal{B}}\frac{|\mathcal{A}|}{|\Delta_\beta|}\Big)
    < \frac{\delta^2}{16\|D\|_\infty^{2}},
  \end{equation*}
  there exist choices of signs $(\varepsilon_I:I\in\mathcal{D}_i, I\subset\Gamma_\alpha)$, $\alpha\in\mathcal{A}$, and
  $(\theta_J:J\in\mathcal{D}_j,J\subset \Delta_\beta)$, $\beta\in\mathcal{B}$, such that for any $(\alpha,\beta)\in\mathcal{A}\times\mathcal{B}$ and
  $\omega,\xi\in\{\pm1\}$ if we put
  \begin{equation*}
    \Gamma_\alpha^\omega
    = \Big\{ \sum_{I\in\mathcal{D}_i\atop I\subset \Gamma_\alpha} \varepsilon_Ih_I=\omega \Big\},
    \qquad \Delta_\beta^\xi
    = \Big\{ \sum_{J\in \mathcal{D}_j\atop J\subset \Delta_\beta}\theta_J k_J=\xi \Big\},
  \end{equation*}
  it follows that
  \begin{equation}
    \Big|
    \sum_{K\in\mathcal{D}_k\atop K\subset\Gamma_\alpha^\omega} \frac{|K|}{|\Gamma_{\alpha}^{\omega}|}
    \sum_{L\in\mathcal{D}_l\atop L\subset\Delta_\beta^\xi} \frac{|L|}{|\Delta_{\beta}^{\xi}|}
    d_{K,L} - d^{i,j}_{\Gamma_\alpha,\Delta_\beta}
    \Big|
    \leq \delta + \Big| d_{\Gamma_\alpha,\Delta_\beta}^{i,j} -  d_{\Gamma_\alpha,\Delta_\beta}^{k,l} \Big|.
  \end{equation}
\end{rem}

\begin{rem}\label{ramsey mod}
  A straightforward modification of the proof of Ramsey's Theorem yields the following.  Let $M$, $N$ be infinite subset of $\mathbb{N}$ and assume
  that we are given a finite coloring of the set
  \begin{equation*} [M,N] := \{(m,n):m<n\text{ with }m\in M, n\in N\}.
  \end{equation*}
  Then, there exist infinite $M'\subset M$, $N'\subset N$ such that $[M',N']$ is monochromatic.
\end{rem}

\begin{pro}\label{two-parameter subdiagonal}
  Let $D\colon \mathcal{V}(\delta^2)\to\mathcal{V}(\delta^2)$ be an $\ell^{\infty}$-bounded diagonal operator, $M$, $N$ be infinite subsets of
  $\mathbb{N}$, $\eta = (\eta_{i,j})_{i,j=0}^\infty\subset(0,1)$.  Then, there exist faithful Haar systems $\tilde H = (h_I)_{I\in\mathcal{D}}$,
  $\tilde K = (k_J)_{J\in\mathcal{D}}$ relative to frequencies $(m_i)_{i=0}^\infty\subset M$ and $(n_j)_{i=0}^\infty\subset N$ with
  $1\leq n_i< m_i<n_{i+1}$, for $i\geq 0$, such that $D|_{\tilde H\otimes\tilde K}$ satisfies the superdiagonal condition for $\eta$.
\end{pro}

\begin{proof}
  For each $(\Gamma,\Delta)\in\sigma(\mathcal{D}_{i_0})\times\sigma(\mathcal{D}_{j_0})$ and $i\geq i_0$, $j\geq j_0$ define
  \begin{equation*}
    d^{i,j}_{\Gamma,\Delta}
    = \sum_{I\in\mathcal{D}_i\atop I\subset \Gamma} \frac{|I|}{|\Gamma|}
    \sum_{J\in\mathcal{D}_j\atop J\subset \Delta} \frac{|J|}{|\Delta|}
    d_{I,J}.
  \end{equation*}
  Take $n_0\in N$ arbitrary and $\tilde k_{[0,1)} = \sum_{L\in\mathcal{D}_{n_0}}k_L$.  We will perform an induction on $k\in\mathbb{N}_0$ and in each
  step we will choose $m_k$, $n_{k+1}$, $M_k$, $N_{k+1}$, $(\tilde h_I)_{I\in\mathcal{D}_k}$, and $(\tilde k_J)_{J\in\mathcal{D}_{k+1}}$ such that the
  following are satisfied.
  \begin{enumerate}[label=(\alph*)]
  \item\label{subdiagonal refined} $M\supset M_0\supset\cdots\supset M_{k}$ and $N\supset N_0\supset\cdots\supset N_{k+1}$.

  \item\label{subdiagonal intertwined} $n_0<m_0<n_1<\cdots<m_k<n_{k+1}$ and $m_k =\min(M_k)$, $n_{k+1} = \min(N_{k+1})$.

  \item\label{subdiagonal the ks} If $k = 0$, for $I = [0,1)$, there are $(\varepsilon_K:K\in\mathcal{D}_{m_0})\in\{\pm1\}^{\mathcal{D}_{m_0}}$ such
    that
    \begin{equation*}
      \tilde h_{[0,1)}
      = \sum_{K\in\mathcal{D}_{m_0}}\varepsilon_Kh_K
    \end{equation*}
    whereas if $k\geq 1$, for $I\in\mathcal{D}_{k}$, with $I = \pi(I)^\omega$, for some $\omega\in\{\pm1\}$, and
    $\Gamma_I = \{\tilde h_{\pi(I)} = \omega\}$, then there are $(\varepsilon_K:K\in\mathcal{D}_{m_k},\;K\subset \Gamma_I)$ in $\{\pm1\}$ such that
    \begin{equation*}
      \tilde h_I
      = \sum_{K\in\mathcal{D}_{m_k}\atop K\subset\Gamma_I}\varepsilon_Kh_K.
    \end{equation*}

  \item\label{subdiagonal the hs} For $J\in\mathcal{D}_{k+1}$ with $J = \pi(I)^\xi$, for some $\xi\in\{\pm1\}$, and
    $\Delta_J = \{\tilde k_{\pi(J)} = \xi\}$, there are $\{\theta_L:L\in\mathcal{D}_{n_{k+1}},\;L\subset\Delta_J\}$ in $\{\pm1\}$ such that
    \begin{equation*}
      \tilde k_J
      = \sum_{L\in\mathcal{D}_{n_{k+1}}\atop L\subset\Delta_J}\theta_Lk_L.
    \end{equation*}

  \item\label{subdiagonal stability} For $I\in\mathcal{D}_{k}$, $J\in\mathcal{D}_{k+1}$, $\omega,\xi\in\{\pm1\}$, and for any
    $(i,j)\in[M_k\setminus\{m_k\},N_{k+1}\setminus\{n_{k+1}\}]$,
    \begin{equation*}
      \Big|d^{m_{k},n_{k+1}}_{\Gamma_I,\Delta_J} - d^{i,j}_{\{\tilde h_I=\omega\},\{\tilde k_J=\xi\}}\Big| < \eta_{k,k+1}.
    \end{equation*} 
  \end{enumerate}
  Let us quickly observe that these conditions yield the desired conclusion.  Indeed, by \ref{subdiagonal the ks} and \ref{subdiagonal the hs},
  $\tilde H = (\tilde h_I)_{I\in\mathcal{D}}$ and $\tilde K = (\tilde k_L)_{L\in\mathcal{D}}$ are faithful Haar systems relative to the frequencies
  $(m_i)_{i=0}^\infty$ and $(n_j)_{j=0}^\infty$ respectively and that $D|_{\tilde H\otimes\tilde K}$ is the diagonal operator with entries
  $((d^{m_i,n_j}_{I,J})_{I\in\mathcal{D}_{i},J\in\mathcal{D}_j})_{i,j=0}^\infty$.  Finally, by \ref{subdiagonal stability} and taking $i = m_{k+1}$,
  $l = n_{k+2}$, we obtain that for all $k\in\mathbb{N}_0$, $I\in\mathcal{D}_{k}$, $J\in\mathcal{D}_{k+1}$, and $\omega,\xi\in\{\pm1\}$,
  $|\tilde d_{I,J} - \tilde d_{I^\omega,L^\xi}| < \eta_{k,k+1}$.
  
  We will only describe the basis step, as it is the same as the general one.  Put $\Delta_{[0,1/2)} = \{\tilde k_{[0,1)}=1\}$ and
  $\Delta_{[1/2,1)} = \{\tilde k_{[0,1)}=-1\}$.  Using \Cref{ramsey mod}, pass to infinite $N'_1\subset N$, and $M'_0\subset M$ such that for any
  $(i,j)$, $(l,m)\in[M_0',N_1']$ and $J \in\mathcal{D}_1$ we have
  \begin{equation*}
    |d^{i,j}_{[0,1),\Delta_J} - d^{l,m}_{[0,1),\Delta_J}|
    < \eta_{0,1}/2.
  \end{equation*}
  Pick $m_0 \in M$ and $n_1 \in N$ sufficiently large to have
  \begin{equation*}
    2^{-m_0}\cdot 2+2^{-n_1}\cdot 2
    < \frac{1}{16} \|D\|_\infty^{-2} \big(\eta_{0,1}/2\big)^2.
  \end{equation*}
  Let $M_0'' = (m_0,\infty)\cap M_0'$ and $N_1'' = (n_1,\infty)\cap N_1'$.  By \Cref{bi-parameter probabilistic choice}, for any
  $(k,l)\in[M_0'',N_1'']$ we can find
  \begin{equation*}
    \varepsilon^{(k,l)}
    = (\varepsilon^{(k,l)}_K:K\in\mathcal{D}_{m_0},K\subset[0,1)),\text{ and  }\theta^{(k,l)}
    = (\theta^{(k,l)}_L:L\in\mathcal{D}_{n_1},L\subset\Delta_J, J\in\mathcal{D}_1)
  \end{equation*}
  such that for $J\in\mathcal{D}_1$ and $\omega,\xi\in\{\pm1\}$, if we put
  \begin{equation*}
    \Gamma_{[0,1)}^\omega(\varepsilon^{(k,l)})
    = \Big\{\sum_{K\in\mathcal{D}_{m_0}\atop K\subset[0,1)}\varepsilon^{(k,l)}_Lh_L = \omega\Big\}
    \quad\text{and}\quad
    \Delta_{J}^\xi(\theta^{(k,l)})
    = \Big\{\sum_{L\in\mathcal{D}_{n_1}\atop K\subset \Delta_J}\theta^{(k,l)}_L k_L = \xi\Big\}
  \end{equation*}
  then
  \begin{equation}
    \big| d^{m_1,n_0}_{\Gamma_I,[0,1)} - d^{j,l}_{\Gamma_I^\varepsilon(\theta_{(l,j)}), \Lambda_{[0,1)}^\delta(\phi_{(l,j)})} \big|
    < \frac{\eta_{1,0}}{2} + \big|d^{m_1,n_0}_{\Gamma_I,[0,1)} - d^{j,l}_{\Gamma_I,[0,1)}\big|
    < \eta_{1,0}.
  \end{equation}
  By \Cref{ramsey mod}, we may choose infinite $\tilde M_0\subset M_0''$ and $\tilde N_1\subset N_1''$ such that on $[\tilde M_0,\tilde N_1]$,
  $\varepsilon^{(k,l)}$ and $\theta^{(k,l)}$ are constant and equal to $\varepsilon$ and $\theta$.  Put $M_0 = \{m_0\}\cup \tilde M_0$,
  $N_1 = \{n_1\}\cup\tilde N_1$, $\tilde h_{[0,1)} = \sum_{K\in\mathcal{D}_{m_0}\atop K\subset[0,1)}\varepsilon_Lh_L$, and for $J\in\mathcal{D}_1$,
  $\tilde k_J = \sum_{L\in\mathcal{D}_{n_1}\atop J\subset\Delta_J}\theta_Jk_J$.  Then, all desired properties are satisfied.
\end{proof}

The proof of the following proposition is performed along the same lines as the proof of \Cref{two-parameter subdiagonal}.  The important difference
is that in each inductive step $k$ one works with $n_k$, $m_k$ (instead of $m_k$, $n_{k+1}$).
\begin{pro}
  \label{two-parameter diagonal}
  Let $D\colon\mathcal{V}(\delta^2)\to\mathcal{V}(\delta^2)$ be an $\ell^{\infty}$-bounded diagonal operator, $M$, $N$ be infinite subsets of
  $\mathbb{N}$, $\eta = (\eta_{i,j})_{i,j=0}^\infty\subset(0,1)$.  Then, there exist faithful Haar systems $\tilde H = (h_I)_{I\in\mathcal{D}}$,
  $\tilde K = (k_J)_{J\in\mathcal{D}}$ relative to frequencies $(m_i)_{i=0}^\infty\subset M$ and $(n_i)_{i=0}^\infty\subset N$, such that
  $D|_{\tilde H\otimes\tilde K}$ satisfies the diagonal condition for $\eta$.
\end{pro}

\begin{pro}
  \label{two-parameter balancing}
  Let $D\colon\mathcal{V}(\delta^2)\to\mathcal{V}(\delta^2)$ be an $\ell^{\infty}$-bounded diagonal operator, $M$, $N$ be infinite subsets of
  $\mathbb{N}$, $\delta > 0$.  Then, there exist faithful Haar systems $\tilde H = (h_I)_{I\in\mathcal{D}}$, $\tilde K = (k_J)_{J\in\mathcal{D}}$
  relative to frequencies $(m_i)_{i=0}^\infty\subset M$ and $(n_j)_{j=0}^\infty\subset N$, such that $D|_{\tilde H\otimes\tilde K}$ satisfies the
  balancing condition for $\delta$.
\end{pro}

\begin{proof}
  Fix $1\leq n_0 < m_0 < n_1$, such that $m_0\in M$, $n_0,n_1\in N$, and
  \begin{equation*}
    2^{-n_0}
    < \frac{\delta^2}{8\|D\|_\infty^2}.
  \end{equation*}
  Define $\tilde h_{[0,1)} = \sum_{K\in\mathcal{D}_{m_0}}h_K$, for $L\in\mathcal{D}_{n_1}$ let
  \begin{equation*}
    D_L
    = \sum_{K\in\mathcal{D}_{m_0}}|K|d_{K,L},
  \end{equation*}
  and note that $|D_J|\leq \|D\|_\infty$.
  
  For $\theta = (\theta_L)_{L\in\mathcal{D}_{n_0}}\in\{\pm1\}^{\mathcal{D}_{n_0}}$ define
  \begin{equation*}
    \tilde k_{[0,1)}^\theta
    = \sum_{L\in\mathcal{D}_{n_0}}\theta_Lk_L,
  \end{equation*}
  \begin{align*}
    \Delta_{[0,1/2)}(\theta)
    &= \{\tilde k_{[0,1)}^\theta = 1\},&\tilde k_{[0,1/2)}^\theta
    &= \sum_{L\in\mathcal{D}_{n_1}\atop K\subset \Delta_{[0,1/2)}(\theta)}k_L,\\
    \Delta_{[1/2,1)}(\theta)
    &= \{\tilde k_{[0,1)}^\theta = -1\},&\tilde k_{[1/2,1)}^\theta
    &= \sum_{L\in\mathcal{D}_{n_1}\atop L\subset \Delta_{[1/2,1)}(\theta)}k_L,
  \end{align*}
  \begin{align*}
    d_{[0,1),[0,1/2)}(\theta)
    &= 2\cdot \sum_{K\in\mathcal{D}_{m_0}\atop K\subset[0,1)} |K|
      \sum_{L\in\mathcal{D}_{n_1}\atop L\subset\Delta_{[0,1/2)}} |L| d_{K,L}
      = 2\cdot \sum_{L\in\mathcal{D}_{n_1}\atop L\subset\Delta_{[0,1/2)}(\theta)} |L| D_L,\text{ and}\\
    d_{[0,1),[0,1/2)}(\theta)
    &= 2\cdot \sum_{K\in\mathcal{D}_{m_0}\atop K\subset[0,1)} |K|
      \sum_{L\in\mathcal{D}_{n_1}\atop L\subset\Delta_{[1/2,1)}} |L| d_{K,L}
      = 2\cdot \sum_{L\in\mathcal{D}_{n_1}\atop L\subset\Delta_{[1/2,1)}(\theta)} |L| D_L.
  \end{align*}

  The goal is to pick $\theta$ such that $|d_{[0,1),[0,1/2)}(\theta) - d_{[0,1),[1/2,1)}(\theta)| < \delta$.  By \Cref{L:1.4.3}, we obtain
  \begin{equation*}
    \cond\big(d_{[0,1),[0,1/2)}(\theta)\big)
    = \cond\big(d_{[0,1),[1/2,1)}(\theta)\big)
    = \sum_{K\in\mathcal{D}_{m_0}} |K|
    \sum_{L\in\mathcal{D}_{n_1}} |L| d_{K,L}
  \end{equation*}
  as well as
  \begin{equation*}
    \var\big(d_{[0,1),[0,1/2)}(\theta)\big)\vee
    \var\big(d_{[0,1),[1/2,1)}(\theta)\big)
    \leq 2^{-n_0}\max_{L\in\mathcal{D}_{n_1}}|D_L|^2
    \leq 2^{-n_0}\|D\|_\infty^2.
  \end{equation*}
  The choice of $n_0$ and Chebyshev's inequality yield the desired $\theta$.  Put $\tilde k_{[0,1)} = \tilde k_{[0,1)}^\theta$,
  $\tilde k_{[0,1/2)} = \tilde k_{[0,1/2)}^\theta$, and $\tilde k_{[1/2,1)} = \tilde k_{[1/2,1)}^\theta$.  The choice of the remaining components of
  $\tilde H$, $\tilde K$ only needs to respect the restriction $n_k <m_k<n_{k+1}$ and $m_k\in M$, $n_k\in N$.
\end{proof}

\begin{proof}[Proof of \Cref{T:3.4}]
  Let $D\colon\mathcal{V}(\delta^2)\to\mathcal{V}(\delta^2)$ be an $\ell^{\infty}$-bounded diagonal operator and denote
  \begin{align*}
    \lambda(D)
    &= \lim_{n\to\mathcal{U}} \lim_{m\to\mathcal{U}} \Big( \sum_{(I,J)\in\mathcal{D}_m\times\mathcal{D}_n} |I| |J| d_{I,J}\Big),\\
    \mu(D)
    &= \lim_{m\to\mathcal{U}} \lim_{n\to\mathcal{U}} \Big( \sum_{(I,J)\in\mathcal{D}_m\times\mathcal{D}_n} |I| |J| d_{I,J}\Big).
  \end{align*}
  For $\delta>0$, there exist infinite $M,N\subset \mathbb{N}$ (not necessarily in $\mathcal{U}$) such that
  \begin{enumerate}
  \item for every $(n,m)\in[N,M]$, $\Big|\sum_{(I,J)\in\mathcal{D}_m\times\mathcal{D}_n} |I| |J| d_{I,J} - \lambda(D)\Big| < 1/n$,

  \item and for every $(m,n)\in[M,N]$, $\Big|\sum_{(I,J)\in\mathcal{D}_m\times\mathcal{D}_n} |I| |J| d_{I,J} - \mu(D)\Big| < 1/m$.
  \end{enumerate}
  We first apply \Cref{stabilization by one-parameter reduction}, to find faithful Haar systems $\widetilde H^{(1)}$ and $\widetilde K^{(1)}$ relative
  to frequencies $(m_i^{(1)})_{i=0}^\infty\subset M$ and $(n_j^{(1)})_{i=0}^\infty\subset N$, such that
  $D|_{\widetilde H^{(1)} \otimes \widetilde K^{(1)}}$ satisfies the upper and lower triangular conditions for appropriate $\eta$.  We next apply, in
  that order \Cref{two-parameter subdiagonal}, \ref{two-parameter diagonal}, and \ref{two-parameter balancing} to find
  $\tilde H = \widetilde H^{(4)} \ast \widetilde H^{(3)} \ast \widetilde H^{(2)}\ast \widetilde H^{(1)}$,
  $\tilde K = \widetilde K^{(4)} \ast \widetilde K^{(3)} \ast \widetilde K^{(2)}\ast \widetilde K^{(1)}$, such that $D_{\tilde H\otimes \tilde K}$ is
  $(\eta,\delta)$-stabilized.

  Regardless of the frequencies associated to $\widetilde H^{(i)}$, $\widetilde K^{(i)}$, $i=2,3,4$, by \Cref{blocking faithful haar}, $\tilde H$ is a
  faithful Haar system relative frequencies $(m_i)_{i=0}^\infty\subset M$, and $\tilde K$ is a faithful Haar system relative to frequencies
  $(n_j)_{j=0}^\infty\subset N$.  But then, by \Cref{P:3.2}
  \begin{align*}
    \lambda_\mathcal{U}(\tilde D)
    &= \lim_{j\to\infty} \lim_{i\to\infty} \sum_{I\in\mathcal{D}_i}\sum_{J\in\mathcal{D}_j}|I||J|\tilde d_{I,J}
      = \lim_{j\to\infty} \lim_{i\to\infty} \sum_{I\in\mathcal{D}_{m_i}}\sum_{J\in\mathcal{D}_{n_j}}|I||J| d_{I,J}
      = \lambda_\mathcal{U}(D)\text{ and}\\
    \mu_\mathcal{U}(\tilde D)
    &= \lim_{i\to\infty} \lim_{j\to\infty} \sum_{I\in\mathcal{D}_i}\sum_{J\in\mathcal{D}_j}|I||J|\tilde d_{I,J}
      = \lim_{i\to\infty} \lim_{j\to\infty} \sum_{I\in\mathcal{D}_{m_i}}\sum_{J\in\mathcal{D}_{n_j}}|I||J| d_{I,J}
      = \mu_\mathcal{U}(D).
  \end{align*}
\end{proof}

\section{Haar system Hardy spaces in two parameters}
\label{sec:haar-system-hardy-1+2}
We prove the basic important properties of Haar system Hardy spaces. In particular, we prove that they satisfy \Cref{two-parametric
  assumption}~\ref{two-parametric assumption a} for $C=1$, as well as the well boundedness of certain coordinate projections.  Before we do so, we
prove and collect some elementary facts about Haar system Hardy spaces.

Let $Z = Z(\boldsymbol{\sigma},X,Y)\in\mathcal{HH}(\delta^2)$, and for each $I,J\in\mathcal{D}$, let the linear functional
$\ell_{I,J}\colon Z\to\mathbb{R}$ be given by
\begin{equation*}
  \ell_{I,J}(z)
  = \int \int h_I(s)k_J(t) z(s,t)\mathrm{d} t\mathrm{d} s.
\end{equation*}

We define the space $Z'$ as the closure of $\langle \{ \ell_{I,J} : I,J\in\mathcal{D} \} \rangle$ in $Z^*$ and equip it with the inherited norm, i.e.,
\begin{equation}\label{eq:69}
  \bigl\|\sum_{I,J\in\mathcal{D}} a_{I,J}\ell_{I,J}\bigr\|_{Z'}
  = \sup_{\|z\|_Z\leq 1}
  \Bigl|\int \int \sum_{I,J\in\mathcal{D}} a_{I,J} h_I(s)k_J(t) z(s,t)\mathrm{d} t\mathrm{d} s\Bigr|.
\end{equation}
Now we define the scalar product $\langle\cdot,\cdot\rangle\colon Z'\times Z\to\mathbb{R}$ by first defining it as the linear extension of
\begin{equation}\label{eq:70}
  \langle \ell_{I,J}, z\rangle
  = \int \int h_I(s)k_J(t) z(s,t)\mathrm{d} t\mathrm{d} s,
  \qquad I,J\in\mathcal{D},\text{ and } z\in Z,
\end{equation}
and then uniquely extending it to all of $Z'$.  We will as usual identify each $\ell_{I,J}$ with $h_I\otimes k_J$ and we will write
$\langle h_I\otimes k_J, z\rangle$ instead of $\langle \ell_{I,J}, z\rangle$.

In the following \Cref{pro:scalar-product}, we calculate the norm of $\ell_{I,J}$.
\begin{pro}\label{pro:scalar-product}
  Let $I,J\in\mathcal{D}$ and recall that $\ell_{I,J}$ is defined in~\eqref{eq:70}.  Then we have that
  \begin{equation}\label{eq:71}
    \sup_{\|z\|_Z\leq 1} |\langle \ell_{I,J}, z\rangle|
    = \|\ell_{I,J}\|_{Z^*}
    = \frac{|I||J|}{\|h_I\|_X\|k_J\|_Y}.
  \end{equation}
\end{pro}

\begin{ntn}
  We define the bijective function $\iota\colon\mathcal{D}\to\mathbb{N}$ by
  \begin{equation*}
    \Big[\frac{i-1}{2^j},\frac{i}{2^j}\Big)
    \overset{\iota}{\mapsto} 2^j + i-1.
  \end{equation*}
  The function $\iota $ defines a linear order on $\mathcal{D}$ that maps each $\mathcal{D}_n$ to $\{2^n,\ldots,2^{n+1}-1\}$. We will sometimes write
  $I\leq J$ if $\iota(I)\leq\iota(J)$, $I,J\in{\mathcal{D}}$.
\end{ntn}
Henceforth, whenever we write $\sum_{I\in\mathcal{D}}$ we will always mean that the sum is taken with this linear order $\iota$.

\begin{proof}
  The lower estimate is easily established by considering $\langle\ell_{I,J},h_I\otimes k_J\rangle$.  For the upper estimate, let
  $z = \sum_{K,L\in\mathcal{D}} a_{K,L} h_K\otimes k_L\in \langle \{h_I\otimes k_J : I,J\in\mathcal{D}\} \rangle$ and define
  \begin{equation}\label{eq:68}
    b_K(t)
    = \sum_{L\in\mathcal{D}} \sigma_K^{(1)}\sigma_L^{(2)}a_{K,L} k_L(t),
    \qquad t\in [0,1),\ K\in \mathcal{D}.
  \end{equation}
  As in the proof of \Cref{pro:bi-monotone} (see~\eqref{eq:96} and~\eqref{eq:97}), the functions
  \begin{align*}
    s &\mapsto \Bigl\|
        t\mapsto \cond \Bigl|
        b_I(t) h_I(s) + \sum_{K < I} b_K(t) h_K(s)
        \Bigr|
        \Bigr\|_Y,\\
    s &\mapsto \Bigl\|
        t\mapsto \cond \Bigl|
        b_I(t) h_I(s) - \sum_{K < I} b_K(t) h_K(s)
        \Bigr|
        \Bigr\|_Y
  \end{align*}
  are equimeasurable, and therefore they have the same norm in $X$.  Thus, the convexity of the norm in $Z$ yields
  \begin{equation*}
    \bigl\| h_I\otimes b_I \bigr\|_Z
    \leq \frac{1}{2} \bigl\|h_I\otimes b_I + \sum_{K < I} h_K\otimes b_K \bigr\|_Z
    + \frac{1}{2} \bigl\|h_I\otimes b_I - \sum_{K < I} h_K\otimes b_K \bigr\|_Z
    = \bigl\|\sum_{K\leq I} h_K\otimes b_K\bigr\|_Z.
  \end{equation*}
  By \Cref{pro:bi-monotone} and recalling~\eqref{eq:68}, we obtain
  \begin{equation}\label{eq:72}
    \bigl\| h_I\otimes b_I \bigr\|_Z
    \leq \bigl\|\sum_{K\leq I} h_K\otimes b_K\bigr\|_Z
    \leq \bigl\|\sum_{K\in\mathcal{D}} h_K\otimes b_K\bigr\|_Z
    = \bigl\|\sum_{K,L\in\mathcal{D}} a_{K,L} h_K\otimes k_L\bigr\|_Z
    = \|z\|_Z.
  \end{equation}
  Using \Cref{pro:bi-monotone} again gives us
  \begin{equation}\label{eq:73}
    \|h_I\otimes b_I\|_Z
    = \|h_I\|_X
    \Bigl\|
    \cond \Bigl|
    \sum_{L\in\mathcal{D}} \sigma_L^{(2)} a_{I,L} k_L
    \Bigr|
    \Bigr\|_Y
    \geq \|h_I\|_X \Bigl\|
    \int_0^1 \Bigl|
    \sum_{L\leq J} \sigma_L^{(2)} a_{I,L} k_L
    \Bigr|
    \Bigr\|_Y.
  \end{equation}
  Since the functions
  \begin{equation*}
    \cond \Bigl|
    \sigma_J^{(2)} a_{I,J} k_J + \sum_{L < J} \sigma_L^{(2)} a_{I,L} k_L
    \Bigr|
    \qquad\text{and}\qquad
    \cond \Bigl|
    \sigma_J^{(2)} a_{I,J} k_J - \sum_{L < J} \sigma_L^{(2)} a_{I,L} k_L
    \Bigr|
  \end{equation*}
  are equimeasurable, thus, they have the same norm in $X$.  Using the convexity of the norm in $Y$ yields
  \begin{align*}
    |a_{I,J}| \|k_J\|_Y
    &= \|a_{I,J} k_J\|_Y
      \leq \frac{1}{2} \Bigl\|
      \int_0^1 \Bigl|
      \sigma_J^{(2)} a_{I,J} k_J + \sum_{L < J} \sigma_L^{(2)}(v) a_{I,L} k_L
      \Bigr|
      \Bigr\|_Y\\
    &\qquad\qquad\qquad\qquad\qquad + \frac{1}{2} \Bigl\| \cond \Bigl|
      \sigma_J^{(2)} a_{I,J} k_J - \sum_{L < J} \sigma_L^{(2)} a_{I,L} k_L
      \Bigr|
      \Bigr\|_Y\\
    &= \Bigl\|
      \cond \Bigl|
      \sum_{L\leq J} \sigma_L^{(2)} a_{I,L} k_L
      \Bigr|
      \Bigr\|_Y.
  \end{align*}
  Combining the latter estimate with~\eqref{eq:73} and~\eqref{eq:72} yields
  \begin{equation*}
    |a_{I,J}| \|h_I\|_X \|k_J\|_Y
    \leq \|z\|_Z,
  \end{equation*}
  and noting that $|\langle \ell_{I,J}, z\rangle| = |I| |J| |a_{I,J}|$ establishes~\eqref{eq:71}.
\end{proof}

\begin{rem}\label{rem:dual-Z}
  Let $Z = Z(\boldsymbol{\sigma},X,Y)\in\mathcal{HH}(\delta^2)$.  By \Cref{pro:scalar-product}, identifying $h_I\otimes k_J$ with $\ell_{I,J}$ in
  $Z'$, we can represent any $z'\in Z'$ as
  \begin{equation*}
    z' = \sum_{I,J\in\mathcal{D}} a_{I,J} h_I\otimes k_J,
  \end{equation*}
  where the convergence of the series is understood in the norm $\|\cdot\|_{Z'}$ given by
  \begin{equation*}
    \|z'\|_{Z'}
    = \sup_{\|z\|_Z\leq 1} \Bigl|\int \int z'(s,t) z(s,t)\mathrm{d} t\mathrm{d} s\Bigr|.
  \end{equation*}
  Altogether, $\langle Z', Z, \langle\cdot,\cdot\rangle\rangle$ is a dual pair of Banach spaces and by the definition of $Z'$ and
  \Cref{pro:bi-monotone} we have
  \begin{equation*}
    |\langle z', z\rangle|
    \leq \|z'\|_{Z'} \|z\|_Z,
    \ z'\in Z,\ z\in Z
    \qquad\text{and}\qquad
    \|z\|_Z
    = \sup_{\|z'\|_{Z'}\leq 1} \langle z', z\rangle,
    \ z\in Z.
  \end{equation*}
\end{rem}

\begin{rem}\label{rem:norm-product}
  If we define $X'$ and $Y'$ analogously to $Z'$, then by \Cref{pro:scalar-product} and the fact that $\|h_{[0,1)}\|_X = \|k_{[0,1)}\|_Y = 1$ (see
  \Cref{dfn:HS-1d}) we have
  \begin{equation*}
    \|h_I\otimes k_J\|_{Z'}\|h_I\otimes k_J\|_Z
    = |I| |J|,
    \qquad
    \|h_I\|_{X'}\|h_I\|_X
    = |I|,
    \qquad
    \|k_J\|_{X'}\|k_J\|_X
    = |J|,
  \end{equation*}
  for all $I,J\in\mathcal{D}$.
\end{rem}

\subsection{Basic operators on Haar system Hardy spaces}
\label{sec:haar-system-hardy-1}
Here, we collect estimates for bi-parameter basis projections, the sub-restriction operator and dyadic scaling operators.
\begin{pro}\label{pro:bi-monotone}
  Let $Z = Z(\mathbf{\sigma},X,Y)\in\mathcal{HH}(\delta^2)$.  Given $I\in\mathcal{D}$, we define the projection $P_{\leq I}, P_{\geq I}\colon H\to H$
  by
  \begin{equation*}
    P_{\leq I} \Bigl(\sum_{K\in\mathcal{D}} a_K h_K\Bigr)
    = \sum_{K\leq I} a_K h_K
    \qquad\text{and}\qquad
    P_{\geq I} \Bigl(\sum_{K\in\mathcal{D}} a_K h_K\Bigr)
    = \sum_{K\geq I} a_K h_K.
  \end{equation*}
  Then for all $I,J\in\mathcal{D}$, we have the estimates
  \begin{align*}
    \|P_{\leq I}\otimes P_{\leq J}\colon  Z\to Z\|
    &\leq 1,
    & \|P_{\geq I}\otimes P_{\geq J}\colon Z\to Z\|
    &\leq 4,\\
    \|P_{\leq I}\otimes P_{\geq J}\colon Z\to Z\|
    &\leq 2,
    &\|P_{\geq I}\otimes P_{\leq J}\colon Z\to Z\|
    &\leq 2.
  \end{align*}
\end{pro}

\begin{rem}\label{rem:cond}
  Given $I\in\mathcal{D}$, we introduce the following additional notation: We define $P_{< I} = 0$, if $I=[0,1)$, and $P_{< I} = P_{\leq J}$, if
  $\iota(J) = \iota(I) + 1$ and put $P_{\geq I} = I - P_{< I}$, $P_{> I} = I - P_{\leq I}$ as well as $p_I = P_{\leq I} - P_{< I}$.  Moreover, given
  $k\in\mathbb{N}$, we define $\cond_k = P_{\leq I}$, where $I\in\mathcal{D}$ is the largest (with respect to the standard linear order) dyadic
  interval with $|I|\geq 2^{-k+1}$.
  
  Thus, by \Cref{pro:bi-monotone}, the projections $P_{\geq I}$, $P_{> I}$, $p_I$ are all bounded by $2$ and $\cond_k$ is bounded by $1$ on any
  one-parameter Haar system Hardy space.  Additionally, we note that tensoring any two of these operators yields a projection on $Z$ bounded by either
  $1$, $2$ or $4$.
\end{rem}

\begin{proof}[Proof of \Cref{pro:bi-monotone}]
  It suffices to estimate on a vector $z = \sum_{K,L\in\mathcal{D}} a_{K,L} h_K\otimes k_L\in \mathcal{V}(\delta^2)$, with $a_{K,L}\in\mathbb{R}$.
  Observe that
  \begin{equation}\label{eq:96}
    \begin{aligned}
      \|P_{\leq I}\otimes P_{\leq J}z\|_{Z}
      &= \Bigl\|
        s\mapsto \Bigl\|
        t\mapsto\cond \Bigl|
        \sum_{K\leq I}\sum_{L\leq J} \sigma_K^{(1)}\sigma_L^{(2)} a_{K,L} h_K(s) k_L(t)
        \Bigr|
        \Bigr\|_Y
        \Bigr\|_X\\
      &= \Bigl\|
        s\mapsto \Bigl\|
        t\mapsto \cond \Bigl|
        \sum_{K\leq I} h_K(s) b_K(t)
        \Bigr|
        \Bigr\|_Y
        \Bigr\|_X
        =\bigl\|\sum_{K\leq I}h_K\otimes b_K\bigl\|_{Z},
    \end{aligned}
  \end{equation}
  where we defined
  \begin{equation}\label{eq:97}
    b_K(t)
    = \sum_{L\leq J} \sigma_K^{(1)}\sigma_L^{(2)} a_{K,L} k_L(t),
    \qquad t\in [0,1),\  K\in\mathcal{D}.
  \end{equation}
  Note that the functions
  \begin{align*}
    s&\mapsto \Bigl\|
       t\mapsto \cond \Bigl|
       \sum_{K < I} h_K(s)b_K(t)  +  h_I(s)b_I(t)
       \Bigr|
       \Bigr\|_Y,\\
    s&\mapsto \Bigl\|
       t\mapsto \cond \Bigl|
       \sum_{K < I} h_K(s)b_K(t)  - h_I(s)b_I(t) 
       \Bigr|
       \Bigr\|_Y
  \end{align*}
  are equimeasurable, and therefore they have the same norm in $X$.  Thus, by convexity of the norm in $Z$, \eqref{eq:96} and~\eqref{eq:97}, we obtain
  \begin{equation*}
    \begin{aligned}
      \bigl\|\sum_{K < I} h_K\otimes b_K \bigr\|_{Z}
      &\leq \frac{1}{2} \bigl\|\sum_{K < I} h_K\otimes b_K  + h_K\otimes b_I\bigr\|_{Z}
        + \frac{1}{2} \bigl\|\sum_{K < I}h_K\otimes b_K - h_K\otimes b_I \bigr\|_{Z}\\
      &= \bigl\|\sum_{K\leq I}h_K\otimes b_K\bigr\|_{Z}.
    \end{aligned}
  \end{equation*}
  Iterating the latter inequality yields
  \begin{equation}\label{eq:74}
    \begin{aligned}
      \|P_{\leq I}\otimes P_{\leq J}z\|_{Z}%
      &\leq \bigl\|\sum_{K\in\mathcal{D}^{1}}h_K\otimes b_K \bigr\|_{Z^\Omega}%
        = \bigl\|\sum_{K\in\mathcal{D}^{1}} \sum_{L\leq J} a_{K,L} h_K\otimes k_L  \bigr\|_{Z}%
        = \|I_X\otimes P_{\leq J} z\|_{Z}.%
    \end{aligned}
  \end{equation}

  Similarly, we note that
  \begin{equation}\label{eq:99}
    \begin{aligned}
      \|I_X\otimes P_{\leq J} w\|_{Z}
      &= \Bigl\|
        s\mapsto \int_0^1 \Bigl\|
        t\mapsto \cond \Bigl|
        \sum_{K\in\mathcal{D}} \sum_{L\leq J} \sigma_K^{(1)}\sigma_L^{(2)}a_{K,L} h_K(s) k_L(t) \Bigr|
        \Bigr\|_Y
        \Bigr\|_X\\
      &= \Bigl\|
        s\mapsto \Bigl\|
        t\mapsto \cond \Bigl|
        \sum_{L\leq J} c_L(s) k_L(t)
        \Bigr|
        \Bigr\|_Y
        \Bigr\|_X
        = \Bigl\| \sum_{L\leq J} k_L \otimes c_L  \Bigl\|_Z,
    \end{aligned}
  \end{equation}
  where we put
  \begin{equation}\label{eq:144}
    c_L(s)
    = \sum_{K\in\mathcal{D}} h_K(s) \sigma_K^{(1)}\sigma_L^{(2)}a_{K,L},
    \qquad s\in [0,1),\ L\in\mathcal{D}.
  \end{equation}
  
  Then for fixed $s\in [0,1)$, the functions
  \begin{align*}
    t\mapsto \cond \Bigl|
    \sum_{L < J}k_L(t) c_L(s) +k_J(t) c_J(s) 
    \Bigr|,\quad
    t\mapsto \cond \Bigl|
    \sum_{L < J} k_L(t) c_L(s) - k_J(t)c_J(s) 
    \Bigr|,
  \end{align*}
  are equimeasurable and therefore have the same norm in $Y$.  Again, using the convexity of the norm in $Z$ yields
  \begin{equation*}
    \bigl\|\sum_{L < J} c_L k_L\bigr\|_Z
    \leq \frac{1}{2} \bigl\|\sum_{L < J} k_L\otimes c_L + k_J\otimes c_J\bigr\|_Z
    + \frac{1}{2} \bigl\|\sum_{L < J} k_L\otimes c_L - k_J\otimes c_J\bigr\|_Z
    \leq \bigl\|\sum_{L\leq J} k_L\otimes c_L\bigr\|_Z.
  \end{equation*}
  Recalling~\eqref{eq:74}, \eqref{eq:99} and iterating the above inequality, we obtain
  \begin{equation*}
    \|P_{\leq I}\otimes P_{\leq J}z\|_Z
    \leq \bigl\|\sum_{L\in\mathcal{D}} k_L\otimes c_L\bigr\|_Z
    = \bigl\|
    \sum_{K,L\in\mathcal{D}} a_{K,L}h_K\otimes k_L
    \bigr\|_Z
    = \|z\|_Z.
  \end{equation*}

  The estimate for $I_X\otimes P_{\leq J}$ follows by choosing $I\in\mathcal{D}$ so large that
  $I_X\otimes P_{\leq J} z = P_{\leq I}\otimes P_{\leq J} z$, and the estimate for $P_{\leq I}\otimes I_Y$ follows by choosing $J\in\mathcal{D}$ so
  large that $P_{\leq I}\otimes I_Yz = P_{\leq I}\otimes P_{\leq J}z$.  The other estimates follow immediately from the above estimates.
\end{proof}

\begin{pro}\label{pro:restriction-operators}
  Let $Z = Z(\mathbf{\sigma},X,Y)\in\mathcal{HH}(\delta^2)$ and let $K_0,L_0\in\mathcal{D}$.  Define the sub-restriction operator
  $S_{K_0, L_0}\colon Z\to Z$ by
  \begin{equation*}
    S_{K_0, L_0} \Bigl(
    \sum_{K,L\in\mathcal{D}} a_{K,L} h_K\otimes k_L
    \Bigr)
    = \sum_{\substack{K\subset K_0\\L\subset L_0}}
    a_{K,L} h_K\otimes k_L.
  \end{equation*}
  Then $S_{K_0, L_0}$ is a Haar multiplier with $\|S_{K_0, L_0}\colon Z\to Z\| \leq 4$.
\end{pro}

\begin{proof}
  Consider $z = \sum_{K,L\in\mathcal{D}} a_{K,L} h_K\otimes k_L\in \mathcal{V}(\delta^2)$ with $a_{K,L}\in\mathbb{R}$.  Note that
  \begin{equation*}
    \|S_{K_0, L_0} z\|_{Z}
    = \Bigl\|
    s\mapsto \Bigl\|
    t\mapsto \cond \Bigl|
    \sum_{\substack{K\subset K_0\\L\subset L_0}} \sigma_K^{(1)}\sigma_L^{(2)}a_{K,L} h_K(s) k_L(t)
    \Bigr|
    \Bigr\|_Y
    \Bigr\|_X.
  \end{equation*}
  Suppose that $|K_0| = 2^{-k_0}$ and $|L_0| = 2^{-l_0}$, then for fixed $s,t\in [0,1)$, we have
  \begin{equation*}
    \Bigl|
    \sum_{\substack{K\subset K_0\\L\subset L_0}} \sigma_K^{(1)}\sigma_L^{(2)}a_{K,L} h_K(s) k_L(t)
    \Bigr|
    \leq \Bigl|
    \sum_{\substack{|K|\leq 2^{-k_0}\\|L|\leq 2^{-l_0}}} \sigma_K^{(1)}\sigma_L^{(2)}a_{K,L} h_K(s) k_L(t)
    \Bigr|.
  \end{equation*}
  Taking the expectation $\cond$ on both sides of the above inequality and using \Cref{RI properties}~\ref{RI properties 5} twice yields
  \begin{align*}
    \|S_{K_0, L_0}z\|_{Z}
    &\leq \Bigl\|
      s\mapsto \Bigl\|
      t\mapsto \cond \Bigl|
      \sum_{\substack{|K|\leq 2^{-k_0}\\|L|\leq 2^{-l_0}}} \sigma_K^{(1)}\sigma_L^{(2)}a_{K,L} h_K(s) k_L(t)
    \Bigr|
    \Bigr\|_Y
    \Bigr\|_X\\
    &= \Bigl\|
      \sum_{\substack{|K|\leq 2^{-k_0}\\|L|\leq 2^{-l_0}}} a_{K,L}\otimes h_K\otimes k_L
    \Bigr\|_{Z}
    = \bigl\| \bigl((I_X - \cond_{k_0})\otimes (I_Y - \cond_{l_0})\bigr) z \bigr\|_{Z}.
  \end{align*}
  Applying \Cref{pro:bi-monotone} (see \Cref{rem:cond} for the definition of $\cond_{k_0}$, $\cond_{l_0}$) proves the estimate for $S_{K_0, L_0}$.
\end{proof}

\begin{pro}\label{pro:down-up-scale}
  Let $Z = Z(\mathbf{\sigma},X,Y)\in\mathcal{HH}(\delta^2)$, $I_0,J_0\in\mathcal{D}$ with $|I_0| = |J_0| = 2^{-l_0}$, and let
  $\rho_0\colon [0,1)\to I_0$ and $\tau_0\colon [0,1)\to J_0$ denote the unique increasing, affine linear and bijective functions.  Then the operators
  $\Delta_{I_0, J_0}, \Upsilon_{I_0, J_0}\colon Z\to Z$ defined by
  \begin{align*}
    \Delta_{I_0, J_0} \Bigl( \sum_{I,J\in\mathcal{D}} a_{I,J} h_I\otimes k_J\Bigr)
    &= \sum_{I,J\in\mathcal{D}} a_{I,J} h_{\rho_0(I)}\otimes k_{\tau_0(J)},\\
    \Upsilon_{I_0, J_0} \Bigl( \sum_{I,J\in\mathcal{D}} a_{I,J} h_I\otimes k_J\Bigr)
    &= \sum_{\substack{I\subset I_0\\J\subset J_0}}
    a_{I,J} h_{\rho_0^{-1}(I)}\otimes k_{\tau_0^{-1}(J)},
  \end{align*}
  satisfy the estimates
  \begin{equation*}
    \|\Delta_{I_0, J_0} z\|_Z
    \leq \|z\|_Z
    \quad\text{and}\quad
    \|\Upsilon_{I_0, J_0} z\|_Z
    \leq 4^{l_0} \|S_{I_0, J_0} z\|_Z
    \leq 4^{l_0+1} \|z\|_Z,
    \qquad z\in Z.
  \end{equation*}
  Moreover, we have that $\Upsilon_{I_0, J_0} \Delta_{I_0, J_0} = I_Z$ and $S_{I_0, J_0}\Delta_{I_0, J_0} = \Delta_{I_0, J_0}$.
\end{pro}

\begin{proof}
  Let $N\in\mathbb{N}$ and $z = \sum_{\substack{I,J\in\mathcal{D}\\|I|,|J|\geq 2^{-N+1}}} a_{I,J} h_I\otimes k_J$.
  \begin{proofcase}[Estimate for $\Delta_{I_0, J_0}$]
    For $s,t s,t\in [0,1)$ we define
    \begin{align*}
      f_1(s,t)
      &= \cond \Bigl|
        \sum_{I,J\in\mathcal{D}}
        \sigma_I^{(1)}\sigma_J^{(2)}a_{I,J} h_I(s) k_J(t)
        \Bigr|,
        \text{ and }  f_2(s,t)
      &= \cond \Bigl|
        \sum_{I,J\in\mathcal{D}} \sigma_{\rho_0(I)}^{(1)}\sigma_{\tau_0(J)}^{(2)}a_{I,J} h_I(s) k_J(t)
        \Bigr|,
    \end{align*}
    and observe that
    \begin{equation*}
      \|z\|_Z
      = \Bigl\|
      s\mapsto \Bigl\|
      t\mapsto f_1(s,t)
      \Bigr\|_Y
      \Bigr\|_X
      = \Bigl\|
      s\mapsto \Bigl\|
      t\mapsto f_2(s,t)
      \Bigr\|_Y
      \Bigr\|_X.
    \end{equation*}
    Now, we pick a measurable set $A$ such that $|A\cap I| = 2^{-l_0} |I|$ for all $I\in\mathcal{D}$ with $|I|\geq 2^{-N}$ and observe that by
    \Cref{RI properties}~\ref{RI properties 5}.
    \begin{equation*}
      \|z\|_Z
      \geq \Bigl\|
      s\mapsto \Bigl\|
      t\mapsto \chi_A(t)f_2(s,t)
      \Bigr\|_Y
      \Bigr\|_X.
    \end{equation*}
    Next, we define the function
    \begin{equation*}
      f_3(s,t)
      = \cond \Bigl|
      \sum_{I,J\in\mathcal{D}} \sigma_{\rho_0(I)}^{(1)}\sigma_{\tau_0(J)}^{(2)}a_{I,J} h_{I}(s) k_{\tau_0(J)}(t)
      \Bigr|,
      \qquad s,t\in [0,1)
    \end{equation*}
    and note that for fixed $s$ the functions $t\mapsto \chi_A(t) f_2(s,t)$ and $t\mapsto f_3(s,t)$ are equimeasurable.  Thus, by \Cref{RI
      properties}~\ref{RI properties 5}
    \begin{equation*}
      \|z\|_Z
      \geq \Bigl\|
      s\mapsto \Bigl\|
      t\mapsto f_3(s,t)
      \Bigr\|_Y
      \Bigr\|_X
      \geq \Bigl\|
      s\mapsto \int_0^1 \Bigl\|
      t\mapsto \chi_A(s) f_3(s,t)
      \Bigr\|_Y
      \Bigr\|_X.
    \end{equation*}
    Next, we define for $s\in [0,1)$
    \begin{align*}
      f_4(s)
      &= \Bigl\|
        t\mapsto \chi_A(s) f_3(s,t)
        \Bigr\|_Y,
        \text{ and }
        f_5(s)
        = \Bigl\|
        t\mapsto \cond \Bigl|
        \sum_{I,J\in\mathcal{D}} \sigma_{\rho_0(I)}^{(1)}\sigma_{\tau_0(J)}^{(2)}a_{I,J} h_{\rho_0(I)}(s) k_{\tau_0(J)}(t)
        \Bigr|
        \Bigr\|_Y,
    \end{align*}
    and observe that $f_4$ and $f_5$ are equimeasurable, which yields
    \begin{equation*}
      \|z\|_Z
      \geq \|f_5\|_X.
    \end{equation*}
    Finally, we note that $\|f_5\|_X = \|\Delta_{I_0, J_0} z\|_Z$ and conclude $\|\Delta_{I_0, J_0} z\|_Z\leq \|z\|_Z$.
  \end{proofcase}

  \begin{proofcase}[Estimate for $\Upsilon_{I_0, J_0}$]
    First, we pick a measurable partition $A_k$, $1\leq k \leq 2^{l_0}$, of $[0,1)$ such that $|A_k\cap I| = 2^{-l_0} |I|$ for all $I\in\mathcal{D}$
    with $|I|\geq 2^{-l_0}$ and $1\leq k \leq 2^{l_0}$.  Secondly, we define the functions
    \begin{align*}
      g_1(s,t)
      &= \cond \Bigl|
        \sum_{\substack{I\subset I_0\\J\subset J_0}} \sigma_{\rho_0^{-1}(I)}^{(1)}\sigma_{\tau_0^{-1}(J)}^{(2)} a_{I,J} h_{\rho_0^{-1}(I)}(s) k_{\tau_0^{-1}(J)}(t)
      \Bigr|,
      && s,t\in [0,1)\\
      g_2(s,t)
      &= \cond \Bigl|
        \sum_{\substack{I\subset I_0\\J\subset J_0}} \sigma_{I}^{(1)}\sigma_{J}^{(2)} a_{I,J} h_{\rho_0^{-1}(I)}(s) k_{\tau_0^{-1}(J)}(t)
      \Bigr|,
      && s,t\in [0,1)
    \end{align*}
    and note that
    \begin{equation*}
      \|\Upsilon_{I_0, J_0}z\|_Z
      = \Bigl\|
      s\mapsto \Bigl\|
      t\mapsto g_1(s,t)
      \Bigr\|_Y
      \Bigr\|_X
      = \Bigl\|
      s\mapsto \Bigl\|
      t\mapsto g_2(s,t)
      \Bigr\|_Y
      \Bigr\|_X.
    \end{equation*}
    Since the $A_k$, $1\leq k\leq 2^{l_0}$ are a measurable partition of the unit interval, we have
    \begin{align*}
      \|\Upsilon_{I_0, J_0}z\|_Z
      &= \Bigl\|
        s\mapsto \Bigl\|
        t\mapsto \sum_{k=1}^{2^{l_0}} \chi_{A_k}(t) g_2(s,t)
        \Bigr\|_Y
        \Bigr\|_X\\
      &\leq \sum_{k=1}^{2^{l_0}} \Bigl\|
        s\mapsto \Bigl\|
        t\mapsto \chi_{A_k}(t) g_2(s,t)
        \Bigr\|_Y
        \Bigr\|_X.
    \end{align*}
    Next, we put
    \begin{equation*}
      g_3(s,t)
      = \cond \Bigl|
      \sum_{\substack{I\subset I_0\\J\subset J_0}}
      \sigma_I^{(1)}\sigma_J^{(2)}a_{I,J} h_{\rho_0^{-1}(I)}(s) k_{J}(t)
      \Bigr|
      \qquad s,t\in [0,1),
    \end{equation*}
    and observe that for fixed $s\in [0,1)$ and $1\leq k\leq 2^{l_0}$, the functions $t\mapsto \chi_{A_k}(t) g_2(s,t)$ and $t\mapsto g_3(s,t)$ are
    equimeasurable.  Hence, by \Cref{RI properties}~\ref{RI properties 5} we obtain
    \begin{equation*}
      \|\Upsilon_{I_0, J_0}z\|_Z
      \leq 2^{l_0} \Bigl\|
      s\mapsto \Bigl\|
      t\mapsto g_3(s,t)
      \Bigr\|_Y
      \Bigr\|_X.
    \end{equation*}
    Again, since the $A_k$, $1\leq k\leq 2^{l_0}$ are a measurable partition of the unit interval
    \begin{align*}
      \|\Upsilon_{I_0, J_0}z\|_Z
      &\leq 2^{l_0} \Bigl\|
        s\mapsto \Bigl\|
        t\mapsto \sum_{k=1}^{2^{l_0}} \chi_{A_k}(s) g_3(s,t)
        \Bigr\|_Y
        \Bigr\|_X\\
      &\leq 2^{l_0} \sum_{k=1}^{2^{l_0}} \Bigl\|
        s\mapsto \Bigl\|
        t\mapsto \chi_{A_k}(s) g_3(s,t)
        \Bigr\|_Y
        \Bigr\|_X.
    \end{align*}
    Now, we define the functions
    \begin{align*}
      g_4^{(k)}(s)
      &= \Bigl\|
        t\mapsto \chi_{A_k}(s) g_3(s,t)
        \Bigr\|_Y,
      &&  s\in [0,1),\ 1\leq k\leq 2^{l_0},\\
      g_5(s)
      &= \Bigl\|
        t\mapsto \cond \Bigl|
        \sum_{\substack{I\subset I_0\\J\subset J_0}} \sigma_{I}^{(1)}\sigma_{J}^{(2)}a_{I,J} h_{I}(s) k_{J}(t)
      \Bigr| \mathrm{d} v \mathrm{d} u_2,
      && s\in [0,1),
    \end{align*}
    and observe that $g_4^{(k)}$ and $g_5$ are equimeasurable for all $1\leq k\leq 2^{l_0}$.  Again, using \Cref{RI properties}~\ref{RI properties 5}
    yields
    \begin{equation*}
      \|\Upsilon_{I_0, J_0}z\|_Z
      \leq 4^{l_0}  \|g_5\|_X.
    \end{equation*}
    Finally, we note that $\|g_5\|_X = \|S_{I_0, J_0} z\|_Z$ and appeal to \Cref{pro:restriction-operators} to conclude
    $\|\Upsilon_{I_0, J_0}z\|_Z\leq 4^{l_0} \|S_{I_0, J_0} z\|_Z\leq 4^{l_0+1} \|z\|_Z$.
  \end{proofcase}

  The additional claims that $\Upsilon_{I_0, J_0} \Delta_{I_0, J_0} = I_Z$ and $S_{I_0, J_0}\Delta_{I_0, J_0} = \Delta_{I_0, J_0}$ are obvious.
\end{proof}

\subsection{Block bases and projections on Haar system Hardy spaces}
Here, we prove that Haar system Hardy spaces satisfy \Cref{two-parametric assumption}~\ref{two-parametric assumption a} for $C = 1$.

\begin{thm}\label{thm:basic-operators}
  Let $Z = Z(\mathbf{\sigma},X,Y)\in\mathcal{HH}(\delta^2)$ and let $(\widetilde h_I : I\in\mathcal{D})$ and $(\widehat k_J : J\in\mathcal{D})$ be
  faithful Haar systems in $X$ and $Y$, respectively.  Let $B,A\colon Z\to Z$ be given by
  \begin{equation}\label{eq:thm:basic-operators}
    Bz
    = \sum_{I,J\in\mathcal{D}}
    \frac{\langle h_I\otimes k_J, z \rangle}
    {|I| |J|}
    \widetilde h_I\otimes \widehat k_J
    \qquad\text{and}\qquad
    Az
    = \sum_{I,J\in\mathcal{D}}
    \frac{\langle \widetilde h_I\otimes \widehat k_J, z \rangle}{|I| |J|}
    h_I\otimes k_J.
  \end{equation}
  Then the operators $A,B$ satisfy
  \begin{equation*}
    I_Z = AB
    \qquad\text{and}\qquad
    \|B\|, \|A\|
    \leq 1.
  \end{equation*}
\end{thm}

In order to prove \Cref{thm:basic-operators}, we need the following lemma (we refer to~\cite[Lemma~4.3]{lechner:speckhofer:2023} for a proof).
\begin{lem}\label{lem:conditional-expectation}
  Let $X\in\mathcal{H}(\delta)$ and let $\mathcal{F}$ denote a $\sigma$-algebra generated by the finite unions of dyadic intervals $A_i$,
  $1\leq i\leq m$.  By $\cond^{\mathcal{F}}$ we denote the conditional expectation with respect to $\mathcal{F}$.  Then
  \begin{equation*}
    \|\cond^{\mathcal{F}} x\|_X
    \leq \|x\|_X,
    \qquad x\in \langle \{h_I : I\in\mathcal{D}\} \rangle.
  \end{equation*}
\end{lem}

\begin{myproof}[Proof of \Cref{thm:basic-operators}]
  Let $z = \sum_{K,L\in\mathcal{D}} a_{K,L} h_K\otimes k_L\in \langle \{h_K\otimes k_L : K,L\in\mathcal{D}\} \rangle$ and suppose that
  \begin{equation}\label{eq:162}
    \widetilde h_I
    = \sum_{K\in\mathcal{A}_I} \theta_K h_K
    \qquad\text{and}\qquad
    \widehat k_J
    = \sum_{L\in\mathcal{B}_J} \varepsilon_L k_L
  \end{equation}
  for suitable collections of dyadic intervals $\mathcal{A}_I$, $\mathcal{B}_J$, $I,J\in\mathcal{D}$.  For any $I,J\in\mathcal{D}$, we put
  $A_I = \bigcup \mathcal{A}_I$ and $B_J = \bigcup \mathcal{B}_J$.  Pick $N\in\mathbb{N}$ such that $a_{K,L} = 0$, whenever there exists
  $I_0\in\mathcal{D}_{N-1}$ and $K_0\in\mathcal{A}_{I_0}$ such that $K\subset K_0$ or there exists $J_0\in\mathcal{D}_{N-1}$ and
  $L_0\in\mathcal{B}_{J_0}$ such that $L\subset L_0$.  We define the collections $\mathcal{A} = \bigcup_{I\in\mathcal{D}}\mathcal{A}_I$ and
  $\mathcal{B} = \bigcup_{J\in\mathcal{D}}\mathcal{B}_J$.  By the definition of the norm in $Z$
  \begin{equation}\label{eq:158}
    \|z\|_Z
    = \Bigl\|
    s\mapsto \Bigl\|
    t\mapsto \cond \Bigl|
    \sum_{K,L\in\mathcal{D}} \sigma_K^{(1)} \sigma_L^{(2)} a_{K,L} h_K(s) k_L(t)
    \Bigr|
    \Bigr\|_Y
    \Bigr\|_X.
  \end{equation}

  \begin{proofcase}[Estimates for $B$] By the definitions of the operator $B$ and the norm in $Z$ and by~\eqref{eq:162}, we obtain that is equal to
    \begin{equation*}
      \|B z\|_Z
      = \Bigl\|
      s\mapsto \Bigl\|
      t\mapsto \cond \Bigl|
      \sum_{I,J\in\mathcal{D}} a_{I,J}
      \sum_{\substack{K\in\mathcal{A}_I\\L\in\mathcal{B}_J}}
      \sigma_K^{(1)} \sigma_L^{(2)} \theta_K \varepsilon_L h_K(s) k_L(t)
      \Bigr|
      \Bigr\|_Y
      \Bigr\|_X.
    \end{equation*}
    For fixed $s,t$, note that for each $J\in\mathcal{D}$ there exists at most one $L\in\mathcal{B}_J$ such that $L\ni t$.  Hence for such $L$,
    replacing the Rademacher $\sigma_L^{(2)}$ by $\sigma_J^{(2)}$ yields
    \begin{equation*}
      \|B z\|_Z 
      = \Bigl\|
      s\mapsto \Bigl\|
      t\mapsto \cond \Bigl|
      \sum_{I,J\in\mathcal{D}} \sigma_J^{(2)}
      a_{I,J}
      \sum_{\substack{K\in\mathcal{A}_I\\L\in\mathcal{B}_J}}
      \sigma_K^{(1)}\theta_K \varepsilon_L h_K(s) k_L(t)
      \Bigr|
      \Bigr\|_Y
      \Bigr\|_X.
    \end{equation*}
    Exploiting similarly that for fixed $s$, and $I\in\mathcal{D}$, there exists at most one $K\in\mathcal{A}_I$ with $I\ni s$, we can replace
    $\sigma_K^{(1)}$ with $\sigma_I^{(1)}$ and obtain
    \begin{equation}\label{eq:163}
      \begin{aligned}
        \|B z\|_Z 
        = \Bigl\|
        &s\mapsto \Bigl\|
          t\mapsto \cond \Bigl|
          \sum_{I,J\in\mathcal{D}} \sigma_I^{(1)} \sigma_J^{(2)}
          a_{I,J}
          \sum_{\substack{K\in\mathcal{A}_I\\L\in\mathcal{B}_J}}
        \theta_K \varepsilon_L h_K(s) k_L(t)
        \Bigr|
        \Bigr\|_Y
        \Bigr\|_X\\
        &= \Bigl\|
          s\mapsto \Bigl\|
          t\mapsto \cond \Bigl|
          \sum_{I,J\in\mathcal{D}} \sigma_I^{(1)} \sigma_J^{(2)} a_{I,J}
          \widetilde h_I(s) \widehat k_J(t)
          \Bigr|
          \Bigr\|_Y
          \Bigr\|_X.
      \end{aligned}
    \end{equation}
    Recall~\eqref{eq:162} for the latter equality.  Note that by equimeasurability in $Y$, we first replace $\widehat k_J$ by $k_J$ in~\eqref{eq:163}
    to obtain
    \begin{equation*}
      \|B z\|_Z
      = \Bigl\|
      s\mapsto \Bigl\|
      t\mapsto \cond \Bigl|
      \sum_{I,J\in\mathcal{D}} \sigma_I^{(1)} \sigma_J^{(2)} a_{I,J}
      \widetilde h_I(s) k_J(t)
      \Bigr|
      \Bigr\|_Y
      \Bigr\|_X,
    \end{equation*}
    and then use equimeasurability in $X$ to replace $\widetilde h_I$ by $h_I$:
    \begin{equation*}
      \|B z\|_Z
      = \Bigl\|
      s\mapsto \Bigl\|
      t\mapsto \cond \Bigl|
      \sum_{I,J\in\mathcal{D}} \sigma_I^{(1)} \sigma_J^{(2)} a_{I,J}
      h_I(s) k_J(t)
      \Bigr|
      \Bigr\|_Y
      \Bigr\|_X
      = \|z\|_Z.
    \end{equation*}
  \end{proofcase}
  
  \begin{proofcase}[Estimates for $A$]
    Observe that by the faithfulness of the Haar systems $(\widetilde h_I : I\in\mathcal{D})$ and $(\widehat k_J : J\in\mathcal{D})$, the collections
    of sets $\{A_{I_0} : I_0\in\mathcal{D}_N\}$ and $\{B_{J_0} : J_0\in\mathcal{D}_N\}$ are both partitions of $[0,1)$.  Let $\mathcal{F}$ denote the
    $\sigma$-algebra generated by $\{A_{I_0} : I_0\in\mathcal{D}_N\}$ and let $\mathcal{G}$ denote the $\sigma$-algebra generated by
    $\{B_{J_0} : J_0\in\mathcal{D}_N\}$.

    First, let us fix $s,t$ and split the random function in~\eqref{eq:158} into two parts:
    \begin{equation}\label{eq:6}
      \|z\|_Z
      = \Bigl\|
      s\mapsto \Bigl\|
      t\mapsto \cond \Bigl|
      f(s,t) + g(s,t)
      \Bigr|
      \Bigr\|_Y
      \Bigr\|_X,
    \end{equation}
    where we put
    \begin{equation}\label{eq:12}
      \begin{aligned}
        f_1(s,t)
        &= \sum_{I,J\in \mathcal{D}^{N-1}} \sum_{\substack{K\in \mathcal{A}_{I}\\L\in \mathcal{B}_J}} \sigma_K^{(1)} \sigma_L^{(2)} a_{K,L} h_K(s) k_L(t),\\
        g(s,t)
        &= \sum_{(K,L)\notin \mathcal{A}\times \mathcal{B}} \sigma_K^{(1)} \sigma_L^{(2)} a_{K,L} h_K(s) k_L(t)
      \end{aligned}
    \end{equation}
    Secondly, define
    \begin{equation*}
      f_1'(s,t)
      = \sum_{I,J\in \mathcal{D}^{N-1}} \sum_{\substack{K\in \mathcal{A}_{I}\\L\in \mathcal{B}_J}} \sigma_K'^{(1)} \sigma_L'^{(2)} a_{K,L} h_K(s) k_L(t),
    \end{equation*}
    where $(\sigma_K'^{(1)})$ is an independent copy of $(\sigma_K^{(1)})$ and $(\sigma_L'^{(2)})$ is an independent copy of $(\sigma_L^{(2)})$.  We
    denote the expectation with respect to these copies by $\cond'$.  Consequently, $f_1(s,t) + g(s,t)$ is equimeasurable to $f_1'(s,t) + g(s,t)$
    (with respect to the product measure of the probability space with its independent copy), and we obtain
    \begin{equation*}
      \|z\|_Z
      = \Bigl\|
      s\mapsto \Bigl\|
      t\mapsto \cond'\cond \Bigl|
      f_1'(s,t) + g(s,t)
      \Bigr|
      \Bigr\|_Y
      \Bigr\|_X.
    \end{equation*}
    For fixed $s,t$, we observe that for each $I,J\in \mathcal{D}$, there exists at most one $K\in \mathcal{A}_I$ with $s\in K$ and at most one
    $L\in \mathcal{B}_J$ with $t\in L$.  Hence, the random function
    \begin{equation}\label{eq:22}
      f_2'(s,t)
      = \sum_{I,J\in \mathcal{D}^{N-1}} \sigma_I^{(1)} \sigma_J^{(2)} \sum_{\substack{K\in \mathcal{A}_{I}\\L\in \mathcal{B}_J}} a_{K,L} h_K(s) k_L(t)
    \end{equation}
    is equimeasurable to $f_1'(s,t)$, and thus
    \begin{equation}\label{eq:21}
      \|z\|_Z
      = \Bigl\|
      s\mapsto \Bigl\|
      t\mapsto \cond'\cond \Bigl|
      f_2'(s,t) + g(s,t)
      \Bigr|
      \Bigr\|_Y
      \Bigr\|_X.
    \end{equation}
    Now, applying \Cref{lem:conditional-expectation} in $X$ and $Y$ to~\eqref{eq:21} together with \Cref{RI properties}~\ref{RI properties 5} yields
    \begin{equation*}
      \|z\|_Z
      \geq \Bigl\|
      s\mapsto \cond^{\mathcal{F}}_s \Bigl\|
      t\mapsto \cond^{\mathcal{G}}_t \cond'\cond \Bigl|
      f_2'(s,t) + g(s,t)
      \Bigr|
      \Bigr\|_Y
      \Bigr\|_X.
    \end{equation*}
    Hence, by Fubini's theorem and Jensen's inequality, we obtain
    \begin{equation}\label{eq:15}
      \|z\|_Z
      \geq \Bigl\|
      s\mapsto \Bigl\|
      t\mapsto \cond'\cond \Bigl|
      \bigl(\cond_s^{\mathcal{F}}\cond_t^{\mathcal{G}} f_2'\bigr)(s,t)
      + \bigl(\cond_s^{\mathcal{F}}\cond_t^{\mathcal{G}} g\bigr)(s,t)
      \Bigr|
      \Bigr\|_Y
      \Bigr\|_X.
    \end{equation}

    We claim the following identities are true:
    \begin{enumerate}[label=(\alph*)]
    \item\label{enu:thm:basic-operators:i} $\cond^{\mathcal{F}} h_K = \theta_K \frac{|K|}{|I|}\widetilde{h}_I$, whenever $K\in\mathcal{A}_I$ for some
      $I\in \mathcal{D}^N$;
    \item\label{enu:thm:basic-operators:ii} $\cond^{\mathcal{G}} k_L = \varepsilon_L \frac{|L|}{|J|}\widehat{k}_J$, whenever $L\in \mathcal{B}_J$ for
      some $J\in \mathcal{D}^{N}$;
    \item\label{enu:thm:basic-operators:iii} $\cond^{\mathcal{F}} h_K = 0$ if $K\notin \mathcal{A}$ as well as $\cond^{\mathcal{G}} k_L = 0$ if
      $L\notin \mathcal{B}$.
    \end{enumerate}
    To see this, we note that by the faithfulness of our Haar system $(\widetilde{h}_I : I\in \mathcal{D})$
    \begin{equation}\label{eq:17}
      |A_{I_0}|^{-1} \chi_{A_{I_0}}
      = \sum_{I\supsetneq I_0} \widetilde{h}_I(A_{I_0}) |A_I|^{-1} \widetilde{h}_I,
      \qquad I_0\in \mathcal{D}_N.
    \end{equation}
    Suppose now that $K\in \mathcal{A}_I$ for some $I\in \mathcal{D}^{N-1}$.  Observe that whenever $K\cap A_{I_0}\neq \emptyset$ for some
    $I_0\in \mathcal{D}_N$, then necessarily $I_0\subset I$.  This observation together with the identity~\eqref{eq:17} and the fact that
    $|A_I| = |I|$ and~\eqref{eq:162} yield
    \begin{align*}
      \cond^{\mathcal{F}} h_K
      &= \sum_{\substack{I_0\in \mathcal{D}_{N}\\I_0\subset I}} \langle \chi_{A_{I_0}}, h_K \rangle |A_{I_0}|^{-1} \chi_{A_{I_0}}
      = \sum_{\substack{I_0\in \mathcal{D}_{N}\\I_0\subset I}} \widetilde{h}_I(A_{I_0}) \theta_K \frac{|K|}{|I|} \chi_{A_{I_0}}\\
      &= \theta_K \frac{|K|}{|I|} \sum_{\substack{I_0\in \mathcal{D}_{N}\\I_0\subset I}} \widetilde{h}_I \chi_{A_{I_0}}
      = \theta_K \frac{|K|}{|I|} \widetilde{h}_I,
    \end{align*}
    as claimed in~\ref{enu:thm:basic-operators:i}.  For the final claim~\ref{enu:thm:basic-operators:iii}, suppose that $K\notin \mathcal{A}$ and
    observe that~\eqref{eq:17} and~\eqref{eq:162} yield immediately $\langle \chi_{A_{I_0}}, h_K \rangle = 0$ for all $I_0\in \mathcal{D}_N$, and
    hence $\cond^{\mathcal{F}} h_K = 0$.  The proofs for~\ref{enu:thm:basic-operators:ii} and the second part of~\ref{enu:thm:basic-operators:iii} and
    the are completely analogous and therefore omitted.

    Combining~\ref{enu:thm:basic-operators:i}--\ref{enu:thm:basic-operators:iii} with~\eqref{eq:12} and~\eqref{eq:22} yields
    \begin{equation}\label{eq:23}
      \begin{aligned}
        \bigl(\cond_s^{\mathcal{F}}\cond_t^{\mathcal{G}} f_2'\bigr)(s,t)
        &= \sum_{I,J\in \mathcal{D}^{N-1}} \sigma_I'^{(1)} \sigma_J'^{(2)} \sum_{\substack{K\in \mathcal{A}_{I}\\L\in \mathcal{B}_J}}
        a_{K,L} \bigl(\cond_s^{\mathcal{F}}h_K\bigr)(s) \bigl(\cond_t^{\mathcal{G}}k_L\bigr)(t)\\
        &= \sum_{I,J\in \mathcal{D}^{N-1}} \sigma_I'^{(1)} \sigma_J'^{(2)} \sum_{\substack{K\in \mathcal{A}_{I}\\L\in \mathcal{B}_J}}
        a_{K,L} \frac{|K| |L|}{|I| |J|}
        \theta_K \varepsilon_L \widetilde{h}_I(s) \widehat{k}_J(t),
      \end{aligned}
    \end{equation}
    as well as
    \begin{equation}\label{eq:24}
      \bigl(\cond_s^{\mathcal{F}}\cond_t^{\mathcal{G}} g\bigr)(s,t)
      = 0.
    \end{equation}
    We insert~\eqref{eq:23} and~\eqref{eq:24} into~\eqref{eq:15} and obtain
    \begin{align*}
      \|z\|_Z
      &\geq \Bigl\|
        s\mapsto \Bigl\|
        t\mapsto \cond' \Bigl|
        \sum_{I_0,J_0\in \mathcal{D}_{N-1}} \sum_{\substack{I\supset I_0\\J\supset J_0}} \sigma_I'^{(1)} \sigma_J'^{(2)} \sum_{\substack{K\in \mathcal{A}_I\\L\in \mathcal{B}_J}}
      a_{K,L} \frac{|K| |L|}{|I| |J|}
      \theta_K \varepsilon_L \widetilde{h}_I(s) \widehat{k}_J(t)
      \Bigr|
      \Bigr\|_Y
      \Bigr\|_X\\
      &\geq \Bigl\|
        s\mapsto \Bigl\|
        t\mapsto \cond' \Bigl|
        \sum_{I_0,J_0\in \mathcal{D}_{N-1}} \sum_{\substack{I\supset I_0\\J\supset J_0}} \sigma_I'^{(1)} \sigma_J'^{(2)} 
      \frac{\langle \widetilde{h}_I\otimes \widehat{h}_{J}, z\rangle}{|I| |J|}
      \widetilde{h}_I(s) \widehat{k}_J(t)
      \Bigr|
      \Bigr\|_Y
      \Bigr\|_X.
    \end{align*}
    Repeating the argument below~\eqref{eq:163}, we can exchange $\widetilde{h}_I$ by $h_I$ and $\widehat{k}_J$ by $k_J$, and thus
    \begin{equation*}
      \|z\|_Z
      \geq \Bigl\|
      s\mapsto \Bigl\|
      t\mapsto \cond' \Bigl|
      \sum_{I_0,J_0\in \mathcal{D}_{N-1}} \sum_{\substack{I\supset I_0\\J\supset J_0}} \sigma_I'^{(1)} \sigma_J'^{(2)}
      \frac{\langle \widetilde{h}_I\otimes \widetilde{k}_J, z \rangle}{|I| |J|}
      h_I(s) k_J(t)
      \Bigr|
      \Bigr\|_Y
      \Bigr\|_X
      = \|Qz\|_Z.
    \end{equation*}
  \end{proofcase}
  By the faithfulness of our Haar systems, it is clear that $I_Z = AB$.
\end{myproof}

\section{Haar multipliers, pointwise multipliers and Capon's projection}
\label{sec:mult-capons-proj}
In this section we prove the required proximity results between Haar multipliers and pointwise multipliers that were stated in \Cref{proximity Haar to
  pointwise}, which yield that Haar system Hardy spaces are 1-Capon spaces (see \Cref{two-parametric assumption}).
\begin{thm}\label{thm:pw-multipliers:1}
  Let $Z = Z(\boldsymbol{\sigma},X,Y)$ be a Haar system Hardy space and let $\mathcal{V}_Z$ denote the space $\mathcal{V}(\delta^2)$ equipped with the
  norm of $Z$, and put $\mathcal{V}_{Z^{\Omega}} = \mathcal{O}(\mathcal{V}_Z)$.  Suppose that the Haar multiplier
  $D\colon \mathcal{V}_Z\to \mathcal{V}_Z$ given by $D h_I\otimes k_J = d_{I,J} h_I\otimes k_J$, $I,J\in \mathcal{D}$ has a finite bi-tree variation
  semi-norm, i.e., $\|D\|_{\mathrm{T^2}\mathrm{S}} < \infty$.

  Then the following estimates are true:
  \begin{subequations}\label{eq:4}
    \begin{align}
      \|(\mathcal{O}D - M_1^D\mathcal{O})\mathcal{C}\colon\mathcal{V}_Z\to Z^\Omega\|
      &\leq 4 \|D\|_{\mathrm{T^2}\mathrm{S}},\\
      \|(\mathcal{O}D - M_2^D\mathcal{O})(I_Z - \mathcal{C})\colon\mathcal{V}_Z\to Z^\Omega\|
      &\leq 4 \|D\|_{\mathrm{T^2}\mathrm{S}},\\
      \|(M_1^D- M_2^D)\mathcal{O}\mathcal{C}\colon \mathcal{V}_Z \to Z^\Omega\|
      &\leq \|D\| + \|D\|_{\infty} + 8\|D\|_{\mathrm{T^2}\mathrm{S}}.
    \end{align}
  \end{subequations}
\end{thm}

\begin{myproof}
  To this end, let $z = \sum_{I,J\in\mathcal{D}} a_{I,J} h_I\otimes k_J\in\mathcal{V}_Z$ with $a_{I,J}\in \mathbb{R}$ and let $s,t\in [0,1)$ be fixed.
  \begin{proofcase}[Estimate for $\|(\mathcal{O}D - M_1^D\mathcal{O})\mathcal{C}\colon\mathcal{V}_Z\to Z^\Omega\|$]
    Observe that for fixed $s,t\in [0,1)$ and $\sigma\in\boldsymbol{\sigma}$, we have
    \begin{align*}
      \bigl|\bigl((\mathcal{O}D& - M_1^D\mathcal{O})\mathcal{C}z\bigr)(s,t,\sigma)\bigr|
                                 = \Bigl|
                                 \sum_{j=0}^{\infty} \sum_{i=j}^{\infty} \sum_{\substack{I\in \mathcal{D}_{i}\\J\in \mathcal{D}_j}} \bigl(d_{I,J} - m_1^D(s,t)\bigr)
      \sigma_I^{(1)} \sigma_J^{(2)} a_{I,J} h_I(s) k_J(t)
      \Bigr|\\
                               &\leq \Bigl|
                                 \sum_{j=0}^{\infty} \sum_{i=j}^{\infty} \bigl(d_{I_j(s),J_j(t)} - m_1^D(s,t)\bigr)
                                 \sigma_{I_i(s)}^{(1)} \sigma_{J_j(t)}^{(2)} a_{I_i(s),J_j(t)} h_{I_i(s)}(s) k_{J_j(t)}(t)
                                 \Bigr|\\
                               &\qquad + \Bigl|
                                 \sum_{j=0}^{\infty} \sum_{i=j}^{\infty} \bigl(d_{I_i(s),J_j(t)} - d_{I_j(s),J_j(t)}\bigr)
                                 \sigma_{I_i(s)}^{(1)} \sigma_{J_j(t)}^{(2)} a_{I_i(s),J_j(t)} h_{I_i(s)}(s) k_{J_j(t)}(t)
                                 \Bigr|
                                 =: A_1 + B_1.
    \end{align*}
    We will now estimate $A_1$ and $B_1$, separately.  First, note that by the definitions of $I_i(s)$, $J_j(t)$, $m_1^D$ as well as \Cref{rem:cond}
    and \Cref{dfn:multiplier-space}, we obtain
    \begin{align*}
      A_1
      &= \Bigl|
        \sum_{j=0}^{\infty} \sum_{i=j}^{\infty} \sum_{k=j}^{\infty}
        \bigl(d_{I_k(s),J_k(t)} - d_{I_{k+1}(s),J_{k+1}(t)}\bigr)
        \sigma_{I_i(s)}^{(1)} \sigma_{J_j(t)}^{(2)} a_{I_i(s),J_j(t)} h_{I_i(s)}(s) k_{J_j(t)}(t)
        \Bigr|\\
      &\leq \sum_{k=0}^{\infty} \sum_{j=0}^k \bigl|d_{I_k(s),J_k(t)} - d_{I_{k+1}(s),J_{k+1}(t)}\bigr|
        \cdot \Bigl| \sum_{i=j}^{\infty} \sigma_{I_i(s)}^{(1)} \sigma_{J_j(t)}^{(2)} a_{I_i(s),J_j(t)} h_{I_i(s)}(s) k_{J_j(t)}(t) \Bigr|\\
      &= \sum_{k=0}^{\infty} \sum_{j=0}^k \bigl|d_{I_k(s),J_k(t)} - d_{I_{k+1}(s),J_{k+1}(t)}\bigr|
        \cdot \bigl| \bigl(\mathcal{O} (I_{Z^{\Omega}} - \cond_j)\otimes (\cond_{j+1} - \cond_j) z\bigr) (s,t)\bigr|.
    \end{align*}
    Similarly, using \Cref{rem:cond} and \Cref{dfn:multiplier-space} yields
    \begin{align*}
      B_1
      &= \Bigl|
        \sum_{j=0}^{\infty} \sum_{i=j}^{\infty} \sum_{k=j}^{i-1}
        \bigl(d_{I_{k+1}(s),J_j(t)} - d_{I_k(s),J_j(t)}\bigr)
        \sigma_{I_i(s)}^{(1)} \sigma_{J_j(t)}^{(2)} a_{I_i(s),J_j(t)} h_{I_i(s)}(s) k_{J_j(t)}(t)
        \Bigr|\\
      &\leq \sum_{k=0}^{\infty} \sum_{j=0}^k \bigl|d_{I_{k+1}(s),J_j(t)} - d_{I_k(s),J_j(t)}\bigr|
        \cdot \Bigl|
        \sum_{i=j\vee k+1}^{\infty} \sigma_{I_i(s)}^{(1)} \sigma_{J_j(t)}^{(2)} a_{I_i(s),J_j(t)} h_{I_i(s)}(s) k_{J_j(t)}(t)
        \Bigr|\\
      &= \sum_{k=0}^{\infty} \sum_{j=0}^k \bigl|d_{I_{k+1}(s),J_j(t)} - d_{I_k(s),J_j(t)}\bigr|
        \cdot \bigl| \bigl(\mathcal{O} (I_{Z^{\Omega}} - \cond_{j\vee k+1})\otimes (\cond_{j+1} - \cond_j) z\bigr) (s,t)\bigr|.
    \end{align*}
    By the monotonicity of the norm in $Z^\Omega$ and \Cref{pro:bi-monotone} we obtain
    \begin{align*}
      &\|(\mathcal{O}D - M_1^D\mathcal{O})\mathcal{C} z\|_{Z^{\Omega}}
        \leq \|A_1\|_{Z^{\Omega}} + \|B_1\|_{Z^{\Omega}}\\
      &\quad \leq \sum_{k=0}^{\infty} (k+1) \max_{\substack{I,J\in \mathcal{D}_k\\\omega,\xi\in\{\pm 1\}}} |d_{I,J} - d_{I^\omega,J^\xi}| \cdot 4 \|z\|_Z
      + \sum_{k=0}^{\infty} \sum_{j=0}^k \max_{\substack{I\in \mathcal{D}_k\\J\in \mathcal{D}_j\\\omega\in\{\pm 1\}}} |d_{I^{\omega},J} - d_{I,J}|
      \cdot 4 \|z\|_Z\\
      &\quad \leq 4 \|D\|_{\mathrm{T^2}\mathrm{S}}\cdot \|z\|_{Z},
    \end{align*}
    which establishes the first inequality in~\eqref{eq:4}.
  \end{proofcase}

  \begin{proofcase}[Estimate for $\|(\mathcal{O}D - M_2^D\mathcal{O})(I_Z - \mathcal{C})\colon\mathcal{V}_Z\to Z^\Omega\|$]
    For fixed $s,t\in [0,1)$ and $\sigma\in \boldsymbol{\sigma}$, we obtain
    \begin{align*}
      \bigl|\bigl((\mathcal{O}D& - M_1^D\mathcal{O})\mathcal{C}z\bigr)(s,t,\sigma)\bigr|
                                 = \Bigl|
                                 \sum_{i=0}^{\infty} \sum_{j=i+1}^{\infty} \sum_{\substack{I\in \mathcal{D}_{i}\\J\in \mathcal{D}_j}} \bigl(d_{I,J} - m_2^D(s,t)\bigr)
      \sigma_I^{(1)} \sigma_J^{(2)} a_{I,J} h_I(s) k_J(t)
      \Bigr|\\
                               &\leq \Bigl|
                                 \sum_{i=0}^{\infty} \sum_{j=i+1}^{\infty} \bigl(d_{I_i(s),J_{i+1}(t)} - m_2^D(s,t)\bigr)
                                 \sigma_{I_i(s)}^{(1)} \sigma_{J_j(t)}^{(2)} a_{I_i(s),J_j(t)} h_{I_i(s)}(s) k_{J_j(t)}(t)
                                 \Bigr|\\
                               &\qquad + \Bigl|
                                 \sum_{i=0}^{\infty} \sum_{j=i+1}^{\infty} \bigl(d_{I_i(s),J_j(t)} - d_{I_i(s),J_{i+1}(t)}\bigr)
                                 \sigma_{I_i(s)}^{(1)} \sigma_{J_j(t)}^{(2)} a_{I_i(s),J_j(t)} h_{I_i(s)}(s) k_{J_j(t)}(t)
                                 \Bigr|\\
                               &=: A_2 + B_2.
    \end{align*}
    By the definitions of $I_i(s)$, $J_j(t)$, $m_2^D$ as well as \Cref{rem:cond} and \Cref{dfn:multiplier-space}, we obtain
    \begin{align*}
      A_2
      &= \Bigl|
        \sum_{i=0}^{\infty} \sum_{j=i+1}^{\infty} \sum_{k=i}^{\infty}  \bigl(d_{I_k(s),J_{k+1}(t)} - d_{I_{k+1}(s),J_{k+2}(t)}\bigr)
        \sigma_{I_i(s)}^{(1)} \sigma_{J_j(t)}^{(2)} a_{I_i(s),J_j(t)} h_{I_i(s)}(s) k_{J_j(t)}(t)
        \Bigr|\\
      &\leq \sum_{k=0}^{\infty} \sum_{i=0}^k |d_{I_k(s),J_{k+1}(t)} - d_{I_{k+1}(s),J_{k+2}(t)}|
        \cdot \Bigl| \sum_{j=i+1}^{\infty} \sigma_{I_i(s)}^{(1)} \sigma_{J_j(t)}^{(2)} a_{I_i(s),J_j(t)} h_{I_i(s)}(s) k_{J_j(t)}(t) \Bigr|\\
      &= \sum_{k=0}^{\infty} \sum_{i=0}^k |d_{I_k(s),J_{k+1}(t)} - d_{I_{k+1}(s),J_{k+2}(t)}|
        \cdot \bigl| \bigl(\mathcal{O} (\cond_{i+1} - \cond_i)\otimes (I_{Z^{\Omega}} - \cond_{i+1}) z\bigr) (s,t)\bigr|.
    \end{align*}
    Similarly, using \Cref{rem:cond} and \Cref{dfn:multiplier-space} yields
    \begin{align*}
      B_2
      &= \Bigl|
        \sum_{i=0}^{\infty} \sum_{j=i+1}^{\infty} \sum_{k=i+1}^{j-1} \bigl(d_{I_i(s),J_{k+1}(t)} - d_{I_i(s),J_k(t)}\bigr)
        \sigma_{I_i(s)}^{(1)} \sigma_{J_j(t)}^{(2)} a_{I_i(s),J_j(t)} h_{I_i(s)}(s) k_{J_j(t)}(t)
        \Bigr|\\
      &\leq \sum_{k=1}^{\infty} \sum_{i=0}^{k-1} |d_{I_i(s),J_{k+1}(t)} - d_{I_i(s),J_k(t)}|
        \cdot \Bigl| \sum_{j=k+1}^{\infty} \sigma_{I_i(s)}^{(1)} \sigma_{J_j(t)}^{(2)} a_{I_i(s),J_j(t)} h_{I_i(s)}(s) k_{J_j(t)}(t) \Bigr|\\
      &= \sum_{k=1}^{\infty} \sum_{i=0}^{k-1} |d_{I_i(s),J_{k+1}(t)} - d_{I_i(s),J_k(t)}|
        \cdot \bigl| \bigl(\mathcal{O} (\cond_{i+1} - \cond_i)\otimes (I_{Z^{\Omega}} - \cond_{k+1}) z\bigr) (s,t)\bigr|.
    \end{align*}
    By the monotonicity of the norm in $Z^\Omega$, and \Cref{pro:bi-monotone} we obtain
    \begin{align*}
      &\|(\mathcal{O}D - M_2^D\mathcal{O})(I_Z - \mathcal{C}) z\|_{Z^\Omega}
        \leq \|A_2\|_{Z^{\Omega}} + \|B_2\|_{Z^{\Omega}}\\
      &\quad \leq \sum_{k=0}^{\infty} (k+1) \max_{\substack{I\in \mathcal{D}_k\\J\in \mathcal{D}_{k+1}\\\omega,\xi\in\{\pm 1\}}} |d_{I,J} - d_{I^\omega,J^\xi}| \cdot 4 \|z\|_Z
      + \sum_{k=1}^{\infty} \sum_{i=0}^{k-1} \max_{\substack{I\in \mathcal{D}_i\\J\in \mathcal{D}_k\\\xi\in\{\pm 1\}}} |d_{I,J^{\xi}} - d_{I,J}| \cdot 4 \|z\|_Z\\
      &\quad \leq 4 \|D\|_{\mathrm{T^2}\mathrm{S}}\cdot \|z\|_{Z},
    \end{align*}
    which establishes the second estimate in~\eqref{eq:4}.
  \end{proofcase}
  
  \begin{proofcase}[Estimate for $\|(M_2^D - M_1^D)\mathcal{O}\mathcal{C}\colon \mathcal{V}_Z\to Z^\Omega\|$]
    Observe that by~\eqref{eq:4}, and by the definition of $M_2^D$, we have $\|M_2^D\|\leq \|D\|$, and thus
    \begin{align*}
      \|D z\|_Z
      &= \|\mathcal{O} D z\|_{Z^{\Omega}}
        = \|\mathcal{O} D \mathcal{C} z + \mathcal{O} D (I_Z-\mathcal{C})\|_ z{Z^{\Omega}}\\
      &\geq \|M_1^D \mathcal{O} \mathcal{C} z + M_2^D\mathcal{O} (I_Z-\mathcal{C}) z\|_{Z^{\Omega}} - 8\|D\|_{\mathrm{T^2}\mathrm{S}}\cdot \|z\|_Z\\
      &= \|(M_1^D - M_2^D) \mathcal{O} \mathcal{C} z + M_2^D\mathcal{O} z\|_{Z^{\Omega}} - 8\|D\|_{\mathrm{T^2}\mathrm{S}}\cdot \|z\|_Z\\
      &\geq \|(M_1^D - M_2^D) \mathcal{O} \mathcal{C} z\|_{Z^{\Omega}} - \|M_2^D\mathcal{O} z\|_{Z^{\Omega}} - 8\|D\|_{\mathrm{T^2}\mathrm{S}}\cdot \|z\|_Z\\
      &\geq \|(M_1^D - M_2^D) \mathcal{O} \mathcal{C} z\|_{Z^{\Omega}} - \|D\|_{\infty}\cdot \|z\|_Z - 8\|D\|_{\mathrm{T^2}\mathrm{S}}\cdot \|z\|_Z,
    \end{align*}
    as claimed.\qedhere
  \end{proofcase}
\end{myproof}

\begin{cor}\label{cor:pw-multipliers:2} Assuming the conditions of \Cref{thm:pw-multipliers:1} and assuming that
  \begin{equation*}
    \|D\|_{{\mathrm{T^2}\mathrm{S}}}
    < |\lambda_{\mathcal U} (D)-\mu_{\mathcal U}(D)|/8
  \end{equation*}
  it follows that
  \begin{equation*}
    \|\mathcal{C}\|_{\mathcal{L}(Z)}
    \le \frac{\|D\|+\|D\|_\infty + 8 \|D\|_{{\mathrm{T^2}\mathrm{S}}}}{|\lambda_{\mathcal U} (D) - \mu_{\mathcal U}(D)| - 8\|D\|_{{\mathrm{T^2}\mathrm{S}}}}.
  \end{equation*}
  and thus, $\mathcal{C}\colon Z\to Z$ is bounded if $\|D\colon Z\to Z\| < \infty$.

  In particular, $\mathcal{C}$ is bounded on $Z$ if and only if there exists a bounded Haar multiplier $D\colon Z\to Z$ such that
  $\lambda_\mathcal{U}(D)\neq\mu_\mathcal{U}(D)$.
\end{cor}

\begin{proof}
  Put $\lambda = \lambda_\mathcal{U}(D)$ and $\mu = \mu_\mathcal{U}(D)$ note that
  \begin{align*}
    \|(\lambda - \mu) \mathcal{C} z\|_Z
    &= \|(\lambda I_{Z^{\Omega}} - M_1^D) \mathcal{O} \mathcal{C} z - (\mu I_{Z^{\Omega}} - M_2^D) \mathcal{O} \mathcal{C} z + (M_1^D - M_2^D)\mathcal{O} \mathcal{C} z\|_{Z^{\Omega}}\\
    &\leq \|(\lambda I_{Z^{\Omega}} - M_1^D) \mathcal{O} \mathcal{C} z\|_{Z^{\Omega}} +  \|(\mu I_{Z^{\Omega}} - M_2^D) \mathcal{O} \mathcal{C} z\|_{Z^{\Omega}} +  \|(M_1^D - M_2^D)\mathcal{O} \mathcal{C} z\|_{Z^{\Omega}}\\
    &\leq 8\|D\|_{\mathrm{T^2}\mathrm{S}}\cdot \|\mathcal{C} z\|_Z + (\|D\| + \|D\|_{\infty} + 8\|D\|_{\mathrm{T^2}\mathrm{S}})\cdot \|z\|_{Z^{\Omega}}.
  \end{align*}
  Hence, if $8\|D\|_{\mathrm{T^2}\mathrm{S}} < |\lambda - \mu|$, then
  \begin{equation*}
    \|\mathcal{C} z\|_Z
    \leq \frac{\|D\| + \|D\|_{\infty} + 8\|D\|_{\mathrm{T^2}\mathrm{S}}}{|\lambda - \mu| - 8\|D\|_{\mathrm{T^2}\mathrm{S}}}\cdot \|z\|_{Z^{\Omega}}.
  \end{equation*}
    
  For the final part, assume that there exists a bounded Haar multiplier $D\colon Z\to Z$ such that $\lambda_\mathcal{U}(D)\neq\mu_\mathcal{U}(D)$. By \Cref{T:3.4} and
  \Cref{thm:basic-operators}, we may pass to a bounded Haar multiplier $\tilde D\colon Z\to Z$ such that
  $\|\tilde D\|_{\mathrm{T^2}\mathrm{S}} < |\lambda_\mathcal{U}(D) - \mu_\mathcal{U}(D)|/8$ and $\lambda_\mathcal{U}(\tilde D) = \lambda_\mathcal{U}(D)$, $\mu_\mathcal{U}(\tilde D) = \mu_\mathcal{U}(D)$, which yields the conclusion.
\end{proof}

\begin{thm}\label{thm:pw-multipliers:2}
  Let $Z = Z(\mathbf{\sigma},X,Y)\in\mathcal{HH}(\delta^2)$ and $\mathcal{V}_Z = \mathcal{V}(\delta^2)$ with the norm induced by $Z$.  Suppose that
  $m\colon [0,1)^2\to \mathbb{R}$ is continuous almost everywhere with $\|m\|_{L^\infty([0,1)^2)}\leq 1$, and let the pointwise multiplication operator
  $M\colon \mathcal{V}_Z\to Z^\Omega$ be given by
  \begin{equation*}
    (Mw)(s,t,\sigma) = m(s,t)\cdot w(s,t,\sigma),
    \qquad s,t\in [0,1),\ \sigma\in \boldsymbol{\sigma}.
  \end{equation*}
  If $\mathcal{C}\colon Z\to Z$ is unbounded but $M\mathcal{O}\mathcal{C}\colon \mathcal{V}_Z \to Z^\Omega$ is bounded, then $M=0$.
\end{thm}

\begin{proof}%[Proof of Theorem \Cref{thm:pw-multipliers:2}]
  Suppose that $|\{ |m| > 0\}| > 0$, then there exists $k_0\in\mathbb{N}$ such that $|\{ |m| > 1/k_0\}| > 0$.  By the almost everywhere continuity of
  $m$, there exist dyadic intervals $I_0,J_0\in\mathcal{D}$ with $|I_0| = |J_0| = 2^{-l_0}$ such that
  $I_0\times J_0\setminus N\subset \{ |m| > 1/k_0\}$, where $|N| = 0$.  Define the pointwise multiplier $R_{I_0, J_0}\colon Z^\Omega\to Z^\Omega$ by
  $(R_{I_0, J_0} w)(s,t,\sigma) = \chi_{I_0\times J_0}(s,t) w(s,t,\sigma)$ and let $z\in \mathcal{V}_Z$.  By the monotonicity of the norm in $Z^\Omega$, we obtain
  \begin{equation}\label{eq:169}
    \begin{aligned}
      \|M\mathcal{O}\mathcal{C}z\|_{Z^\Omega}
      &\geq \|R_{I_0, J_0} M\mathcal{O}\mathcal{C}z\|_{Z^\Omega}\\
      &= \Bigl\|
        s\mapsto \Bigl\|
        t\mapsto \cond \bigl|
        \chi_{I_0\times J_0}(s,t) m(s,t) (\mathcal{O}\mathcal{C}z)(s,t,\sigma)
        \bigr|
        \Bigr\|_Y
        \Bigr\|_X\\
      &\geq \frac{1}{k_0} \Bigl\|
        s\mapsto \Bigl\|
        t\mapsto \cond \bigl|
        \chi_{I_0\times J_0}(s,t) (\mathcal{O}\mathcal{C} w)(s,t,\sigma)
        \bigr|
        \Bigr\|_Y
        \Bigr\|_X\\
      &= \frac{1}{k_0} \|R_{I_0, J_0} \mathcal{O}\mathcal{C} z\|_{Z^\Omega}.
    \end{aligned}
  \end{equation}
  In the following, we will use the identity
  \begin{equation}\label{eq:168}
    I_{Z}
    = (I_X - \cond_{l_0})\otimes (I_Y - \cond_{l_0})
    + (I_X - \cond_{l_0})\otimes \cond_{l_0}
    + \cond_{l_0}\otimes (I_Y - \cond_{l_0})
    + \cond_{l_0}\otimes \cond_{l_0}
  \end{equation}
  to split $R_{I_0, J_0} \mathcal{O}\mathcal{C}z$ into parts and estimate them separately.  Let
  $z = \sum_{I,J} a_{I,J} h_I\otimes k_J\in \mathcal{V}_Z$ with $a_{I,J}\in \mathbb{R}$.

  First, observe that by the definition of $\cond_{l_0}$ and $\mathcal{C}$
  \begin{equation}\label{eq:165}
    R_{I_0, J_0}\mathcal{O}(I_X - \cond_{l_0})\otimes \cond_{l_0}\mathcal{C}z
    = 0.
  \end{equation}

  Secondly, the definitions of $R_{I_0, J_0}$, $S_{I_0, J_0}$, $\cond_{l_0}$ and the monotonicity of the norm in $Z^\Omega$ yield
  \begin{align*}
    \|R_{I_0, J_0}\mathcal{O}(\cond_{l_0}\otimes (I_Y - \cond_{l_0})) \mathcal{C}z\|_{Z^\Omega}
    &= \Bigl\|
      \sum_{\substack{|I| > 2^{-l_0}\\|J|\leq 2^{-l_0}}}
    a_{I,J} ((\chi_{I_0}h_I)\otimes (\chi_{J_0}k_J))\otimes (\sigma_I^{(1)}\otimes\sigma_J^{(2)})
    \Bigr\|_{Z^\Omega}\\
    &\leq \sum_{|I| > 2^{-l_0}} \Bigl\|
      \sum_{J\subset J_0} a_{I,J} ((\chi_{I_0} h_I)\otimes k_J)\otimes (\sigma_I^{(1)}\otimes\sigma_J^{(2)})
      \Bigr\|_{Z^\Omega}\\
    &\leq \sum_{|I| > 2^{-l_0}} \Bigl\|
      \sum_{J\subset J_0} a_{I,J} (h_I\otimes k_J)\otimes (\sigma_I^{(1)}\otimes\sigma_J^{(2)})
      \Bigr\|_{Z^\Omega}\\
    &= \sum_{|I| > 2^{-l_0}} \Bigl\|
      \sum_{J\subset J_0} a_{I,J} h_I\otimes k_J
      \Bigr\|_{Z}\\
    &= \sum_{|I| > 2^{-l_0}} \bigl\|
      S_{I, J_0}(p_I\otimes I_Y)z
      \bigr\|_{Z}.
  \end{align*}
  By \Cref{pro:bi-monotone} and \Cref{pro:restriction-operators}, we obtain
  \begin{equation}\label{eq:159}
    \|R_{I_0, J_0}\mathcal{O}\cond_{l_0}\otimes (I_Y - \cond_{l_0}) \mathcal{C}z\|_{Z^\Omega}
    \leq \sum_{|I| > 2^{-l_0}} 8 \|z\|_{Z}
    \leq 2^{l_0+3} \|z\|_{Z}.
  \end{equation}

  Thirdly, note that by \Cref{pro:bi-monotone}
  \begin{align*}
    4\|z\|_{Z}
    &\geq \|p_I\otimes p_Jz\|_{Z}
      = \Bigl\|a_{I,J} (h_I\otimes k_J)\otimes (\sigma_I^{(1)}\otimes\sigma_J^{(2)})\Bigr\|_{Z^\Omega},
  \end{align*}
  hence, by the above inequality, we obtain
  \begin{align*}
    \|R_{I_0, J_0}\mathcal{O}(\cond_{l_0}\otimes\cond_{l_0}) \mathcal{C}z\|_{Z^\Omega}
    &= \Bigl\|
      \sum_{|I|,|J| > 2^{-l_0}} a_{I,J} (\chi_{I_0}h_I\otimes \chi_{J_0}k_J)\otimes (\sigma_I^{(1)}\otimes\sigma_J^{(2)})
      \Bigr\|_{Z^\Omega}\\
    &\leq \sum_{|I|,|J| > 2^{-l_0}} \Bigl\|
      a_{I,J} (h_I\otimes k_J)\otimes (\sigma_I^{(1)}\otimes\sigma_J^{(2)})
      \Bigr\|_{Z^\Omega}
      \leq \sum_{|I|,|J| > 2^{-l_0}} 4 \|z\|_{Z^\Omega}.
  \end{align*}
  We record the estimate
  \begin{equation}\label{eq:167}
    \|R_{I_0, J_0}\mathcal{O}\cond_{l_0}\otimes\cond_{l_0} \mathcal{C} z\|_{Z^\Omega}
    \leq 4^{l_0+1} \|z\|_{Z}.
  \end{equation}
  Combining~\eqref{eq:169}, \eqref{eq:168}, \eqref{eq:165}, \eqref{eq:159} and \eqref{eq:167} yields
  \begin{align*}
    k_0\cdot \|M\mathcal{O}\mathcal{C}z\|_{Z^\Omega}
    &\geq \|
      R_{I_0, J_0} \mathcal{O}(I_X - \cond_{l_0})\otimes I_Y - \cond_{l_0}) \mathcal{C} z
      \|_{Z^\Omega}
      - \|R_{I_0, J_0} \mathcal{O}(I_X - \cond_{l_0})\otimes \cond_{l_0} \mathcal{C} z\|_{Z^\Omega}\\
    &\qquad - \|R_{I_0, J_0} \mathcal{O}\cond_{l_0}\otimes (I_Y - \cond_{l_0}) \mathcal{C} z\|_{Z^\Omega}
      - \|R_{I_0, J_0} \mathcal{O}\cond_{l_0}\otimes \cond_{l_0} \mathcal{C} z\|_{Z^\Omega}\\
    &\geq \|
      R_{I_0, J_0} \mathcal{O}(I_X - \cond_{l_0})\otimes (I_Y - \cond_{l_0}) \mathcal{C} z
      \|_{Z^\Omega}
      - (2^{l_0+3} + 4^{l_0+1})\|z\|_{Z}.
  \end{align*}
  Note that by definition of $R_{I_0, J_0}$, $\cond_{l_0}$ and $S_{I_0, J_0}$, we obtain
  \begin{equation*}
    R_{I_0, J_0} \mathcal{O}(I_X - \cond_{l_0})\otimes (I_Y - \cond_{l_0}) \mathcal{C}z
    = \mathcal{O}S_{I_0, J_0} \mathcal{C}z,
  \end{equation*}
  and thus, since $\mathcal{O}$ is an isometry, the latter estimate yields
  \begin{equation}\label{eq:170}
    \|S_{I_0, J_0} \mathcal{C}z\|_{Z}
    = \|\mathcal{O}S_{I_0, J_0} \mathcal{C}z\|_{Z^\Omega}
    \leq (k_0\cdot \|M\mathcal{O}\mathcal{C}\| + 2^{l_0+3} + 4^{l_0+1})\cdot \|z\|_{Z}.
  \end{equation}

  On the other hand, since the Capon projection $\mathcal{C}$ commutes with both $\Delta_{I_0, J_0}$ and $S_{I_0, J_0}$, we obtain by \Cref{pro:down-up-scale},
  and~\eqref{eq:170} that
  \begin{align*}
    \|\mathcal{C} z\|_{Z}
    &= \|\Upsilon_{I_0, J_0}\Delta_{I_0, J_0}\mathcal{C} z\|_{Z}
      \leq \|\Upsilon_{I_0, J_0}\|\cdot \|\mathcal{C}\Delta_{I_0, J_0} z\|_{Z}
      = \|\Upsilon_{I_0, J_0}\|\cdot \|\mathcal{C}S_{I_0, J_0}\Delta_{I_0, J_0} z\|_{Z}\\
    &\leq \|\Upsilon_{I_0, J_0}\|\cdot \|S_{I_0, J_0}\mathcal{C}\|\cdot
      \|\Delta_{I_0, J_0} z\|_{Z}
      \leq 4^{l_0}\cdot \bigl( k_0\cdot \|M\mathcal{O}\mathcal{C}\| + (2^{l_0+3} + 4^{l_0+1})
      \bigr)
      \cdot \|z\|_{Z}.
  \end{align*}
  Contrary to our hypothesis, this implies the boundedness of the Capon projection $\mathcal{C}\colon Z\to Z$; hence, $|\{ |m| > 0\}| = 0$, i.e., $M = 0$.\qedhere
\end{proof}

\bibliographystyle{abbrv}%
\bibliography{bibliography}%

\end{document}